\documentclass{ws-m3as-arxiv}

\usepackage{amsmath}

\usepackage{enumitem,indentfirst}
\usepackage{algorithm,algpseudocode,mathtools}
\usepackage{amsfonts, amssymb, mathrsfs, pifont}
\usepackage{graphicx, subcaption, float, xcolor}
\graphicspath{{./Figure/}{./}}
\usepackage{mathbbol}
\usepackage{pbox}
\usepackage{booktabs, cellspace, hhline, empheq}
\allowdisplaybreaks[4]

\usepackage{hyperref} %
\hypersetup{
	plainpages=false,
    colorlinks,
    linktocpage=true,   %
    linkcolor={red!50!black},
    citecolor={blue!50!black},
    urlcolor={blue!80!black}
} 
\usepackage[hyperpageref]{backref}
\usepackage{comment,multicol,xspace}
\usepackage{makecell}
\usepackage[all]{xy} %
\usepackage{tikz-cd}
\usepackage{adjustbox}
\usepackage{multirow}

\numberwithin{equation}{section}
\numberwithin{figure}{section}

\usepackage{etoolbox}
\patchcmd{\thebibliography}{\chapter*}{\section*}{}{}
\usepackage[title]{appendix}

\usepackage{lipsum}

\newcommand{\commentout}[1]{{}} %

\newcommand{\abs}[1]{\left|#1\right|}

\newcommand{\bfa}{{\bf a}}

\newcommand{\bfb}{{\bf b}}

\newcommand{\bfc}{{\bf c}}

\newcommand{\bfe}{{\bf e}}

\newcommand{\bff}{{\bf f}}

\newcommand{\bfg}{{\bf g}}
\newcommand{\bfH}{{\bf H}}

\newcommand{\bfL}{{\bf L}}

\newcommand{\bfn}{{\bf n}}

\newcommand{\bfP}{{\bf P}}
\newcommand{\bfp}{{\bf p}}

\newcommand{\bfq}{{\bf q}}

\newcommand{\bfr}{{\bf r}}

\newcommand{\bfS}{{\bf S}}
\newcommand{\bft}{{\bf t}}

\newcommand{\bfu}{{\bf u}}
\newcommand{\bfV}{{\bf V}}
\newcommand{\bfv}{{\bf v}}
\newcommand{\bfw}{{\bf w}}

\newcommand{\bfx}{{\bf x}}

\newcommand{\bfz}{{\bf z}}

\newcommand{\bfpsi}{\boldsymbol{\psi}}

\newcommand{\bfxi}{\boldsymbol{\xi}}
\newcommand{\bfeta}{\boldsymbol{\eta}}
\newcommand{\bfvarphi}{\boldsymbol{\varphi}}
\newcommand{\bfPi}{\boldsymbol{\Pi}}

\newcommand{\grad}{\operatorname{grad}}
\newcommand{\rot}{\operatorname{rot}}
\newcommand{\curl}{\operatorname{curl}}
\renewcommand{\div}{\operatorname{div}}

\newcommand{\dd}{\,{\rm d}}

\newcommand{\vertiii}[1]{{\left\vert\kern-0.25ex\left\vert\kern-0.25ex\left\vert #1
    \right\vert\kern-0.25ex\right\vert\kern-0.25ex\right\vert}}
    \newcommand{\vertii}[1]{{\left\vert\kern-0.25ex\left\vert #1
    \right\vert\kern-0.25ex\right\vert}}

\makeatletter
\MHInternalSyntaxOn
\def\MT_leftarrow_fill:{%
  \arrowfill@\leftarrow\relbar\relbar}
\def\MT_rightarrow_fill:{%
  \arrowfill@\relbar\relbar\rightarrow}
\newcommand{\xrightleftarrows}[2][]{\mathrel{%
  \raise.55ex\hbox{%
    $\ext@arrow 0359\MT_rightarrow_fill:{\phantom{#1}}{#2}$}%
  \setbox0=\hbox{%
    $\ext@arrow 3095\MT_leftarrow_fill:{#1}{\phantom{#2}}$}%
  \kern-\wd0 \lower.55ex\box0}}
\MHInternalSyntaxOff
\makeatother

\newcommand{\LC}[1]{\textcolor{black}{#1}}
\newcommand{\RG}[1]{\textcolor{black}{#1}}

\begin{document}

\markboth{S. Cao, L. Chen \& R. Guo}{Immersed VEM for electromagnetic interface problems in 3D}

\title{Immersed Virtual Element Methods for Electromagnetic Interface Problems in Three Dimensions}

\author{Shuhao Cao}
\address{Division of Computing, Analytics, and Mathematics,\\
School of Science and Engineering, University of Missouri-Kansas City\\ 
Kansas City, MO 64110\\
scao@umkc.edu}

\author{Long Chen}
\address{Department of Mathematics, University of California Irvine\\
 Irvine, CA 92697\\
chenlong@math.uci.edu}

\author{Ruchi Guo}
\address{Department of Mathematics, University of California Irvine\\
 Irvine, CA 92697\\
ruchig@uci.edu}

\maketitle
\begin{abstract}
Finite element methods for electromagnetic problems modeled by Maxwell-type equations are highly sensitive to the conformity of approximation spaces, and non-conforming methods may cause loss of convergence. This fact leads to an essential obstacle for almost all the interface-unfitted mesh methods in the literature regarding the application to electromagnetic interface problems, as they are based on non-conforming spaces. In this work, a novel immersed virtual element method for solving a 3D $\bfH(\curl)$ interface problem is developed, and the motivation is to combine the conformity of virtual element spaces and robust approximation capabilities of immersed finite element spaces. The proposed method is able to achieve optimal convergence. To develop a systematic framework, the $H^1$, $\bfH(\curl)$ and $\bfH(\div)$ interface problems and their corresponding problem-orientated immersed virtual element spaces are considered all together. In addition, the de Rham complex will be established based on which the Hiptmair-Xu (HX) preconditioner can be used to develop a fast solver for the $\bfH(\curl)$ interface problem. 
\end{abstract}



\section{Introduction}

 In this article, we shall develop a systematic framework to construct three-dimensional (3D) $H^1$, $\bfH(\curl)$, and $\bfH(\div)$ virtual element spaces 
 involving discontinuous coefficients, referred to as the immersed virtual element (IVE) spaces, 
 that can be used to solve the corresponding interface problems described in Section \ref{subsec:model} on unfitted meshes. 
 The proposed method is particularly important for electromagnetic interface problems 
 as the current unfitted-mesh methods in the literature have essential difficulty in handling $\bfH(\curl)$ problems, 
 see the detailed discussion in Section \ref{subsec:challenges}.
 
\subsection{Model problems}
\label{subsec:model}

Let $\Omega\subseteq\mathbb{R}^3$ denote an open and bounded modeling domain, 
and a subdomain $\Omega^-\subseteq \Omega$ contains the medium which has the physical property distinguished from the background medium 
occupying the subdomain $\Omega^+ = \Omega\backslash \overline{\Omega^-}$.
The surface $\Gamma =\partial \Omega^-$ is called interface and assumed to be sufficiently smooth 
with the normal vector $\bfn$ pointing from $\Omega^+$ to $\Omega^-$. 
We introduce two discontinuous piecewise constant parameters representing the medium properties:
\begin{equation}
\label{discont_coef_alpbeta}
\alpha=
\begin{cases}
   \alpha^-   & \text{in} ~ \Omega^-, \\
   \alpha^+   & \text{in} ~ \Omega^+,
\end{cases}
~~~~~~~~
\beta=
\begin{cases}
   \beta^-   & \text{in} ~ \Omega^-, \\
   \beta^+   & \text{in} ~ \Omega^+,
\end{cases}
\end{equation}
where $\alpha^{\pm}$ and $\beta^{\pm}$ are assumed to be positive constants.

The classic $H^1$-elliptic interface problem reads as
\begin{equation} 
\label{model_H1}
\begin{aligned}
      -\nabla\cdot(\beta\nabla u)&=f \quad \text{ in }  \Omega^-  \cup \Omega^+, 
\end{aligned}
\end{equation}
with $f\in L^2(\Omega)$, subject to certain boundary conditions on $\partial \Omega$ and jump conditions on the interface $\Gamma$:
\begin{subequations}
\label{H1_jump}
\begin{align}
       [u]_{\Gamma} & := u^+ - u^- = 0, \\
  [\beta \nabla u \cdot \bfn]_{\Gamma} &:=  \beta^{+}\nabla u^{+}\cdot \mathbf{ n}  - \beta^{-}\nabla u^{-}\cdot \mathbf{ n} =0,
\end{align}
\end{subequations}
where the parameter $\beta$ may represent, for example, the conductivity in electrical applications~\cite{2005HolderDavid,2010VallaghePapadopoulo}, 
or the dielectric constant in Poisson-Boltzmann equations~\cite{2007ChenHolstXu,2015YingXie}.

For electromagnetic interface problems, we consider the following $\curl\curl$-elliptic model
\begin{equation}
\label{model}
\curl (\alpha \curl\, \bfu) + \beta \bfu = \bff \;\;\;\;  \text{in} \; \Omega = \Omega^-  \cup \Omega^+, 
\end{equation}
which is derived from discretizing a time-dependent Maxwell system in which the magnetic field is eliminated.
Here for simplicity we assume $\bff\in\bfH(\div;\Omega)$. 
If positive piecewise constant parameters $\epsilon$, $\sigma$ and $\mu$ represent the electric permeability, 
conductivity and magnetic permeability of the medium respectively,
then $\alpha = \mu^{-1}$ and $\beta = \epsilon\triangle t^{-2} + \sigma \triangle t^{-1}$. 
Due to the interface, the following jump conditions are imposed for the electrical field $\bfu^{\pm}$ at the interface:
\begin{subequations}
 \label{inter_jc}
\begin{align}
[\bfu \times \bfn]_{\Gamma} &:=  \bfu^+\times \bfn -  \bfu^-\times \bfn = 0,  \label{inter_jc_1} \\
[\alpha \curl\, \bfu\times \bfn]_{\Gamma} &:=  \alpha^+\curl\, \bfu^+\times \bfn -   \alpha^- \curl\, \bfu^-\times \bfn   = 0,
\label{inter_jc_2} \\
[\beta \bfu\cdot \bfn]_{\Gamma} &:=  \beta^+ \bfu^+\cdot \bfn -   \beta^- \bfu^-\cdot \bfn   = 0.
\label{inter_jc_3}
\end{align}
\end{subequations}
In the problem above, $\beta$ has a similar physical meaning to that in the $H^1$ interface problem, for example, to represent the electric conductivity. 
\RG{In addition, the condition \eqref{inter_jc_1} comes from the tangential continuity of electrical fields. 
As $\alpha \curl \bfu$ describes the temporal change of magnetic fields, 
the condition \eqref{inter_jc_2} is related to the tangential continuity of magnetic fields.}  
In fact, such physical relations are naturally encoded in a de Rham complex, see the discussion in Section \ref{sec:Sspace}. 

Electromagnetic interface problems are of great importance due to a large variety of science and engineering applications. 
Typical examples include electromagnetic motors and actuators involving metal-air or metal-metal interface~\cite{1996BrauerRuehl,2014ErcanJaewook} 
and electromagnetic inverse scattering~\cite{1996ColtonKress,2003Monk} that use electromagnetic waves to detect objection. 
Solving the $\bfH(\curl)$ interface problem with optimal convergence is a challenging goal that conventional unfitted-mesh methods fail to meet(e.g., see the discussion in Section \ref{subsec:challenges}), 
and trying to overcome this difficulty is the main motivation for the present research.

For only the purpose of completeness, the $\bfH(\div)$ interface problem is given by
\begin{equation}
\label{model_Hidv}
- \nabla \div(\bfu) + \alpha \bfu = \bff\quad \text{ in }  \Omega^-  \cup \Omega^+, 
\end{equation}
with $\bff \in \bfH(\curl;\Omega)$, subject to a certain boundary condition on $\partial \Omega$ and the jump conditions
\begin{subequations}
\label{Hdiv_jump}
\begin{align}
       [\bfu\cdot\bfn]_{\Gamma} & := \bfu^+\cdot\bfn - \bfu^-\cdot\bfn = 0, \\
    [\alpha  \bfu \times \bfn]_{\Gamma} &:=  \alpha^{+} \bfu^{+} \times \mathbf{ n}  - \alpha^{-} \bfu^{-} \times \mathbf{ n} = \mathbf{ 0}, \\
  [\div(\bfu)]_{\Gamma} : &= \div(\bfu^+) - \div(\bfu^-) = 0.
\end{align}
\end{subequations}
The system comes from a mixed finite element method with a gradient formulation~\cite{1997ARNOLDRICHARDWINTHER}. 
The related $\bfH(\div)$ interface problem and $\bfH(\div)$-immersed element have been discussed in Ref.~\refcite{2010HiptmairLiZou,2022JiImmersed} 
and {thus will not be the focus of this work}. 
The parameter $\alpha$ here is inherited from the $\bfH(\curl)$ case.

\subsection{Challenges of electromagnetic interface problems on unfitted meshes}
\label{subsec:challenges}

For conforming finite element methods (FEMs) to perform optimally, the mesh has to fit or approximate the interface geometry ``well enough''. 
However, an efficient high-quality 3D mesh generation itself remains a challenging problem, which is particularly expensive for complicated geometries (see e.g., Chapter 5.6 in Ref.~\refcite{2014LoFinite}). 
A promising solution, to alleviate the difficulty in mesh generation, is to generate a cheap background unfitted mesh, and then to further triangulate those elements cut by the interface~\cite{2009ChenXiaoZhang}. 
The modification is highly efficient since this extra procedure only needs to be done locally around the interface. 
\RG{However, this approach in general cannot yield shape-regular elements near the interface; instead the shape regularity of triangulation is relaxed to the maximum angle condition.} 
The interpolation estimates based on the maximum angle condition have been widely studied for Lagrange elements~\cite{1976BabuskaAziz}, Raviart-Thomas elements~\cite{1999AcostaRicardo,1985Donatella}, and 3D N\'ed\'elec elements~\cite{2005BuffaCostabelDauge}.
This approach is very successful in the two-dimensional (2D) case~\cite{2015ChenwuXiao} as an admissible local triangulation satisfying the maximum angle condition always exists for a shape-regular background mesh, e.g., see Lemma 3.1 in Ref. \refcite{2021CaoChenGuo} and Proposition 2.4 in Ref. \refcite{2017ChenWeiWen}. 
Nevertheless, in the 3D case, these locally re-meshed triangulations may not form a globally admissible mesh, as it may not necessarily satisfy the Delaunay property and/or the maximum angle condition due to the existence of slivers~\cite{2000Pflaum}.

To overcome the difficulty of 3D mesh generation, Ref.~\refcite{2017ChenWeiWen} proposed a novel method, which uses polyhedra as interface-fitted elements cut from a background Cartesian mesh, rather than to further triangulate to obtain a tetrahedral mesh. 
To handle the discretization on polyhedra, a virtual element method (VEM)~\cite{Beirao-da-Veiga;Brezzi;Cangiani;Manzini:2013principles} is used. 
In fact, $\bfH(\curl)$ virtual element spaces have been constructed and applied to Maxwell's equations in Ref. \refcite{BEIRAODAVEIGA2021,2017VeigaBrezziDassiMarini,2016VeigaBrezziMarini,2020BeiroMascotto}. 
However, the analysis for $\bfH(\curl)$ problems is quite a different story. Some more recent error analysis for VEM, e.g., the ones developed in Ref.~\refcite{beirao2017stability,2018BrennerSung,Cao;Chen:2018Anisotropic}, cannot be directly used to obtain even optimal error estimates, 
and some more delicate techniques are needed on an ad hoc basis, e.g.,Ref.~\refcite{BEIRAODAVEIGA2021,2021CaoChenGuoIVEM,2017VeigaBrezziDassiMarini}. 
For interface problems, an extra layer of difficulty is to make error bounds robust with respect to potential anisotropic element shapes. 
In Ref.~\refcite{Cao;Chen:2018Anisotropic}, a more rigorous analysis is given on anisotropic elements generated from Cartesian meshes cut by the interface for the 2D $H^1$ case.

Meanwhile, on unfitted meshes, another direction to circumvent the mesh generation issue is to modify finite element (FE) spaces such that the new spaces can capture the jump behaviors in an optimal sense. 
There have been extensive works in this direction including immersed finite element (IFE) methods~\cite{2015AdjeridChaabaneLin,2015LinLinZhang,2010GongLi}, CutFEM or Nitsche's penalty methods~\cite{Bordas_Burman_Larson_Olshanskii2017,2015BurmanClaus,1971Nitsche,2010WuXiao,2020LiuZhangZhangZheng}, multiscale FEMs~\cite{2010ChuGrahamHou} and so on, which are widely applied to various interface problems. 
We also refer readers to FDTD methods~\cite{2004ZhaoWei} based on finite difference formulation for Maxwell's equations with material interfaces. 
For almost all the unfitted-mesh methods in the literature, the FE space modification is usually applied element-wise or piecewise relative to the interface. 
Thus, this practice results in discontinuities across \RG{interface} elements' non-interface boundaries or at the interface. 
Such non-conformity can be handled by penalties on element boundaries to impose continuity such as in the IFE methods~\cite{2007GongLiLi,2016GuoLin,2015LinLinZhang}, or on the interface itself to impose jump conditions such as the Nitsche's methods~\cite{Bordas_Burman_Larson_Olshanskii2017,2015BurmanClaus,1971Nitsche,2010WuXiao,2020LiuZhangZhangZheng}. 
With suitable penalties, robust optimal convergence rates can be indeed obtained for the $H^1$-type interface problems \eqref{model_H1}, but to the authors' best knowledge, not the considered electromagnetic interface problem \eqref{model}.

Compared with the analysis for $H^1$ problems, the most drastic difference stems from the underlying Sobolev space $\bfH^s(\curl;\Omega)$. 
In particular, for many non-conforming and discontinuous Galerkin (dG)-type methods, 
one needs to estimate the penalty term which, by the standard techniques (e.g., see Lemma 5.52 in~Ref. \refcite{2003Monk}), leads to estimates as follows
\begin{equation}
\label{h-1stab}
h^{-1/2}\| \bfu - \pi_F \bfu \|_{\bfL^2(F)} \lesssim h^{s-1} \| \bfu \|_{\bfH^s(\curl;K)},
\end{equation}
where $\pi_F$ is a certain projection operator on a face $F$ of an element $K$. 
The order in \eqref{h-1stab} implies that even a moderate regularity $s=1$ yields no approximation order due to the presence of a penalty/stabilization term in the form of $h^{-1}\int_F [\mathbf{u}_h\times \mathbf{ n}]\cdot [\mathbf{ v}_h\times \mathbf{ n}] \dd s$. 
\RG{For standard dG methods, the work in Ref.~\refcite{2004HoustonPerugiaSchotzau,2005HoustonPerugiaSchneebeli} can circumvent the suboptimality, 
and the analysis relies on a $\bfH(\curl)$-conforming subspace of the broken dG space on tetrahedral meshes.}
However, for unfitted-mesh methods for interface problems in the literature, this problem becomes more severe, since such a conforming subspace may not exist. 
Numerically, the loss of convergence has been observed and reported in a series of works~\cite{2016CasagrandeHiptmairOstrowski,2016CasagrandeWinkelmannHiptmairOstrowski,2020GuoLinZou} for the $\bfH(\curl)$ interface problem. 
\RG{In Section \ref{sec:hcurl-examples}, we also present one numerical example to show that a penalty-type IFE method cannot achieve optimal convergence.} 
So we believe that this difficulty is essential rather than caused by the limitation of analysis techniques.

The scaling factor $h^{-1}$ in this essential issue is commonly used for stabilization in dG methods, 
but shows to be too ``strong'' for the space $\bfH^1(\curl;\Omega)$. 
In fact, for a Lipschitz domain $D$, it is well-known that the trace of $H^1(D)$ is in $H^{1/2}(\partial D)$. 
While for $\bfH(\curl;D)$, the tangential trace is merely in $\bfH^{-1/2}(\div;\partial D)$, 
which should lead to different scaling factors for the stabilization terms on faces. 
Here we refer readers to Ref. \refcite{2008BrennerCuiLiSung} for the analysis of the relation between the scaling factor of a non-conforming method and function's regularity, and more recently a weighted Sobolev space treatment~\cite{2022BarkerCaoSternnonconforming}. 
In summary, the scaling factor, which is traditionally viewed to be strong enough to ensure stability for $H^{1}$ problems, leads to suboptimal convergence in non-conforming methods for $\bfH(\curl)$ problems.
On the contrary, various conforming VEMs~\cite{BeiraodaVeigaBrezziDassiEtAl2018Lowest,BEIRAODAVEIGA2021,2021CaoChenGuo,2020BeiroMascotto},
can use a correct scaling $h$ in \eqref{h-1stab} to achieve optimal convergence, but they are not easy to adapt to 3D anisotropic meshes. 
More recently in Ref. \refcite{2022Guomaximum}, the virtual element method for the 3D $H^{1}$-interface problem is analyzed under the setting of anisotropic meshes near the interface.

In conclusion, developing unfitted-mesh methods for the $\bfH(\curl)$ interface problem is much more challenging than its $H^1$ counterpart. 
For non-matching mesh methods, some work can obtain optimal convergence under the usual $\bfH^1(\curl)$-regularity by making a certain assumption of meshes being coupled at the interface, see Ref. \refcite{2008HuShuZou,2000ChenDuZou}.
For many unfitted-mesh methods, the meshes or spaces are generally completely broken, 
then at least the $H^2$-regularity has to be assumed to achieve optimal convergence, e.g., see Ref. \refcite{2001BenBuffaMaday,2020LiuZhangZhangZheng}. 
In Ref.~\refcite{2020GuoLinZou} for the 2D case and Ref.~\refcite{2022ChenGuoZoufamily} for the 3D case, Petrov-Galerkin methods are developed that can achieve optimal convergence, but results in a non-symmetric scheme. 
A robust optimal convergence for VEM is established in Ref.~\refcite{2021CaoChenGuo}, but it relies on a ``virtual'' triangulation satisfying the maximum angle condition, which may not be available in 3D. 
Therefore, to our best knowledge, currently there seems no satisfactory methodology for the 3D $\bfH(\curl)$ interface problem considered.

\subsection{A Novel Method}
To develop unfitted-mesh methods for the $\bfH(\curl)$ interface problem, based on the discussion above, it is preferable to use a conforming space. In the meantime, this space must admit sufficient approximation capabilities robust with respect to the anisotropy of subelements. 
This consideration motivates us to combine the conformity of virtual element spaces and the approximation capabilities of IFE spaces. 
In our recent work~\cite{2021CaoChenGuoIVEM}, we have successfully realized this idea for the 2D case, which is referred to as immersed virtual element (IVE) methods.

The fundamental idea is to impose local PDEs on interface elements to enforce both conformity and jump condition of which the solutions are used as the spaces for discretization. 
The IVE spaces can be understood as a special family of $H^1$, $\bfH(\curl)$ and $\bfH(\div)$ virtual element spaces~\cite{Beirao-da-Veiga;Brezzi;Cangiani;Manzini:2013principles,2017VeigaBrezziDassiMarini,2016VeigaBrezziMarini,BEIRAODAVEIGA2021} with discontinuous coefficients. 
For the $H^1$ case, it is also exactly the space of special FEMs by Babu\v{s}ka et al.~\cite{1994BabuskaCalozOsborn,1983BabuskaOsborn} for a simple 1D case, 
and it becomes the multiscale FE space~\cite{2010ChuGrahamHou} for higher dimensional cases where the local PDEs are solved on sub-grids. 
The similar idea was also employed in the enriched IFE method~\cite{2020AdjeridBabukaGuoLin}.
The proposed IVE discretization follows the meta-framework of VEM: the local PDEs need not be solved exactly, certain projections are computed instead with sufficient approximation capability to capture the jump conditions. 
It can successfully yield optimal convergence rates for the $\bfH(\curl)$ interface problem, which has been rigorously proved in the 2D case~\cite{2021CaoChenGuoIVEM}. 
In this work, we focus on the development of the IVE spaces, the scheme, and the implementation in the 3D case. 
We leave the theoretical part to another upcoming work \LC{as a rigorous error analysis involves much more technicalities in 3D and is not a trivial generalization of that in 2D}.

\RG{Developing 3D IVE spaces is significantly more complicated than the 2D case, especially for the $\bfH(\curl)$ space. }
An immediate question is how to design appropriate $\div$-$\curl$ systems as local problems with discontinuous coefficients that have a rigorous well-posedness.  
Here, special attention must also be paid to designing the local problems such that their solutions have computable projections to IFE spaces. 
The key is to modify the source terms for the local problems and to construct certain weighted projections with regard to the weights as Hodge star operators. 
The second issue is to design suitable trace spaces on element boundaries, in which the functions serve as the boundary conditions for those local problems. 
The trace spaces need to provide sufficient and robust approximation properties. 
In the 2D case, the boundary space consists of piecewise constants or linear functions on each edge, 
the simpleness of which attributes to the trivial geometry of the element boundaries, see Ref.~\refcite{Beirao-da-Veiga;Brezzi;Cangiani;Manzini:2013principles,Cao;Chen:2018Anisotropic,2021CaoChenGuoIVEM}. 
However, in the 3D case, it becomes much more obscure. 
For the classical virtual spaces~\cite{Beirao-da-Veiga;Brezzi;Cangiani;Manzini:2013principles,2017VeigaBrezziDassiMarini,2016VeigaBrezziMarini,BEIRAODAVEIGA2021}, 
the trace spaces are generally formulated by solutions of some extra 2D local problems defined on polygonal faces. 
In this work, we propose a rather different yet simpler approach: 
to use the standard FE spaces defined on a 2D triangulation satisfying the maximum angle condition on each element face. 
Such a triangulation not only benefits the robust approximation property due to the maximum angle condition, 
but also facilitates the code development leading to an efficient implementation with suitable data structures. 
\LC{In summary, on the boundary faces, we opt for an interface-fitted 2D triangulation and use local problems to extend the shape functions to the interior of each interface element.} 
Hence, the present research is, in fact, a combination of the three classical methodologies: VEM, IFE and FEM, towards solving the challenging electromagnetic interface problem \LC{efficiently}.

Although our focus is on electromagnetic interface problems, we shall develop a systematic framework for all the $H^1$, $\bfH(\curl)$ and $\bfH(\div)$ interface problems contributing to a solid mathematical foundation. 
They are connected by the following de Rham complex, and are shown to have the usual nodal, edge, and face degrees of freedom (DoFs), through which both the exact sequence and the commutative property can be established.

\begin{equation}
   \label{IVE_deRham}
\adjustbox{scale=0.9,center}{%
\begin{tikzcd}
  \mathbb{R} \arrow[r,"\hookrightarrow",shorten >= 0.5ex] 
   & H^2(\beta;\mathcal{T}_h) \arrow[d, "I^n_h"] \arrow[r, "\grad"] 
   & \bfH^1(\curl,\alpha,\beta;\mathcal{T}_h) \arrow[d, "I^e_h"] \arrow[r, "\curl"] 
   & \bfH^1(\div,\alpha;\mathcal{T}_h) \arrow[d, "I^f_h"] \arrow[r, "\div"] 
   & H^1(\mathcal{T}_h) \arrow[d, "\Pi^0_h"] \arrow[r] & 0 
   \\
  \mathbb{R} \arrow[r, "\hookrightarrow",shorten >= 1ex] 
   & V^n_h \arrow[r, "\grad", shorten <= 1ex, shorten >= 1ex] 
   & \bfV^e_h \arrow[r, "\curl", shorten <= 1ex, shorten >= 1ex] 
   & \bfV^{f}_h \arrow[r, "\div", shorten <= 1ex, shorten >= 1ex] 
   & Q_h \arrow[r, shorten <= 1ex] & 0
\end{tikzcd}
}
\end{equation}

Another major challenge for 3D interface problems is an appropriate fast solver. Multigrid methods are widely used, and we refer readers to Ref.~\refcite{2016XuZhang} for FEM and Ref.~\refcite{2017ChenWeiWen} for VEM, both of which study the $H^1$-interface problem. 
For $H(\curl)$ equations, the fast solvers are even more challenging~\cite{2011XuZhu,2007HiptmairXu} due to the non-trivial kernel space \LC{of the $\curl$ operator}, which is another motivation to lay out the de Rham complex \eqref{IVE_deRham} for the proposed VEM spaces.
In this work, we generalize multigrid-based Hiptmair-Xu (HX) preconditioner~\cite{2007HiptmairXu,ChenWuEtAl2018Multigrid} from regular $\bfH(\curl)$ problem to the interface case. 
Moreover, for fitted mesh methods, the condition numbers may still suffer from the possible anisotropic element shapes, even though the error bounds are robust. In Ref.~\refcite{2016XuZhang}, 
the DoFs near the interface and in the background mesh are split to form ``fine-coarse'' block matrices, thus an optimal two-level solver is developed. 
In this paper, a block diagonal smoother is proposed to handle the anisotropic element shape near the interface, similar to the practice in Ref.~\refcite{2016XuZhang}. 
To our best knowledge, this is the first research towards applying the HX preconditioner to VEM, and also the first fast solver for unfitted-mesh methods for solving electromagnetic interface problems. 
Numerical results demonstrate that the solver is robust with respect to \LC{both the mesh size and the shape of small-cutting elements}.

This article has additional $6$ sections. In the next section, we introduce the meshes, focusing especially on the element boundary triangulation. 
In Section \ref{sec:Sspace}, we describe the desired Sobolev spaces encoding jump conditions that are to be approximated. 
In Sections \ref{sec:IFE} and \ref{sec:IVE}, we introduce IFE spaces and IVE spaces, respectively. 
In Section \ref{sec:IVEscheme}, we describe the computation scheme, fast solvers, and implementation aspects. 
In the last section, we present a group of numerical experiments.

\section{Meshes}

In this article, we focus on a given interface-independent and shape-regular tetrahedral mesh of $\Omega$, 
but the proposed method can be also adapted to any Cartesian cubic meshes. 
This tetrahedral mesh is referred to as the background mesh, and is denoted by $\mathcal{T}_h$.
If an element in $\mathcal{T}_h$ intersects the interface, then it is called an interface element, or a non-interface element otherwise. 
The collection of interface elements is denoted as $\mathcal{T}^{i}_h$. 
For each $K\in\mathcal{T}^i_h$, we denote $\Gamma^K = \Gamma\cap K$. 
Let $\mathcal{T}^{n}_h$ be the set of non-interface elements. \LC{All elements are considered as open sets}.

One of the critical ingredients to formulate the $H^1$, $\bfH(\curl)$ and $\bfH(\div)$ local interface problems is to impose appropriate boundary conditions on element boundary. 
Different from the prevailing approach in the literature, 
the authors in Ref. \refcite{2017ChenWeiWen} proposes a novel approach by using exclusively square and triangular faces in interface polyhedra. 
In Ref.~\refcite{2017ChenWeiWen}, a Delaunay triangulation routine is called for the nodes in the background mesh, interface nodes, and some added vertices near the interface, 
then the triangular faces are extracted for their corresponding polyhedra. 
In this work, similar to the practice in Ref.~\refcite{2017ChenWeiWen}, standard FE functions on a 2D triangulation of any given element face are used as the boundary conditions. 
Hence, we make a fundamental assumption called \textit{interface fitted boundary triangulation}:
\begin{itemize}
  \item[(\textbf{A})] \label{asp:A} For each interface element $K$, each of its face admits a triangulation satisfying the maximum angle condition. 
  The triangles are formed by only the vertices of $K$ and/or the cutting points of the original interface, i.e., there are no newly-added interior vertices to form the edges. 
  If a face is cut by the interface, then this triangulation must be fitted to the interface \LC{in the sense that the curve (the intersection of the face and interface) is approximated by an edge of this triangulation with error in the order of $\mathcal O(h_K^2)$.}
\end{itemize} 
We illustrate the Assumption \hyperref[asp:A]{A} in Figure~\ref{fig:face_mesh}: 
each face of an interface element is partitioned into multiple triangles by the newly added edges including the one connecting the cutting points (red points in the figure). 
It can be understood that a local 2D fitted mesh satisfying the maximum angle condition is generated around the interface but only on faces. 
Although, as aforementioned in the introduction, generating a 3D interface-fitted mesh may be difficult or even impossible in certain situations, 
it is much easier to generate a 2D interface-fitted mesh. 
In particular, since the considered original background tetrahedral meshes only have triangular faces,
the boundary triangulation with the maximum angle condition is always guaranteed by~Lemma 3.1 in Ref. \refcite{2021CaoChenGuo}. 
See, for example, the left plot in Figure \ref{fig:face_mesh} if the element is cut by the interface only once. 
If cubic meshes are used, then~Proposition 2.4 in Ref. \refcite{2017ChenWeiWen} guarantees an admissible \textit{boundary triangulation}. 
But unlike Ref.~\refcite{2017ChenWeiWen} where the element is divided into two polyhedrons, here the cut tetrahedron is still treated as one element with more than $4$ triangular faces. 

We highlight that the \textit{interface fitted boundary triangulation} is able to link the fitted 2D and unfitted 3D meshes, 
and also bridges the standard 2D FE spaces and 3D virtual element spaces. 
Our previous work in the 2D case~\cite{2021CaoChenGuoIVEM} suggests that it is also one of the keys to overcome suboptimal convergence caused by non-conforming spaces for Maxwell's equations, 
as well as help in anisotropic analysis for the virtual spaces. 
In addition, the proposed \textit{boundary triangulation}, in fact, greatly benefits the computation. 
One of the difficult aspects of implementing polytopal finite element approximation is the ever-changing number of DoFs in an element. 
In our approach, since only triangular faces are present in every interface or non-interface element, 
the assembling can be uniformly handled by fixed-width matrices in the face-oriented data structure, 
please refer to Section~\ref{subsec:implement} for details, see also Ref.~\refcite{2017ChenWeiWen,BeiraodaVeigaBrezziDassiEtAl2018Lowest} for a face-based approach.

Another advantage of the proposed method is the flexibility to handle complex interface element geometry. 
Specifically, on elements that are cut by interface multiple times, e.g., see the right plot in Figure~\ref{fig:face_mesh}, 
the proposed IVE spaces can be easily constructed as long as an admissible \textit{boundary triangulation} can be constructed.

\begin{figure}
\centering
\includegraphics[width=1.2in,height=1.4in]{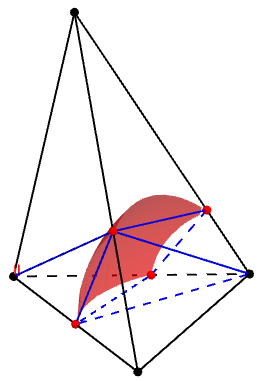}
~~
   ~~
   \includegraphics[width=1.2in,height=1.4in]{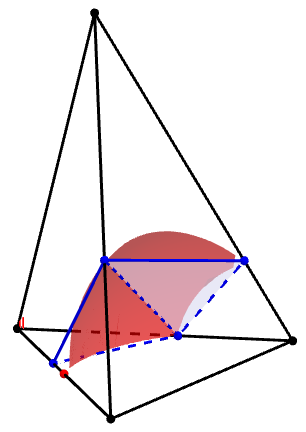}
   ~~
   \includegraphics[width=1.7in]{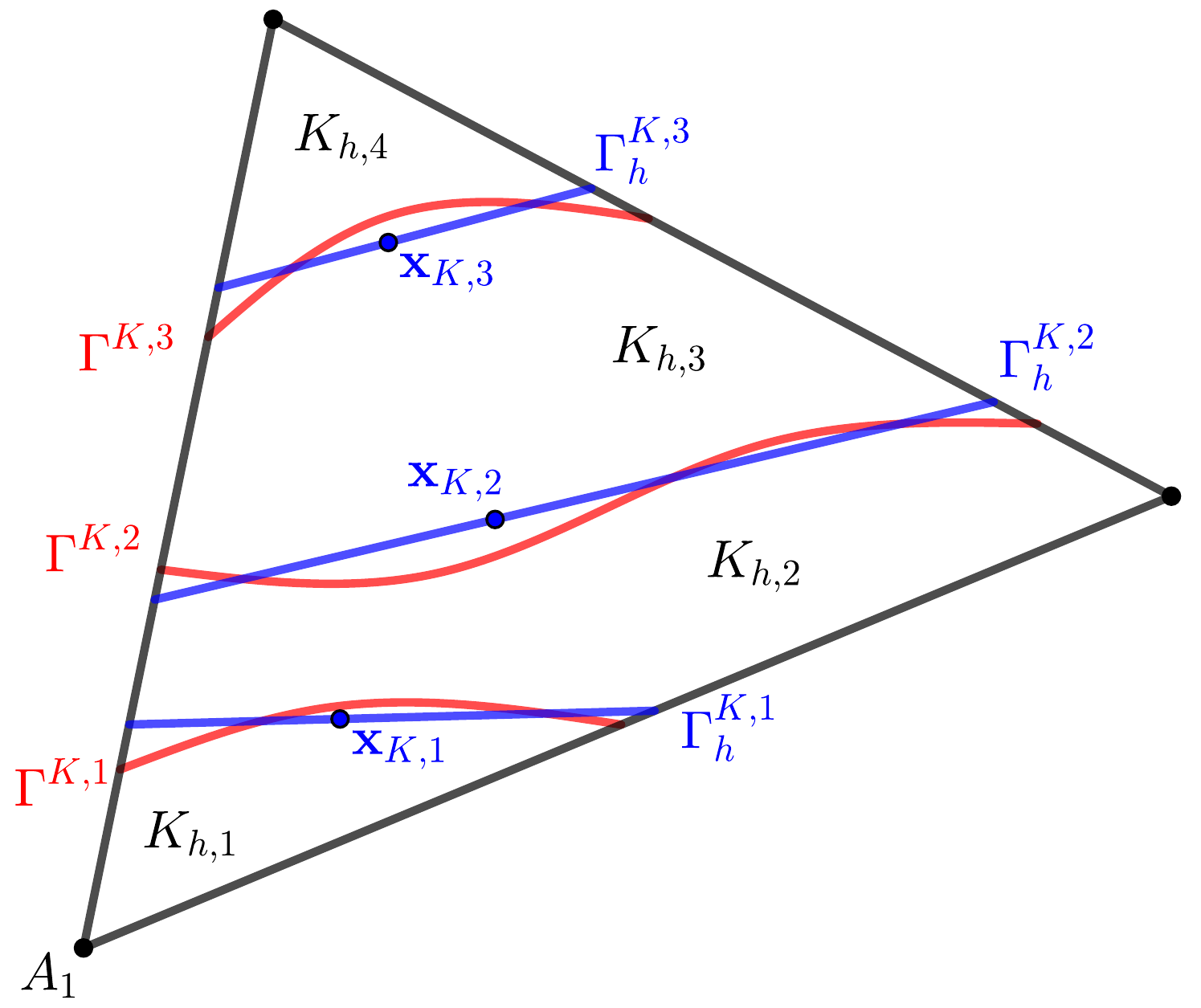}
\caption{Left: illustration of boundary triangulation of Assumption \hyperref[asp:A]{A}. Middle: an approximate plane $\Gamma^K_h$ to $\Gamma^K$. Right: 2D illustration of an element cut by the interface multiple times.}
\label{fig:face_mesh}
\end{figure}

\section{Some Sobolev Spaces and Well-posedness}
\label{sec:Sspace}

In this section, we describe a group of modified Sobolev spaces that incorporate the interface conditions. 
Let us first recall some standard spaces. 
Given an \LC{open} subdomain $D\subseteq \Omega$, for $s \ge 0$, we let $H^s(D)$ be the standard scalar Sobolev space and $\bfH^s(D) := (H^s(D))^3$. 
Now introduce %
\begin{subequations}
\label{sob_space}
\begin{align}
    & \bfH^s(\curl;D) = \{ \bfu\in \bfH^s(D): \curl \, \bfu\in \bfH^s(D) \},  \\
    & \bfH^s(\div;D) = \{ \bfu\in \bfH^s(D): \div \, \bfu\in H^s(D) \}.
\end{align}
\end{subequations}
If $D\cap\Gamma \neq \emptyset$, we let $D^{\pm}=\Omega^{\pm}\cap D$, and further let $H^s(\cup D^{\pm})$, $\bfH^s(\curl;\cup D^{\pm})$ and $\bfH^s(\div;\cup D^{\pm})$ consist of functions that belong to the corresponding spaces on each $D^{\pm}$ but without any conditions on $\partial D^{\pm}$.

With the mesh $\mathcal{T}_h$ we are ready to define the interface-encoded Sobolev spaces:
\begin{subequations}
\label{tildeHspace}
\begin{align}
    H^2(\beta;\mathcal T_h) = & \, H^1(\Omega)  \cap  \{ u \in H^2(\cup K^{\pm}) : \beta\nabla u \in \bfH(\div;K),  \, \forall K\in\mathcal{T}_h \},  \label{tildeHspace1} \\
    \bfH^1(\curl,\alpha,\beta;\mathcal T_h)  =  &\, \bfH(\curl;\Omega) \cap \{ \bfu \in \bfH^1(\curl;\cup K^{\pm}) : \beta \bfu \in \bfH(\div;K), \notag \\
    & \quad\quad\quad\quad\quad\quad\quad\quad \alpha\curl\,\bfu \in \bfH(\curl;K), \, \forall K\in\mathcal{T}_h  \},  \label{tildeHspace2}    
    \\
    \bfH^1(\div,\alpha;\mathcal T_h)  = & \, \bfH(\div;\Omega) \cap   \{ \bfu\in \bfH^1(\div; \cup K^{\pm}) : \alpha \bfu \in \bfH(\curl;K), \notag \\
&\quad\quad\quad\quad\quad\quad\quad\quad \div \bfu\in H^1(K),  \, \forall K\in\mathcal{T}_h  \},   \label{tildeHspace3}  
    \\
    H^{1}(\mathcal T_h) = & \, L^2(\Omega)\cap \{u\in H^1(K), \, \forall K\in\mathcal{T}_h \}. \label{tildeHspace4} 
\end{align}
\end{subequations}
Note that, on a non-interface element $K$, the conditions are trivial since they are just consequences of $H^2(K)$, $\bfH^1(\curl;K)$ and $\bfH^1(\div;K)$. 
On an interface element $K$ those conditions exactly encode both the conformity and interface information. 
To see the relation more clearly, we let $H^2(\beta;K)$, $\bfH^1(\curl,\alpha,\beta;K)$, $\bfH^1(\div,\alpha;K)$ and $H^1(K)$ be the local spaces on $K$ of their global counterparts in~\eqref{tildeHspace}. 
The spaces above are just constructed so that the following diagram is well-defined:

\begin{equation}
   \label{Hodge_H_local}
\adjustbox{scale=0.9,center}{%
\begin{tikzcd}
  \mathbb{R} \arrow[r,"\hookrightarrow"] 
   & H^2(\beta;K) \arrow[d, "I"] \arrow[r, "\nabla"] 
   & \bfH^1(\curl,\alpha,\beta;K) \arrow[d, "\beta"] \arrow[r, "\curl"] 
   & \bfH^1(\div,\alpha;K) \arrow[d, "\alpha"] \arrow[r, "\div"] 
   & H^1(K) \arrow[d, "I"] \arrow[r]
   & 0
   \\
0 
   &    L^2(K) \arrow[l] 
   & \bfH(\div;K) \arrow[l, "\div"', shorten <= 1ex, shorten >= 1ex] 
   & \bfH(\curl;K)\arrow[l, "\curl"', shorten <= 1ex, shorten >= 1ex] 
   & H^1(K) \arrow[l, "\nabla"', shorten <= 1ex, shorten >= 1ex] 
   & \mathbb R \arrow[l,"\hookleftarrow"', shorten >= 1ex]
\end{tikzcd}
}
\end{equation}
In Diagram \eqref{Hodge_H_local}, $\alpha$ and $\beta$ could be understood as Hodge star operators mapping $k$-forms to $(3-k)$-forms for $k=2,1$, respectively. 
Take $\beta: \bfH^1(\curl,\alpha,\beta;K)  \to  \bfH(\div;K)$ as an example. A function $\bfu$ in $\bfH^1(\curl,\alpha,\beta;K)$ can be thought of as a vector proxy of a $1$-form. 
Then $\beta \bfu \in \bfH(\div;K)$ is a $2$-form. The jump conditions \eqref{inter_jc} on the interface are from the continuity of the mapped forms. 
Construction of the desired virtual spaces is to mimic this diagram in the discretized level.

In the following discussion, given any face $F$ in the mesh, we shall denote the tangential component of $\bfu$ by $\bfu^{\tau}|_F$ for admissible $\bfu$ defined in the bulk $\omega$ such that $F\subseteq \partial \omega$, 
and we will drop $|_F$ if there is no danger of causing confusion. 
In addition, we will also frequently use the 2D rotation operator denoted by $\rot_F$ on $F$. 
Let $\nabla_{F}$ denote the surface gradient. Then, for a function $\bfvarphi$ defined on $F$, $\rot_F$ is defined in the distributional sense such that 
\begin{equation}
\langle \rot_{F} \bfvarphi, v \rangle_{F} :=  (\bfvarphi, \nabla_{F} v \times \bfn)_{F}, \quad \forall v\in H_0^1(F).
\end{equation}
For each subdomain $\omega\subseteq\Omega$, and $\bfvarphi$ defined on $\partial \omega$, $\rot_{\partial \omega} \bfvarphi$ can be defined similarly, 
\begin{equation}
\langle \rot_{\partial \omega} \bfvarphi, v \rangle_{\partial \omega} :=  (  \bfvarphi, \nabla_F v  \times \bfn)_{\partial \omega}, \quad \forall v\in H^1(\omega),
\end{equation}
while it can be verified that $\rot_{\partial \omega} \bfvarphi|_F = \rot_{F} \bfvarphi$ \LC{when $\rot_{\partial \omega} \bfvarphi\in L^2(\omega)$}. Here
$\langle \cdot, \cdot \rangle_{\partial \omega}$ denotes the usual pairing between $H^{-1/2}(\partial \omega)$--$H^{1/2}(\partial \omega)$, 
and $\langle \cdot, \cdot \rangle_{F}$ is defined similarly for $F\subset \partial \omega$. 
In particular, the well-known formula states%
\begin{equation}
\label{rotcurl}
\curl \bfu \cdot\bfn_F = \rot_F \bfu,  ~~\text{ for }\bfu \in \bfH(\curl;\omega), ~~~ \text{on} ~ F
\subset \partial \omega,
\end{equation}
where $\bfn_F$ is the exterior unit normal vector of $F$ with respect to $\omega$. 

Note that the proposed global problems as well as the local problems all involve discontinuous coefficients. 
In order to pursue a rigorous definition of the IVE spaces, 
we discuss the well-posedness of some $\div$-$\curl$ systems with discontinuous coefficients. 
The systems with constant coefficients are discussed in~Ref. \refcite{1998AmroucheBernardiDaugeGirault}, and the results with general coefficients can be found in Ref.~\refcite{1982Saranen,1983Saranen,FernandesGilardi1997Magnetostatic}. 
But here we present more detailed analysis to show that the constants in the {\it a priori} estimates are independent of interface location. 

\begin{lemma}
\label{lem_divcurl}
Let $\omega$ be a Lipschitz polyhedral domain which is simply-connected, 
let an interface $\gamma$ separate $\omega$ into $\omega^{\pm}$, 
and define a piecewise constant function $b=b^{\pm}>0$ in $\omega^{\pm}$. 
For the data functions $\bff\in \bfH(\div;\omega)\cap\ker(\div)$, $h\in L^2(\omega)$ and $g\in H^{-1/2}(\partial \omega)$ 
such that the compatibility condition holds:
\begin{equation}
\label{lem_divcurl_com_eq}
\int_{\omega} h \dd x = \langle g, 1\rangle_{\partial \omega},
\end{equation}
then the following problem admits a unique solution $\bfvarphi\in\bfH(\div;\omega)$ and $b\bfvarphi\in\bfH(\curl;\omega)$ 
\begin{equation}
\label{lem_divcurl_eq0}
\curl\, (b  \bfvarphi) = \bff, ~~~ \div( \bfvarphi) = h ~~ \text{in} ~\omega, ~~~ \bfvarphi\cdot\bfn = g \quad \text{ on } \partial \omega.
\end{equation}
If additionally $g\in L^2(\partial\omega)$, the stability result holds:
\begin{equation}
\label{lem_divcurl_eq01}
b_{\min}(1+b_{\max})^{-1}\| \bfvarphi \|_{L^2(\omega)} \le C_{\omega} ( \| \bff \|_{L^2(\omega)} + \| h \|_{L^2(\omega)} + \| g \|_{L^2(\partial\omega)}),
\end{equation}
where $b_{\min} = \min\{b^-,b^+\}$, $b_{\max} = \max\{b^-,b^+\}$ and the constant $C_{\omega}$ only depends on the geometry of $\omega$. 
If $\omega$ is star-convex with respect to a ball of the radius $\rho_{\omega}$, then $C_\omega= C(h_{\omega}/\rho_{\omega})$ where $h_\omega$ is the diameter of $\omega$.
Furthermore, if $g\in H^{1/2}(\partial \omega)$, $\omega$ is convex, 
and $\gamma$ is a closed surface that is sufficiently smooth and does not intersect the boundary, 
then $\bfvarphi\in \bfH^1(\cup\omega^{\pm})$. 
\end{lemma}
\begin{proof}
Since $\omega$ is assumed to be simply-connected, 
by~Theorem 1.1 in Ref. \refcite{1983Saranen}, we know the solution $\bfvarphi$ to \eqref{lem_divcurl_eq0} admits the following decomposition
\begin{equation}
\label{lem_divcurl_eq1}
\bfvarphi = b^{-1} \nabla v + \curl \bfw
\end{equation}
where $v$ is the solution to the equation
\begin{equation}
\label{lem_divcurl_eq2}
\div(b^{-1} \nabla v) = h~~ \text{in} ~ \omega, ~~~ b^{-1} \nabla v\cdot \bfn = g ~~ \text{on} ~ \partial\omega, ~~~ \int_{\omega}v \dd x =0,
\end{equation}
and $\bfw$ is the solution to the equation
\begin{equation}
\label{lem_divcurl_eq3}
\curl(b\curl \bfw) = \bff ~~ \text{in} ~ \omega,~~~ \div(\bfw) = 0~~ \text{in} ~ \omega, ~~~ \bfw\times \bfn = \mathbf{ 0} ~~ \text{on} ~ \partial\omega.
\end{equation}
Here \eqref{lem_divcurl_eq2} is a standard well-posed elliptic interface problem and the well-posedness of \eqref{lem_divcurl_eq3} can be found in Ref.~\refcite{1982Saranen}. 

\RG{To show \eqref{lem_divcurl_eq01}, based on \eqref{lem_divcurl_eq1}, we show the a-priori estimates for both $v$ and $\bfw$ in terms of data. Testing \eqref{lem_divcurl_eq2} by $v$ and using integration by parts we have
\begin{equation}
\label{lem_divcurl_eq4}
\int_{\omega} b^{-1} \nabla v\cdot\nabla v \dd x = -\int_{\omega} h v \dd x + \int_{\partial\omega} g v \dd s.
\end{equation} 
It implies, with Poincar\'e inequality and the trace inequality for $v$, that
\begin{equation}
\begin{split}
\label{lem_divcurl_eq5}
 \| b^{-1/2} \nabla v \|^2_{L^2(\omega)} & \le \| h \|_{L^2(\omega)} \| v \|_{L^2(\omega)} + \| g \|_{L^2(\partial\omega)} \| v \|_{L^2(\partial\omega)} \\
& \lesssim (\| h \|_{L^2(\omega)} + \| g \|_{L^2(\partial\omega)}  ) \| \nabla v \|_{L^2(\omega)}.
\end{split}
\end{equation}
Cancelling one $\| \nabla v \|_{L^2(\omega)}$ in \eqref{lem_divcurl_eq5} yields the estimate: $b_{\max}^{-1} \| \nabla v \|_{L^2(\omega)}\lesssim \| h \|_{L^2(\omega)} + \| g \|_{L^2(\partial\omega)}$. As for $\bfw$, we note that $\bfw\in \bfH(\curl)\cap \bfH(\div)$ with $\div(\bfw) = 0$, and thus we can apply~Corollary 3.51 in Ref. \refcite{2003Monk} to obtain
\begin{equation}
\label{lem_divcurl_eq6}
\| \bfw \|_{L^2(\omega)} \lesssim \| \curl \bfw \|_{L^2(\omega)} + \| \bfw\times \bfn \|_{L^2(\partial\omega)} = \| \curl \bfw \|_{L^2(\omega)}.
\end{equation}
Then, testing \eqref{lem_divcurl_eq3} with $\bfw$, applying the integration by parts and using \eqref{lem_divcurl_eq6}, we have
\begin{equation}
\label{lem_divcurl_eq7}
b_{\min}\| \curl \bfw \|_{L^2(\omega)} \lesssim \| \bff \|_{L^2(\omega)}.
\end{equation}
Combining the estimates above, we have \eqref{lem_divcurl_eq01}, 
and the dependence of the generic constants follows from the constants in the trace and Poincar\'e inequalities used above, see Ref.~\refcite{2018BrennerSung}. %
}

For $g\in H^{1/2}(\partial \omega)$, convex $\omega$, 
and closed smooth $\gamma$ not intersecting the boundary, there certainly holds that $v\in H^2(\cup\omega^{\pm})$ and thus $\nabla v \in \bfH^1(\cup \omega^{\pm})$ \cite{2002HuangZou,2007HuangZou}. As for $\bfw$, following the argument of Theorem 5.2 in Ref. \refcite{2007HuangZou}, we construct $\tilde{\bfw}\in \bfH^1(\omega)$ such that $\curl(\tilde{\bfw}) = \bff$, $\div(\tilde{\bfw}) = 0$ and $\tilde{\bfw}\cdot\bfn = 0$ due to $\div(\bff) = 0$, see Theorem 3.8 in Ref. \refcite{2011GiraultRaviart}. Then, we have $\curl(b\curl\bfw - \tilde{\bfw}) = \mathbf{ 0}$ in $\omega$ and $(b\curl\bfw - \tilde{\bfw})\cdot\bfn = 0$ and $\partial\omega$, and conclude by exact sequence that $b\curl\bfw - \tilde{\bfw} = \nabla \phi$ for some $\phi\in H^1(\omega)$ such that
\begin{equation*}
\label{lem_divcurl_eq3_1}
\left\{
\begin{aligned}
    \div(b^{-1}\nabla \phi) =  -\div(b^{-1} \tilde{\bfw}) & ~~\text{in} ~ \omega^{\pm}, 
    \\
    [\phi]_{\gamma} = 0 & ~~\text{on} ~ \gamma, 
    \\
    [b^{-1}\nabla \phi\cdot\bfn]_{\gamma} = -[b^{-1}\tilde{\bfw}\cdot\bfn]_{\gamma} & ~~\text{on} ~ \gamma, 
    \\
    \nabla\phi \cdot\bfn = 0, & ~~\text{on} ~\partial \omega.
\end{aligned}
\right.
\end{equation*}
Note that this is an interface problem with the non-homogeneous flux jump condition. As $\div(b^{-1} \tilde{\bfw}|_{\omega^{\pm}})\in L^2(\omega^{\pm})$ and $\tilde{\bfw}\cdot\bfn|_{\gamma} \in H^{1/2}(\gamma)$, further by the assumption that $\gamma$ does not intersect $\partial\omega$, we have $\phi\in H^2(\cup\omega^{\pm})$ and thus $\nabla\phi\in \bfH^1(\cup\omega^{\pm})$ by Ref. \refcite{1998ChenZou}. Therefore, we conclude $\curl \bfw \in \bfH^1(\cup\omega^{\pm})$, and $ \bfvarphi \in \bfH^1(\cup\omega^{\pm})$ by the decomposition \eqref{lem_divcurl_eq1}.

\end{proof}

\begin{lemma}
\label{lem_divcurl2}
Given a simple-connected domain $\omega$ with Lipschitz boundary, 
let an interface $\gamma$ separates $\omega$ into $\omega^{\pm}$ and define a piecewise constant function $a=a^{\pm}>0$ in $\omega^{\pm}$. 
For the data functions $\bff\in \bfH(\div;\omega)\cap\ker(\div)$, $h\in L^2(\omega)$ and $\bfg\in \bfH^{-1/2}(\partial\omega)$ such that the compatibility condition holds: 
\begin{equation}
\label{lem_divcurl2_com_eq}
\langle \bff \cdot \bfn, v\rangle_{\partial \omega}  = 
\langle \rot_{\partial \omega} (\bfn\times \bfg), v \rangle_{\partial \omega} 
\quad \forall v\in H^1(\omega),
\end{equation}
then the following problem admits a solution $\bfvarphi\in\bfH(\curl;\omega)$ and $a\bfvarphi\in\bfH(\div;\omega)$ 
\begin{equation}
\label{lem_divcurl2_eq0}
\curl\, ( \bfvarphi) = \bff, ~~~ \div( a \bfvarphi) = h ~~ \text{in} ~\omega, ~~~ \bfvarphi \times \bfn = \bfg, ~~ \text{on} ~ \partial\omega.
\end{equation}
\RG{Furthermore, if $\bfg \in \mathbf{ L}^2(\partial\omega)$, the following a-priori estimate holds:
\begin{equation}
\label{lem_divcurl2_eq01}
a_{\min}(1+a_{\max})^{-1} \| \bfvarphi \|_{L^2(\omega)} \lesssim \| \bff \|_{L^2(\omega)} +  \| h \|_{L^2(\omega)} +  \| \bfg \|_{L^2(\partial\omega)},
\end{equation}
where $a_{\min} = \min\{a^-,a^+\}$, $a_{\max} = \max\{a^-,a^+\}$ and the constant $C_{\omega}$ only depends the geometry of $\omega$. If $\omega$ is star-convex, then $C_\omega= C(h_{\omega}/\rho_{\omega})$.}
\end{lemma}
\RG{
\begin{proof}
As $\omega$ is assumed to be simply-connected, by Theorem 1.2 in~Ref. \refcite{1983Saranen}, 
the solution $\bfvarphi$ to \eqref{lem_divcurl2_eq0} has the following decomposition:
\begin{equation}
\label{lem_divcurl2_eq1}
\bfvarphi = \nabla v + a^{-1} \curl (\bfw),
\end{equation}
where $v$ is the solution of
\begin{equation}
\label{lem_divcurl2_eq2}
\div(a \nabla v) = h ~~ \text{in} ~ \omega, ~~~ v =0 ~~ \text{on} ~ \partial\omega,
\end{equation}
and $\bfw$ is the solution of 
\begin{equation}
\begin{split}
\label{lem_divcurl2_eq3}
& \curl (a^{-1} \curl \bfw) = \bff, ~~~ \div(\bfw) = 0 ~~~~~ \text{in} ~ \omega,  \\
& a^{-1} \curl(\bfw) \times\bfn = \bfg, ~~~ \bfw\cdot\bfn = 0, ~~~~~~ \text{on} ~ \partial \omega.
\end{split}
\end{equation}
By the similar argument to Lemma \ref{lem_divcurl} with the last remark in Ref. \refcite{1990Costabel}, we have
\begin{equation}
\label{lem_divcurl2_eq5} 
\begin{split}
&a_{\min}\| \nabla v \|_{L^2(\omega)} \lesssim \| h \|_{L^2(\omega)} \\
&a^{-1}_{\max} \| \curl (\bfw) \|_{L^2(\omega)} \lesssim \| \bff \|_{L^2(\omega)} + \| \bfg \|_{L^2(\partial\omega)},
\end{split}
\end{equation}
which leads to the desired estimate by \eqref{lem_divcurl2_eq1}.
\end{proof}
\begin{remark}
\label{rem_apriori_est}
The key of the {\it a priori} estimates of \eqref{lem_divcurl_eq01} and \eqref{lem_divcurl2_eq01} is the independence with respect to the interface location. 
The result of Corollary 3.51 in Ref.~\refcite{2003Monk} employs a compactness argument for $\curl$-$\div$ systems on homogeneous media. 
One may indeed use this technique to obtain the similar estimates for interface problems, which, however, 
may lead to constants depending on the interface location. %
\end{remark}}

\RG{
\begin{lemma}
    \label{lem_curlcurl}
    Given a simple-connected domain $\omega$ with Lipschitz boundary, 
    let an interface $\gamma$ separates $\omega$ into $\omega^{\pm}$ and define two
    piecewise constant functions $a=a^{\pm}>0$ and $b=b^{\pm}>0$ in $\omega^{\pm}$. 
    For the data functions $\bff\in \bfH(\div;\omega)\cap\ker(\div)$ and 
    $\bfg\in \bfH(\rot;\partial\omega)$ 
    such that 
    \begin{equation}
       \label{lem_curlcurl_eq01}
    \int_{\partial\omega} \rot_{\partial\omega} \bfg \dd s = 0.
    \end{equation}
    Then, the equation
    \begin{equation}
        \label{lem_curlcurl_eq02}
        \curl (a \curl \bfvarphi) = \bff, ~~~ \div(b \bfvarphi) = 0, ~~~ \bfvarphi^{\tau} = \bfg,
    \end{equation}
    admits a unique solution $\bfvarphi$ satisfying $\bfvarphi\in \bfH(\curl;\omega)$, $b\bfvarphi\in\bfH(\div;\omega)$,
    $a\curl\bfvarphi\in \bfH(\curl;\omega)$.
\end{lemma}
\begin{proof}
    We first consider a potential $\bfpsi$ such that 
    \begin{equation}
    \label{lem_curlcurl_eq1}
    \curl (a \bfpsi) = \bff, ~~ \div(\bfpsi) = 0, ~~ \text{in} ~ \omega, 
    ~~~~ \bfpsi\cdot\bfn = \rot_{\partial \omega} \bfg,  ~~ \text{on} ~ \partial \omega.
    \end{equation}
    The condition \eqref{lem_curlcurl_eq01} together with Lemma \ref{lem_divcurl} shows the unique existence of $\bfpsi$. Then, it is easy to see that \eqref{lem_curlcurl_eq02} can be equivalently written as
    \begin{equation}
    \label{lem_curlcurl_eq2}
    \curl (\bfvarphi) = \bfpsi, ~~ \div(b \bfvarphi) = 0 ~~ \text{in} ~ \omega, ~~~ \bfvarphi^{\tau}= \bfg ~~~~ \text{on} ~ \partial \omega.
    \end{equation}
    Note that the boundary condition in \eqref{lem_curlcurl_eq2} is equivalent to 
    $\bfvarphi\times \bfn = \bfg\times\bfn$. 
    As $\bfn\times(\bfg \times \bfn) = \bfg$, with integration by parts, 
    we have for any $v\in H^1(\omega)$,
    \begin{equation}
    \label{lem_curlcurl_eq3}
    \int_{\partial \omega} (\bfn\times(\bfg \times \bfn))(\nabla v \times \bfn) \dd s = \int_{\partial \omega} \bfg(\nabla v \times \bfn) \dd s = \int_{\partial \omega}  (\rot_{\partial \omega} \bfg) v \dd s.
    \end{equation}
    Then, the boundary condition in \eqref{lem_curlcurl_eq1} shows that the compatibility condition in \eqref{lem_divcurl2_com_eq} indeed holds. 
    Thus, the well-posedness follows from Lemma \ref{lem_divcurl2}.
\end{proof}
}
At last, we present the complex formed by the new globally-defined spaces of~\eqref{tildeHspace}.

\begin{lemma}
\label{lem_IVE_deRham_H_complex}
The following sequence is a complex:
\begin{equation}\label{IVE_deRham_H_complex}
\adjustbox{scale=0.95,center}{%
\begin{tikzcd}[column sep=scriptsize]
  \mathbb{R} \arrow[r,"\hookrightarrow"] 
   & H^2(\beta;\mathcal{T}_h) \arrow[r, "\grad"] 
   & \bfH^1(\curl,\alpha,\beta;\mathcal{T}_h) \arrow[r, "\curl"] 
   & \bfH^1(\div,\alpha;\mathcal{T}_h)  \arrow[r, "\div"] 
   & H^1(\mathcal{T}_h)  \arrow[r]
   & 0.
\end{tikzcd}
}
\end{equation}
When $\Omega$ is a convex polyhedron, 
and the interface is also sufficiently smooth not intersecting $\partial \Omega$, it is also exact.
\end{lemma}
\begin{proof}

We first verify that $\nabla H^2(\beta;\mathcal{T}_h)\subseteq \ker(\curl)\cap \bfH^1(\curl,\alpha,\beta;\mathcal{T}_h)$. 
This is true due to the jump conditions associated with $\nabla H^2(\beta;\mathcal{T}_h)$, and $\curl\nabla H^2(\beta;\mathcal{T}_h) = 0$. 
Similarly, $\curl \bfH^1(\curl,\alpha,\beta;\mathcal{T}_h) \subseteq \ker(\div)\cap \bfH(\div,\alpha;\mathcal{T}_h)$ due to the jump condition associated with $\bfH^1(\curl,\alpha,\beta;\mathcal{T}_h)$, 
and the fact $\div\curl = 0$. 
Finally, it is trivial that $\div \bfH^1(\div,\alpha;\mathcal T_h) \subseteq H^{1}(\mathcal T_h)$. 
These results together finish the proof.

We then verify the exactness. We first show $\nabla H^2(\beta;\mathcal{T}_h)= \ker(\curl)\cap \bfH^1(\curl,\alpha,\beta;\mathcal{T}_h)$. 
Given each $\bfu\in \ker(\curl)\cap \bfH^1(\curl,\alpha,\beta;\mathcal{T}_h)$, by the classic exact sequence, 
there exists $u\in H^1(\Omega)$ such that $\nabla u =\bfu$, and by the jump conditions associated with $\bfH^1(\curl,\alpha,\beta;\mathcal{T}_h)$ we have $u$ also satisfies those of $H^2(\beta;\mathcal{T}_h)$. In addition, on each element $K$, $\nabla u = \bfu \in \bfH^1(\cup K^{\pm})$ implies $u\in H^2(\cup K^{\pm})$. Therefore, $u\in H^2(\beta;\mathcal{T}_h)$.

We then verify $\curl \bfH^1(\curl,\alpha,\beta;\mathcal{T}_h) = \ker(\div)\cap \bfH^1(\div,\alpha;\mathcal{T}_h)$. Given a function $\bfu\in \ker(\div)\cap\bfH^1(\div,\alpha;\mathcal{T}_h)$, we consider a function $\bfvarphi\in \bfH(\div;\Omega)$ satisfying
\begin{equation}
\label{lem_IVE_deRham_H_exact_eq1}
\curl\, (\beta^{-1}  \bfvarphi) = \bfu, ~~~ \div(\bfvarphi) = 0 ~~ \text{in} ~\Omega, ~~~ \bfvarphi\cdot\bfn =  0 ~~ \text{on} ~ \partial \Omega.
\end{equation}
By Lemma~\ref{lem_divcurl} with $\omega=\Omega$, $\gamma=\Gamma$ and $b=\beta^{-1}$, 
we have this system being well-defined with $\beta^{-1}\bfvarphi\in \bfH(\curl;\Omega)$ and $\bfvarphi\in\bfH_0(\div;\Omega)$. 
Using Lemma~\ref{lem_divcurl} again, by the geometric condition of $\Omega$ and $\Gamma$, we also have $\bfvarphi\in \bfH^1(\cup\Omega^{\pm})$. 
Thus, we obtain $\bfv := \beta^{-1}\bfvarphi \in \bfH(\curl;\Omega)\cap\bfH^1(\cup\Omega^{\pm})$ and $\beta \bfv \in \bfH(\div;\Omega)$. 
Furthermore, on each element $K$, $\bfu\in \bfH^1(\div;K^{\pm})$ implies $\bfv\in \bfH^1(\curl;K^{\pm})$. 
In addition, $[\alpha\curl\bfv\times \bfn]_{\Gamma^K}=0$ is trivial by the property of $\bfu$.

We next show $\div \bfH^1(\div,\alpha;\mathcal T_h) = H^{1}(\mathcal T_h)$. Given each $f\in H^{1}(\mathcal T_h)$. We consider a $\phi$ satisfying
\begin{subequations}
\label{lem_IVE_deRham_H_exact_eq2}
\begin{empheq}[left=\empheqlbrace]{align}
    \div(\alpha^{-1} \nabla \phi) = f, & ~~ \text{in} ~\Omega,  \label{lem_IVE_deRham_H_exact_eq2_1} 
    \\
    [\phi]_{\Gamma} =  0, ~~ [\alpha^{-1} \nabla \phi\cdot\bfn]_{\Gamma} = 0,  
    & ~~ \text{on} ~\Gamma\label{lem_IVE_deRham_H_exact_eq2_2} 
    \\
    \phi = 0 & ~~ \text{on} ~ \partial \Omega.  \label{lem_IVE_deRham_H_exact_eq2_3}
\end{empheq}
\end{subequations}
By the elliptic regularity~\cite{1970Babuska}, we have $\phi\in H^2(\cup\Omega^{\pm})$, and let $\bfw = \alpha^{-1}\nabla\phi \in \bfH^1(\cup\Omega^{\pm})$. 
On each element $K$, as $f\in H^1(\cup K^{\pm})$, we have $\bfw\in\bfH^1(\div;K^{\pm})$. 
At last,~\eqref{lem_IVE_deRham_H_exact_eq2_2} leads to $[\alpha\bfw\times\bfn]_{\Gamma^K}=\mathbf{ 0}$ and $[\bfw\cdot\bfn]_{\Gamma^K}=0$. %
\end{proof}

\begin{remark}
\label{rem_diff_complex}
The classic de Rham complex with higher smoothness is given by Ref.~\refcite{2011GiraultRaviart,2005TaiWinther}:
\begin{equation}
\label{diff_complex_eq1}
\begin{tikzcd}
\mathbb{R} \arrow[r,"\hookrightarrow"] 
& H^2(\Omega) \arrow[r, "\grad"] 
& \bfH^1(\curl;\Omega)  \arrow[r, "\curl"] 
& \bfH^1(\Omega)   \arrow[r, "\div"] 
& L^2(\Omega) \arrow[r] & 0 .
\end{tikzcd}
\end{equation}
We note that this sequence can be simply revised to be
\begin{equation}
\label{diff_complex_eq2}
\begin{tikzcd}
\mathbb{R} \arrow[r,"\hookrightarrow"] 
& H^2(\Omega) \arrow[r, "\grad"] 
& \bfH^1(\curl;\Omega)  \arrow[r, "\curl"] 
& \bfH^1(\div;\Omega)  \arrow[r, "\div"] 
& H^1(\Omega)   \arrow[r] & 0 .
\end{tikzcd}
\end{equation}
The revision can be understood immediately from 
$\ker(\div)\cap \bfH^1(\Omega) = \ker(\div)\cap \bfH^1(\div;\Omega)$.
The proposed new sequence \eqref{IVE_deRham_H_complex} is a further generalization of~\eqref{diff_complex_eq2} in which the jump information is incorporated. Finite element counterparts of \eqref{diff_complex_eq1} and \eqref{diff_complex_eq2} can be found in Ref. \refcite{ChenHuang2022Finite}. Virtual element discretization of \eqref{IVE_deRham_H_complex} will be discussed in Section \ref{sec:discretedeRham}.
\end{remark}

\section{Immersed Finite Element Spaces}
\label{sec:IFE}

In this section, we present $H^1$, $\bfH(\curl)$ and $\bfH(\div)$ IFE functions. 
The basis functions are some piecewise polynomials satisfying the jump conditions in certain sense 
to ensure the local approximation property. 
Particularly, the $H^1$ IFE space has been developed in Ref.~\refcite{2005KafafyLinLinWang}, 
but this is the first time that $\bfH(\curl)$ and $\bfH(\div)$ IFE spaces are systematically 
developed. Different from all the IFE spaces in literature, 
the spaces constructed here serve the purpose for approximation under the VEM framework, 
thus are not limited by the constraint that DoFs need to be imposed on their associated 
geometric objects. Instead, the DoFs are handled by the IVE spaces discussed 
in Section \ref{sec:IVE}. To facilitate a simple presentation, 
we shall focus on the case that elements are only cut by the interface once, 
i.e., each edge has at most one cutting point, 
which is a reasonable assumption employed by many works in the literature~\cite{2020GuoLin,2020GuoZhang,2005KafafyLinLinWang}. 
In fact, the interface elements may generally satisfy this assumption if the background mesh is sufficiently fine, 
namely the interface is locally flat enough. 
\LC{We also remark that IFE spaces can be constructed for more complicated interface element geometry violating this assumption,} 
which we leave to Appendix \ref{sec:appen} for this general case.

We need a linear approximation to the interface portion $\Gamma^K$, denoted by $\Gamma^K_h$.
For example, in~Ref. \refcite{2020GuoLin} $\Gamma^{K}_h$ is constructed as a plane passing through 
the three cutting points forming a triangle satisfying the maximum angle condition, 
see the middle plot in Figure \ref{fig:face_mesh} for an illustration. 
The following lemma essentially acknowledges this setting. 
Another widely-used linear approximation approach is to use 
$\Gamma_h:\phi_h(\bfx)=0$ with $\phi_h$ being the linearization of the sign-distance function 
$\phi$ of $\Gamma$ on the same mesh. These choices indicate that the interface can be well-resolved 
by a mesh that is sufficiently fine. %
\begin{lemma}
\label{lemma_interface_flat}
Suppose the mesh is sufficiently fine such that $h<h_0$ for a fixed threshold $h_0>0$, then on each interface element $K\in\mathcal{T}^i_h$, there exist constants $C_{\Gamma}$ independent of the interface location and mesh size $h_K$ such that for every point $X\in\Gamma^K$ with its orthogonal projection $X^{\bot}$ onto $\Gamma^{K}_h$,
\begin{align}
     \| X-X^{\bot} \| \le  C_{\Gamma} h_{K}^2. \label{lemma_interface_flat_eq1}  
\end{align}
\end{lemma} 
\RG{As the Maxwell equations generally have low regularity near the interface, we only consider the lowest 
order methods, and thus the $O(h^2)$ approximation geometrical accuracy in \eqref{lemma_interface_flat_eq1} is sufficient. 
If high-order methods are desired, one needs to either resolve the interface exactly by 
using the blending element techniques~\cite{1971Gordon} or approximate the interface by polynomials of
order at least $2p-1$~\cite{2010LiMelenkWohlmuthZou}. 
It is also worthwhile to mention a recent work of VEM \cite{2021DassiFumagalliLosapio} for 2D elements with curved edges.}

\RG{Let $\Gamma^K_h$ partition $K$ into $K^{\pm}_h$, and let $\alpha_h$ and $\beta_h$ be the piecewise constant functions 
whose jumps are now across $\Gamma^K_h$ instead of $\Gamma^K$; namely
\begin{equation}
  \label{discont_coef_alpbeta_h}
  \alpha_h=
  \begin{cases}
     \alpha^-   & \text{in} ~ K^-_h, \\
     \alpha^+   & \text{in} ~ K^+_h,
  \end{cases}
  ~~~~~~~~
  \beta_h=
  \begin{cases}
     \beta^-   & \text{in} ~ K^-_h, \\
     \beta^+   & \text{in} ~ K^+_h.
  \end{cases}
\end{equation}}
But, here we shall postpone the specific parameters $\alpha_h$ and $\beta_h$ in the PDEs to a later discussion, 
and focus on a generic piecewise constant function denoted as $c_h$ to present the IFE functions. 
In the following discussion, $\mathcal{P}_k(K)$ denotes the polynomial space with degree $k$ on $K$.
Let $\bar{\bfn}_K$ be the normal vector to $\Gamma^{K}_h$ that is approximately in the same direction with $\bfn_K$ to $\Gamma^K $. 
Define two piecewise constant vector spaces:
\begin{subequations}
\label{AB_space1}
\begin{align}
    &   \bfP^e_0(c_h;K) = \{ \bfc: \bfc^{\pm}= \bfc|_{K^{\pm}_h} \in [\mathcal{P}_0(K^{\pm}_h)]^3, ~ \bfc\in\bfH(\curl;K) , ~ c_h\bfc\in \bfH(\div;K) \},\\
    &  \bfP^f_0(c_h;K) = \{ \bfc: \bfc^{\pm}= \bfc|_{K^{\pm}_h} \in [\mathcal{P}_0(K^{\pm}_h)]^3, ~ \bfc\in\bfH(\div;K) , ~ c_h\bfc\in \bfH(\curl;K) \}.
\end{align}
\end{subequations}
The super scripts $e$ and $f$ emphasize that the two spaces are, respectively, in the edge and face spaces 
(1-form and 2-form in the language of differential forms); 
namely 
$$
\bfP^f_0(c_h;K) \subset \bfH^1(\div, c_h; K)\cap \ker(\div) ~~~ \text{and} ~~~ \bfP^e_0(c_h;K)\subset \bfH^1(\curl, c'_h, c_h; K)\cap \ker(\curl),
$$ 
where $c'_h$ may be any arbitrary piecewise constant due to the curl-free property of $\bfP^e_0(c_h;K)$. Hence, for the parameters $\alpha_h$ and $\beta_h$, there particularly hold
$$
\bfP^e_0(\beta_h;K)\subset \bfH^1(\curl, \alpha_h, \beta_h; K)~~~ \text{and} ~~~\bfP^f_0(\alpha_h;K) \subset \bfH^1(\div, \alpha_h; K) ,
$$
which give certain reasonable approximations to these two desired Sobolev spaces. 
We shall see that these two spaces are the fundamental ingredients to construct all the $H^1$, $\bfH(\curl)$ and $\bfH(\div)$ IFE functions, 
as well as to construct and project the IVE spaces.

Furthermore, from the definition, it is not hard to conclude the following relation
\begin{equation}
\label{AB_HodgeStar1}
 \bfP^e_0(c_h;K) \xrightleftarrows[~~~c^{-1}_h~~~]{~~c_h~~}  \bfP^f_0(c^{-1}_h;K).%
\end{equation}
Here the discontinuous coefficient $c_h$ can be viewed as a Hodge star operator. 
This perspective is the key for computing the projection of the proposed IVE spaces, see Section \ref{subsec:comment}.

In order to derive explicit formulas for the functions in the spaces \eqref{AB_space1}, 
we further let $\bar{\bft}^1_K$ and $\bar{\bft}^2_K$ be the two orthonormal tangential unit vectors to $\Gamma^{K}_h$, 
and denote the matrix $T_K=[\bar{\bfn}_K, \bar{\bft}^1_K,\bar{\bft}^2_K]$. Then, we define the matrices: 
\begin{equation}
\label{lem_ABeigen_eq1}
M^{f,c_h}_{K} = T_K 
\begin{bmatrix}
 1 & 0 & 0 \\ 0 & \tilde{c} & 0 \\ 0 & 0 & \tilde{c}  
\end{bmatrix}
(T_K)^{\top} ~~~\text{and} ~~~ 
M^{e,c_h}_{K} = T_K 
\begin{bmatrix}
 \tilde{c} & 0 & 0 \\ 0 & 1 & 0 \\ 0 & 0 & 1
\end{bmatrix}
 (T_K)^{\top},
\end{equation}
where $\tilde{c}=c^+_h/c^-_h$. Clearly, both $M^{f,c_h}_{K}$ and $M^{e,c_h}_{K}$ are symmetric and positive definite. Thus, the spaces $\bfP^e_0(c_h;K)$ and $\bfP^f_0(c_h;K)$ can be rewritten as
\begin{subequations}
\label{AB_space2}
\begin{align}
    &   \bfP^e_0(c_h;K) = \left \{  \bfc: \bfc^{\pm}= \bfc|_{K^{\pm}_h}, ~ \bfc^- = M^{e,c_h}_{K} \bfc^+, ~ \bfc^+\in [\mathcal{P}_0(K^{+}_h)]^3 \right \},\\
    &   \bfP^f_0(c_h;K) = \left \{  \bfc: \bfc^{\pm}= \bfc|_{K^{\pm}_h}, ~ \bfc^- = M^{f,c_h}_{K} \bfc^+, ~ \bfc^+\in [\mathcal{P}_0(K^{+}_h)]^3 \right \}.
\end{align}
\end{subequations}
$\bfP^f_0(c_h;K)$ and $\bfP^e_0(c_h;K)$ are subspaces of the piecewise constant vector functions (dimension 6). 
With the jump conditions as the constraints, it can be easily verified that the dimension of both $\bfP^f_0(c_h;K)$ and $\bfP^e_0(c_h;K)$ is $3$.

Now, we proceed to present the $H^1$, $\bfH(\curl)$ and $\bfH(\div)$ IFE functions. 
We consider $\bfP^f_0(a_h;K)$ and $\bfP^e_0(b_h;K)$, formed by two general positive piecewise constant functions $a_h$ and $b_h$. 
Then, all the $H^1$, $\bfH(\curl)$ and $\bfH(\div)$ IFE functions with the general parameters $a_h$ and $b_h$ have simple formulas presented in Table \ref{tab_IFE_func} where $\bfx_K$ is any point on $\Gamma^K$. 
One can directly verify that they belong to the corresponding Sobolev spaces and satisfy the associated jump conditions in the table. 
Note that the normal jump condition in the $\bfH(\curl)$ case and the tangential jump condition in the $\bfH(\div)$ case only hold at the single point $\bfx_K$ instead of the entire $\Gamma^K_h$. 
This does not violate the necessary continuities for these two spaces to be in $\bfH(\curl;K)$ and $\bfH(\div;K)$. 
Of course, different choices of $\bfx_K$ lead to different spaces. 
In addition, compared with the standard Lagrange, N\'ed\'elec, and Raviart-Thomas elements, 
the only difference for their IFE counterparts is to replace the constant vectors in $[\mathcal{P}_0(K)]^3$ by the vectors in $\bfP^f_0(a_h;K)$ and $\bfP^e_0(b_h;K)$, 
thus providing the necessary piecewise constant approximation to $\beta \nabla u$, $\alpha\curl \bfu$, and $\beta\bfu$ on interface elements, respectively.

\begin{table}[htp]
  \centering 
	\renewcommand{\arraystretch}{2}  
  \begin{tabular}{ c c c c } 
  \toprule
 IFE spaces & $S^n_h(b_h;K)$ & $\bfS^e_h(a_h,b_h;K)$ & $\bfS^f_h(a_h;K)$  \medskip  \\ 
Dimension & $4$ & $6$ & $4$
\medskip \\
\makecell{ Sobolev \\ spaces} & $H^1(K)$ & $\bfH(\curl;K)$ & $\bfH(\div;K)$ \medskip \\ 

\makecell{ Function \\ format } & \makecell{ $\bfb\cdot(\bfx-\bfx_K)+c$ \\ 
 $ \bfb \in \bfP^e_0(b_h;K),$\\
 $ c\in \mathcal{P}_0(K)$ } & \makecell{ $\bfa\times(\bfx-\bfx_K) + \bfb$ \\ 
 $ \bfa \in \bfP^f_0(a_h;K),$\\
 $ \bfb \in \bfP^e_0(b_h;K)$ } &  \makecell{ $c(\bfx-\bfx_K)+\bfa$\\ $c\in \mathcal{P}_0(K),$\\
 $ \bfa \in \bfP^f_0(a_h;K)$ } \medskip  \\ 
 \makecell{Jump \\ conditions} 
 & \makecell{ $[v_h]_{\Gamma^K_h} =0$ \\ $[b_h \nabla v_h\cdot \bar{\bfn}]_{\Gamma^K_h} = 0$ } 
 & \makecell{ $[\bfv_h\times \bar{\bfn}]_{\Gamma^K_h} = \mathbf{ 0}$ \\ $[a_h\curl\bfv_h\times \bar{\bfn}]_{\Gamma^K_h} = \mathbf{ 0}$ \\ $[b_h \bfv_h\cdot \bar{\bfn}]_{\bfx_K} = { 0}$ }
 & \makecell{ $[\bfv_h\cdot \bar{\bfn}]_{\Gamma^K_h} = { 0}$ \\ $[a_h\bfv_h\times \bar{\bfn}]_{\bfx_K} = \mathbf{ 0}$ \\ $[\div \bfv_h ]_{\Gamma^K_h} = { 0}$ } \medskip \\
 \bottomrule
\end{tabular}
  \caption{IFE spaces, their dimensions, their function format, the corresponding jump conditions and the Sobolev spaces to which they belong, where $\bfx_K$ is any point at $\Gamma^K_h$.}
  \label{tab_IFE_func}
\end{table}

In addition, on each interface element, these spaces admit a local exact sequence established in the following lemma.
\begin{lemma}
\label{lem_loc_exact}
The following sequence is a complex and is exact:
\begin{equation}
\label{local_DR_IFE}
\adjustbox{scale=0.95,center}{%
\begin{tikzcd}[column sep=scriptsize]
\mathbb{R} \arrow[r,"\hookrightarrow"] 
&[1em]  S^{n}_h(b_h;K) \arrow[r, "\grad"] 
&[1em]  \bfS^e_h(a_h,b_h;K)  \arrow[r, "\curl"] 
&[1em]  \bfS^{f}_h(a_h;K)   \arrow[r, "\div"] 
&[1em]  \mathcal{P}_0(K) \arrow[r] & 0 .
\end{tikzcd}
}
\end{equation}
Furthermore, the constant vector spaces $\bfP^f_0(a_h;K)$ and $\bfP^e_0(b_h;K)$, respectively, are the $\curl$-free and $\div$-free subspaces of $\bfS^e_h(a_h,b_h;K)$ and $\bfS^f_h(a_h;K)$:
\begin{subequations}
\label{AB_space3}
\begin{align}
    &  \bfP^e_0(b_h;K) = \grad S^n_h(b_h;K) = \bfS^e_h(a_h,b_h;K)\cap\ker(\curl), \label{AB_space3_1}   \\
    &  \bfP^f_0(a_h;K) = \curl \bfS^e_h(a_h,b_h;K) = \bfS^f_h(a_h;K)\cap\ker(\div). \label{AB_space3_2}
\end{align}
\end{subequations}
\end{lemma}
\begin{proof}
It can be verified directly.
\end{proof}

\begin{remark}
In computation, the IVE functions and their curls are projected to the constant spaces $\bfP^e_0(\beta_h;K)$ and $\bfP^f_0(\alpha_h;K)$.
To ensure an optimal first order convergence, the projections need not be to the full IFE spaces $S^n_h(\beta_h; K)$, $\bfS^e_h(\alpha_h,\beta_h;K)$ and $\bfS^f_h(\beta_h;K)$. 
But these spaces will be useful in the computation procedure of the projections.
\end{remark}

\RG{To end this section, we show trace inequalities for piecewise constant IFE functions. The key is the generic constant is independent of the location of the interface.
\begin{lemma}
\label{lem_trace_inequa}
Given each interface element $K$ and one of its face $F$, for every $\bfc \in \bfP^e_0(b_h;K)$ or $\bfc \in \bfP^f_0(a_h;K)$, there holds that
\begin{equation}
\label{lem_trace_inequa_eq0}
\| \bfc \|_{L^2(F)}\lesssim h^{-1/2}_K \| \bfc \|_{L^2(K)},
\end{equation}
where the generic constant is independent of the location of the interface but depends on $a_h$ or $b_h$.
\end{lemma}
\begin{proof}
By \eqref{lem_ABeigen_eq1} and \eqref{AB_space2}, we know that there is a matrix $M$ with $\| M \|_{\infty} \lesssim 1$ such that
\begin{equation}
\label{lem_trace_inequa_eq1}
\| \bfc^- \| = \| M \bfc^+ \|\lesssim \| \bfc^+\|.
\end{equation}
Without loss of generality, we only consider the case that $F$ intersects with the interface and assume $F$ is cut into $F^{\pm}_h$. 
By geometry, it is not hard to see that either there is a pyramid $P\subseteq K^+_h$ which has the base $F^+_h$ and height $\mathcal{O}(h_K)$ or this is true for the ``-" piece. 
Again, without loss of generality, we assume the former case is true. Then, the standard trace inequality on $P$ simply implies
\begin{equation}
\label{lem_trace_inequa_eq2}
\| \bfc^+ \|_{L^2(F^+_h)} \lesssim h^{-1/2}_K \| \bfc^+ \|_{L^2(K^+_h)} .
\end{equation}
As for the ``$-$" piece, we apply the trace inequality on the entire element $K$ with \eqref{lem_trace_inequa_eq1} to obtain
\begin{equation}
\label{lem_trace_inequa_eq3}
\| \bfc^- \|_{L^2(F^-_h)} \lesssim h^{-1/2}_K \| \bfc^- \|_{L^2(K)} \lesssim h^{-1/2}_K ( \| \bfc^- \|_{L^2(K^-_h)} + \| M \bfc^+ \|_{L^2(K^+_h)} )  \lesssim h^{-1/2}_K \| \bfc \|_{L^2(K)}.
\end{equation}
Combining \eqref{lem_trace_inequa_eq2} and \eqref{lem_trace_inequa_eq3}, we have the desired estimate.
\end{proof}
}

\section{Immersed Virtual Element Spaces}
\label{sec:IVE}

It is generally not possible to construct conforming piecewise polynomial spaces to the Sobolev spaces in \eqref{tildeHspace}. 
Traditionally, Lagrange, N\'ed\'elec and Raviart-Thomas elements are conforming to $H^1$, $\bfH(\curl)$ and $\bfH(\div)$ spaces, 
yet they cannot provide sufficient approximation when a mesh-cutting interface is present. 
The IFE spaces introduced above can capture the jump information, but at the cost of losing conformity. 
In this section, we construct immersed virtual element (IVE) spaces based on solutions to local interface problems. 
IVE spaces can be both conforming and satisfy interface conditions perfectly. 
\LC{For a non-interface element $K$, the local finite element space is simply defined as the linear polynomial space $\mathcal{P}_1(K)$, 
the lowest order N\'ed\'elec space $\mathcal{ND}_0(K)$~\cite{Nedelec1980}, and the lowest order Raviart-Thomas space $\mathcal{RT}_0(K)$~\cite{Raviart.P;Thomas.J1977}.}

Given an interface element $K\in\mathcal{T}^i_h$, we let $\mathcal{N}_K$ and $\mathcal{E}_K$ be the collection of all the nodes and edges in the triangulation of $\partial K$. 
Note that the nodes include the vertices in the background mesh and cutting points, and the edges include all the sub-edges cut by the interface and newly-added edges. 
Then $\mathcal{F}_K$ denotes the resulting triangular faces.
Given a $T$ which may be a cube, square, tetrahedron or triangle, we let $\mathcal{ND}_0(T)$ and $\mathcal{RT}_0(T)$ be the first family of N\'ed\'elec polynomial space and the Raviart-Thomas polynomial space of the lowest degree on $T$. 
The Lagrange space is simply the first-degree polynomial space $\mathcal{P}_1(T)$.

Next, we also need two weighted projections onto the piecewise constant vector spaces $\bfP^f_0(c_h;K)$ and $\bfP^e_0(c_h;K)$ which will be used in the definition of the IVE spaces as well as the computation:
\begin{subequations}
\label{weight_proj}
\begin{align}
    &  \bfPi^{e,c_h}_K: \bfH(\curl;K)\rightarrow \bfP^e_0(c_h;K), ~~ \text{satisfying} \label{B_proj_1} \\
     & ~~~~~~~ \int_K c_h \bfPi^{e,c_h}_K \bfv_h\cdot \bfp_h \dd \bfx = \int_K c_h  \bfv_h\cdot \bfp_h \dd \bfx, ~~~ \forall \bfp_h\in\bfP^e_0(c_h;K), \nonumber \\
    &  \bfPi^{f,c_h}_K: \bfH(\div;K) \rightarrow \bfP^f_0(c_h;K), ~~ \text{satisfying}  \label{A_proj_1} \\
    & ~~~~~~~  \int_K c_h \bfPi^{f,c_h}_K \bfv_h\cdot \bfp_h \dd \bfx = \int_K c_h  \bfv_h\cdot \bfp_h \dd \bfx, ~~~ \forall \bfp_h\in\bfP^f_0(c_h;K). \nonumber
\end{align}
\end{subequations}
The super-scripts, $e$ and $f$, still emphasize the distinct Sobolev spaces, i.e., the images of $\bfPi^{e,c_h}_K$ and $\bfPi^{f,c_h}_K$ belong to $\bfH(\curl;K)$ and $\bfH(\div;K)$, respectively.

\subsection{The $H^1$ IVE Space}
We first consider the $H^1$ case. Given the boundary triangulation, we define the boundary function space:
\begin{equation}
\label{H1_IVEM_bd}
\mathcal{B}^n_h(\partial K) = \{ v_h\in C(\partial K) ~:~ v_h|_T\in\mathcal{P}_1(T), ~ \forall T \in \mathcal{F}_K\}. 
\end{equation}
Then, on an interface element $K$, the $H^1$ IVE space involving the discontinuous coefficient $\beta$ is defined as
\begin{equation}
\begin{split}
\label{H1_IVEM}
V^n_h(K) = \big\{ v_h~:~ 
\beta \nabla v_h \in \bfH(\div;K),
~ \nabla \cdot (\beta \nabla v_h) = 0, 
~ v_h|_{\partial K}\in \mathcal{B}^n_h(\partial K) \big\}.
 \end{split}
\end{equation}
Clearly $V^n_h(K)\subseteq H^1(K)$. On the boundary $\partial K$ we use the continuous $\mathcal{P}_1$ finite element space on the body-fitted surface triangulation. 
In the interior we use $\beta$-harmonic extension so that the shape functions satisfy the jump conditions on the interface. 

The property of the nodal DoFs is given by the following lemma.

\begin{lemma}
The space $V^n_h(K)$ has nodal DoFs $\{v_h(\bfz)$, $\bfz\in\mathcal{N}_K\}$.
\end{lemma}
\begin{proof}
First, $v_h\in V^n_h(K)$ is uniquely determined by the boundary condition in $\mathcal{B}^n_h(\partial K)$. 
The space $\mathcal{B}^n_h(\partial K)$ has the nodal DoFs associated with the nodes in $\mathcal{N}_K$. 
So functions in $V^n_h(K)$ are uniquely determined by their nodal values. 
\end{proof}

On any non-interface element, the standard Lagrange FE space, i.e., $\mathcal{P}_1(K)$, is used. 
Thus, with the nodal DoFs, we are able to define the global $H^1$-conforming IVE space as
\begin{equation}
\label{virtual_space_glob}
V^n_h= \{ v_h\in H^1_0(\Omega): v_h|_K\in V^n_h(K), ~ \forall K\in\mathcal{T}^i_h, ~~\text{and} ~~ v_h|_K\in \mathcal{P}_1(K), ~ \forall K\in\mathcal{T}^n_h  \}.
\end{equation}
Functions in $V^n_h$ are piecewise linear on the element boundary triangulation and in general non-polynomial inside interface element, 
which is the key to capture both the jump conditions and conformity.

Similar to the standard VEMs, the function values in the interior are not needed, 
and projections to certain spaces with approximation properties are computed instead. 
In the following paragraph, we show how to compute $\bfPi^{e,\beta_h}_K\nabla v_h$ for $v_h \in V^n_h(K)$. 
For every $\bfp_h\in \bfP^e_0(\beta_h;K)$, by \eqref{AB_HodgeStar1} there holds $\beta_h\bfp_h \in \bfP^f_0(\beta^{-1}_h;K) \subset  \bfH(\div;K) \cap\ker(\div)$. 
Then, the integration by parts shows
\begin{equation}
\begin{split}
\label{H1_proj1}
\int_K \beta_h \bfPi^{e,\beta_h}_K\nabla v_h \cdot  \bfp_h \dd \bfx =  \int_K  \nabla v_h \cdot  (\beta_h\bfp_h) \dd \bfx   =   \int_{\partial K} v_h (\beta_h\bfp_h\cdot\bfn) \dd s,
\end{split}
\end{equation}
of which the right-hand side is computable. The $L^2$ projection of $v_h$ to $\mathcal{P}_0(K)$ is not computable by the definition of the current space. 
However, this is not needed as $v_h$ itself can be approximated by the formula in Table \ref{tab_IFE_func} using the gradient obtained from \eqref{H1_proj1}. 
Denote this weighted $H^1$ projection by $\bar{v}_h$, the constant $c$ in Table \ref{tab_IFE_func} can be chosen such that $\int_{\partial K} \bar{v}_h = \int_{\partial K} v_h$. 
This constraint gives compactness, thus a sufficient approximation for computing the right-hand side term to guarantee the first order optimal convergence. %

\subsection{The $\bfH(\curl)$ IVE Space}

For the $\bfH(\curl)$ case, the boundary space is defined as
\begin{equation}
\label{Hcurl_VEM_B2}
\mathcal{B}^e_h(\partial K) = \{ \bfv_h : ~  \bfv_h|_{T} \in \mathcal{ND}_0(T), ~\forall T\in  \mathcal{F}_K, ~ (\bfv_h\cdot\bft)|_e ~ \text{is continuous on each} ~ e\in\mathcal{E}_K\}.
\end{equation}
Each $\mathcal{ND}_0(T)$ contains the 2D vector polynomials tangentially defined on the planar triangle $T$.
By formulations from trace finite elements on triangulated surfaces~\cite{OlshanskiiReuskenEtAl2009finite}, 
$\mathcal{B}^e_h(\partial K)$ is a well-defined finite element space on its own, and has the DoFs of $\int_e\bfv_h\cdot\bft \dd s$, $e\in\mathcal{E}_K$.
With this boundary space, we first introduce an auxiliary $\bfH(\curl)$ IVE space:
\begin{equation}
\begin{split}
\label{Hcurl_IVE_local}
\widetilde{\bfV}^e_h(K) = \{ &\bfv_h\in \; \bfH(\curl;K):  \beta\bfv_h \in \bfH(\div;K),~ \div(\beta\bfv_h)=0, ~  
\\ 
& \alpha\curl\,\bfv_h\in \bfH(\curl;K), ~ \curl\alpha\curl\,\bfv_h \in \bfP^f_0(\beta^{-1}_h;K), ~ \bfv^{\tau}_h \in \mathcal{B}^e_h(\partial K)  \}.
\end{split}
\end{equation}
Here, $\alpha$ and $\beta$ can be also understood as Hodge star operators exactly mimicking the second and the third vertical mappings in the desired diagram \eqref{Hodge_H_local} between different Sobolev spaces. %

The following lemma gives the well-posedness and DoFs of $\widetilde{\bfV}^e_h(K)$.
\begin{lemma}
\label{lem_Hcurl_IVE_DoF}
$\widetilde{\bfV}^e_h(K)$ is unisolvent with respect to the DoFs $\{\int_e\bfv_h\cdot\bft \dd s,~ e\in\mathcal{E}_K\}$ and $\{\int_K \beta_h  \bfv_h\cdot \bfp_h \dd \bfx,~ \bfp_h \in  \bfP^e_0(\beta_h;K)\}$.
\end{lemma}
\begin{proof}
Let $\bfq_h\in \bfP^f_0(\beta^{-1}_h;K)$ and $\bfr_h\in\mathcal{B}^e_h(\partial K)$ be some arbitrary data functions. 
Let us formulate the following local interface problem arising from the definition in \eqref{Hcurl_IVE_local}:
\begin{equation}
\begin{split}
\label{lem_Hcurl_IVE_DoF_eq1}
\curl\alpha\curl\,\bfv_h = \bfq_h, ~~ \div(\beta\bfv_h) = 0, ~~ \text{in} ~ K, ~~~ \bfv_h^{\tau}= \bfr_h ~~ \text{on} ~ \partial K.
\end{split}
\end{equation}
Note that the solutions $\bfv_h$ of \eqref{lem_Hcurl_IVE_DoF_eq1} form the space $\widetilde{\bfV}^e_h(K)$.
\RG{The well-posedness of \eqref{lem_Hcurl_IVE_DoF_eq1} is given by Lemma \ref{lem_curlcurl}. 
It implies that the dimension of the solution space is equal to the (finite) dimension of the space of possible data functions: 
the boundary data function $\bfr_h$ which is uniquely determined by $\{\int_e\bfr_h\cdot\bft \dd s, ~ e\in\mathcal{E}_K\}$ and
the right-hand side $\bfq_h\in \bfP^f_0(\beta^{-1}_h;K)$ which has dimension 3.} 
Therefore, the dimension of solution space matches the DoFs count $3+|\mathcal{E}_K|$ for this virtual space.

In the rest of the proof, it needs to be established that the given moments on edges and element interior are indeed DoFs. 
To this end, it suffices to show that a function with vanishing DoFs is trivial in this space. 
Noticing $\curl\alpha\curl\,\bfv_h \in \bfP^f_0(\beta^{-1}_h;K) = \beta_h \bfP^e_0(\beta_h;K)$ by \eqref{AB_HodgeStar1}, 
thus from the interior DoFs in Lemma \ref{lem_Hcurl_IVE_DoF} we have
\begin{equation}
\label{lem_Hcurl_IVE_DoF_eq2}
0 = \int_K \bfv_h\cdot \curl\alpha\curl\,\bfv_h \dd \bfx = \int_K \alpha \curl\,\bfv_h\cdot \curl\,\bfv_h \dd \bfx - \int_{\partial K} \bfv^{\tau}_h\cdot (\alpha \curl\bfv_h\times \bfn) \dd s,
\end{equation}
where we have used integration by parts in the second equality. 
As the edge moments are zero, we know $\bfv^{\tau}_h =0 $. So, we have $ \int_K \alpha \curl\,\bfv_h\cdot \curl\,\bfv_h \dd \bfx = 0$ 
which implies $\curl\bfv_h = \mathbf{ 0}$ as $\alpha$ is positive. 
In addition, by $\div(\beta \bfv_h) = 0$ and the vanishing trace, we derive from Lemma \ref{lem_divcurl2} that $\bfv_h = \mathbf{ 0}$.
\end{proof}

Similar to the classical VEM, both $\bfv_h$ and $\curl \bfv_h$ are not computable. 
But we shall see that their weighted projections to the IFE spaces are computable.
 We first address the projection of $\curl \bfv_h$ which will be then used to develop a new IVE space as a subspace of $\widetilde{\bfV}^e_h$ that only has the edge DoFs.

By the Hodge star property \eqref{AB_HodgeStar1}, we argue that $\bfPi^{f,a_h}_K\curl\bfv_h$ defined in \eqref{A_proj_1} is computable for any positive $a_h$. 
In particular, for each $\bfp_h\in\bfP^f_0(a_h;K)$ we know $a_h\bfp_h \in \bfH (\curl; K)\cap \ker(\curl)$. 
Then, applying the projection $\bfPi^{f,a_h}_K$ to $\curl\bfv_h$ with the integration by parts, we obtain
\begin{equation}
\begin{split}
\label{A_proj_2}
& \int_K a_h \bfPi^{f,a_h}_K \curl\bfv_h\cdot \bfp_h \dd \bfx =  \int_K  \curl\bfv_h\cdot (a_h\bfp_h) \dd \bfx = \int_{\partial K}  (\bfv_h\times \bfn) (a_h \bfp_h)^{\tau} \dd s.
\end{split}
\end{equation}
The right-hand side above is computable through only the edge DoFs $\int_e\bfv_h\cdot\bft \dd s$, $e\in \mathcal{E}_K$, since the boundary triangulation is known.
In computation, $\bfPi^{f,\alpha_h}_K \curl\bfv_h$ provides a sufficient approximation order on elements intersecting the interface.

Now, we let $\bfS^e_h(a_h,0;K)$ be the subspace of $\bfS^e_h(a_h,b_h;K)$ that has the fixed $\bfb=\mathbf{ 0}$ in Table \ref{tab_IFE_func}. 
Namely, the functions in $\bfS^e_h(a_h,0;K)$ are just $\bfw_h = \bfa\times(\bfx-\bfx_K)$, $\bfa \in \bfP^f_0(a_h;K)$, and thus the space only has the dimension 3.
Then, based on the auxiliary space $\widetilde{\bfV}^e_h(K)$, we introduce its subspace:
\begin{equation}
\begin{split}
\label{Hcurl_IVE_local_2_new}
\bfV^e_h(K) = \{ \bfv_h \in \widetilde{\bfV}^e_h(K): &\int_K  \curl \bfv_h\cdot \bfw_h \dd \bfx = 0, ~ \forall \bfw_h\in \bfS^e_h(\beta^{-1}_h,0;K) \}. %
\end{split}
\end{equation}
Clearly, $\bfV^e_h(K) \subset \widetilde{\bfV}^e_h(K)$, and the following lemma ensures that only the edge DoFs are needed for the unisolvency of this subspace due to the extra constraint. 
\begin{lemma}
\label{lem_Hcurl_IVE_DoF_2}
${\bfV}^e_h(K)$ is unisolvent with respect to the edge DoFs $\{\int_e\bfv_h\cdot\bft \dd s$, $e\in\mathcal{E}_K\}$.
\end{lemma}
\begin{proof}
It suffices to show that this extra condition \eqref{Hcurl_IVE_local_2_new}, in fact, makes the interior DoFs $\int_K\beta_h \bfv_h\cdot\bfp_h \dd \bfx$, $ \bfp_h\in  \bfP^e_0(\beta_h;K)$ fixed thus not degrees of freedom anymore. To see this, for each $\bfp_h\in \bfP^e_0(\beta_h;K)$, we have $\beta_h\bfp_h\in\bfP^f_0(\beta^{-1}_h;K)$ \LC{and $\div \beta_h\bfp_h = 0$}. Therefore, by the local exact sequence \eqref{local_DR_IFE} and \eqref{AB_HodgeStar1}, there exists $\bfw_h\in \bfS^e_h(\beta^{-1}_h,0;K)$
such that $\curl \bfw_h = \beta_h \bfp_h$. Then, the integration by parts shows
\begin{equation}
\label{Hcurl_IVE_local_4}
\int_K \beta_h \bfv_h\cdot  \bfp_h \dd \bfx = \int_K \bfv_h\cdot  \curl\bfw_h \dd \bfx =  \underbrace{ \int_K   \curl \bfv_h\cdot \bfw_h \dd \bfx }_{({\rm I})} -  \underbrace{ \int_{\partial K} \bfv^{\tau}_h\cdot (\bfw_h\times\bfn) \dd s }_{({\rm II})}.
\end{equation}
Note that $({\rm I}) =0$ by the extra condition in the definition. 
For $({\rm II})$, $\bfv^{\tau}_h\in \mathcal{B}^e_h(\partial K)$ is also solely determined by the edge DoFs. Therefore, under the extra constraint, the right-hand side of \eqref{Hcurl_IVE_local_4} is computable for every pair of $\bfp_h$ and $\bfw_h$, as long as the edge DoFs $\{\int_e \bfv_h\cdot\bft \dd s$, $e\in\mathcal{E}_K\}$, are given.
\end{proof}

For the new space $\bfV^e_h(K)$, the identity \eqref{Hcurl_IVE_local_4} also gives a simple formula for computing $\bfPi^{e,\beta_h}_K\bfv_h$:
\begin{equation}
\begin{split}
\label{B_proj_curl_new}
\int_K  \beta_h \bfPi^{e,\beta_h}_K \bfv_h \cdot  \bfp_h \dd \bfx  %
= \int_K  \beta_h \bfv_h \cdot  \bfp_h \dd \bfx
 =  - \int_{\partial K} \bfv^{\tau}_h\cdot (\bfw_h\times\bfn) \dd s ,
\end{split}
\end{equation}
with $\bfw_h=(\beta_h\bfp_h/2)\times(\bfx-\bfx_K)$ given by Table \ref{tab_IFE_func}.%

With the edge DoFs, we are able to define a global $\bfH(\curl)$-conforming space that uses \eqref{Hcurl_IVE_local_2} on interface elements and standard N\'ed\'elec elements on non-interface elements:
\begin{equation}
\label{Hcurl_IVE_global}
\bfV^e_h = \{ \bfv_h \in \bfH(\curl): \bfv_h|_K\in \bfV^e_h(K), ~ K\in \mathcal{T}^i_h ~ \text{and} ~ \bfv_h|_K\in \mathcal{ND}_0(K), ~ K\in \mathcal{T}^n_h \}.
\end{equation}

\subsection{The $\bfH(\div)$ IVE Space}

For the $\bfH(\div)$ case, the boundary function space is defined as
\RG{\begin{equation}
\label{Hdiv_VEM_B}
\mathcal{B}^f_h(\partial K) = \{ v_h : ~  v_h|_T \in \mathcal{P}_0(T), ~  \forall T\in\mathcal{F}_K \}.
\end{equation}}
Similar to the $\bfH(\curl)$ space, let us define the IVE space in the following:
\begin{equation}
\begin{split}
\label{Hdiv_IVE_local}
\bfV^f_h(K) = \{ \bfv_h\in H(\div;K): & \div(\bfv_h)\in \mathcal{P}_0(K), ~~ \bfv_h\cdot\bfn |_{\partial K} \in \mathcal{B}^f_h(\partial K) , \\
&  ~ \alpha\bfv_h\in \bfH(\curl;K), ~ \curl(\alpha\bfv_h) = \mathbf{ 0} \}.
\end{split}
\end{equation}
Again, the discontinuous coefficient $\alpha$ serves as a Hodge star operator that is to mimic the third vertical mapping in Diagram \ref{Hodge_H_local}. 
The well-posedness and DoFs of this space are given by the following lemma.
\begin{lemma}
\label{lem_Hdiv_IVE_DoF}
$\bfV^f_h(K)$ is unisolvent with respect to the DoFs $\{\bfv_h|_F\cdot\bfn_F $, $F\in\mathcal{F}_K\}$.
\end{lemma}
\begin{proof}
We let $c_h\in \mathcal{P}_0(K)$ and $r_h\in \mathcal{B}^f_h(\partial K)$ be some arbitrary data functions for the definition \eqref{Hdiv_IVE_local} satisfying the compatibility condition
\begin{equation}
\label{lem_Hdiv_IVE_DoF_eq1}
\int_K c_h \dd x = \int_{\partial K} r_h \dd s.
\end{equation}
Then, we consider the following $\div$-$\curl$ system: 
\begin{equation}
\begin{split}
\label{lem_Hdiv_IVE_DoF_eq2}
\div(\bfv_h) = c_h, ~~~ \curl (\alpha\bfv_h) = \mathbf{ 0} ~ ~ \text{in} ~ K, ~~~ \text{and} ~~~ \bfv_h\cdot\bfn = r_h~ \text{on} ~ \partial K
\end{split}
\end{equation}
whose solutions $\bfv_h$ form the space ${\bfV}^f_h(K)$.
Due to the compatibility condition \eqref{lem_Hdiv_IVE_DoF_eq1}, by Lemma \ref{lem_divcurl} the system in \eqref{lem_Hdiv_IVE_DoF_eq2} is well-posed 
and admits a unique solution $\bfv_h\in \bfH(\div;K)$ with $\alpha \bfv_h\in \bfH(\curl;K)$. 
The dimension of ${\bfV}^f_h(K)$ is just the dimension of space of the independent data functions, $|\mathcal{F}_K|$, 
where we note that $c_h$ does not count as it is determined by $r_h$ from \eqref{lem_Hdiv_IVE_DoF_eq1}.

Next, to show that the given moments are indeed DoFs, we suppose they all vanish, and \LC{ thus $c_h=r_h=0$ in \eqref{lem_Hdiv_IVE_DoF_eq2}} which immediately implies $\bfv_h=\mathbf{ 0}$ by Lemma \ref{lem_divcurl}.
\end{proof}

Now, we discuss \RG{how to compute projections of the proposed $\bfH(\div)$ IVE space}. 
The identity \eqref{lem_Hdiv_IVE_DoF_eq1}, in fact, yields a formula for computing $\div(\bfv_h)$:
\begin{equation}
\label{Hdiv_proj_1}
\div(\bfv_h) = |K|^{-1} \int_{\partial K} \bfv_h\cdot\bfn \dd s
\end{equation}
which is computable. As for $\bfv_h$ itself, we then argue that $\bfPi^{f,a_h}_K\bfv_h$ defined in \eqref{A_proj_1} is always computable for any given positive $a_h$. 
Similar to \eqref{A_proj_2}, given each $\bfp_h\in\bfP^f_0(a_h;K)$, we have $a_h\bfp_h\in \bfP^e_0(a^{-1}_h;K)$. 
Then, there exists $\psi_h=(a_h\bfp_h)\cdot(\bfx - \bfx_K)\in S^n_h(a^{-1}_h;K)$ such that $\nabla \psi_h = a_h\bfp_h$, and thus we can derive
\begin{equation}
\begin{aligned}
\label{Hdiv_proj_2}
 \int_K a_h \bfPi^{f,a_h}_K \bfv_h\cdot \bfp_h \dd \bfx  &=  \int_K a_h  \bfv_h\cdot \bfp_h \dd \bfx = \int_K   \bfv_h\cdot \nabla \psi_h \dd \bfx 
\\
  &= - \int_K\div(\bfv_h)\psi_h \dd \bfx + \int_{\partial K}   \bfv_h\cdot\bfn \psi_h \dd s
\end{aligned}
\end{equation}
which is computable. %
In computation, $a_h$ in this space is set to $\alpha_h$ \LC{as $\alpha_h \bfp_h\in \bfH(\curl)$}.

Thanks to the face DoFs, we can define a global $\bfH(\div)$-conforming space:
\begin{equation}
\label{Hdiv_IVE_global}
\bfV^f_h = \{ \bfv_h\in \bfH(\div;\Omega): \bfv_h\in\bfV^f_h(K), ~ K\in \mathcal{T}^i_h ~ \text{and} ~ \bfv_h|_K\in \mathcal{RT}_0(K), ~ K\in \mathcal{T}^n_h \}.
\end{equation}

\subsection{Some Comments and Alternative Definitions}
\label{subsec:comment}

The proposed $\bfH(\curl)$ and $\bfH(\div)$ IVE spaces above are exactly the extension of the classical virtual spaces in literature~
\cite{Beirao-da-Veiga;Brezzi;Cangiani;Manzini:2013principles,2017VeigaBrezziDassiMarini,2016VeigaBrezziMarini,BeiraodaVeigaBrezziDassiEtAl2018Lowest,BEIRAODAVEIGA2021} to the case of discontinuous coefficients. 
The modification also includes the source terms and the boundary conditions of the local interface problems, by which the weighted projections are computable. 

\RG{The face triangulation and the associated standard FE space is critical for the 
appropriate definition of the IVE spaces. Note that the 2D IVE spaces in Ref.
\refcite{2021CaoChenGuoIVEM} may not be developed on faces intersection with the interface, 
as jump conditions on faces are quite obscure. Take the $H^1$ interface problem as an example.
Given a face $F$, the desired jump information $[\beta\nabla_Fv\cdot\hat{\mathbf{n}}_{\Gamma,F}]_{\Gamma}$ 
cannot be derived from $[v]_{\Gamma}$ and $[\beta\nabla v\cdot\mathbf{n}]_{\Gamma}$, where $\nabla_F$ is the surface 
gradient on $F$ and $\hat{\mathbf{n}}_{\Gamma,F}$ is the normal vector to $\Gamma\cap F$ but parallel to $F$.
We refer readers to the derivation of the jump conditions on interface edges in the 2D case~\cite{2010ChuGrahamHou} 
that has to introduce the derivative along the normal direction of element boundary and thus adds much more complexity. 
Instead, we use well-defined finite element spaces for an interface-fitted triangulation on the boundary faces, 
which actually makes the theory and computation much simpler.
In addition, this approach can also provide sufficient approximation capabilities and keep the DoFs. 
Furthermore, in the next section, we shall see that it can also benefit implementation through the proposed data structure.}

\RG{One may note that the definition of the IVE spaces above do not rely on the assumption that the interface only cuts elements once. 
In fact, all the local problems are automatically well-posed for almost arbitrary interface element configuration, as long as the face triangulation exists. 
For example, they can be used on elements shown in the right plot of Figure \ref{fig:face_mesh} (a 2D illustration). 
This very feature together with the IFE spaces in Appendix \ref{sec:appen} makes the proposed method much more flexible than the traditional IFE methods in the literature.}

Next, let us summarize the relationship between the involved spaces and weights in the computation of projections, which may be unified as 
\begin{equation}
\begin{split}
\label{unifyProj}
&\int_K c_h \bfPi^{s,c_h}_K \bfv_h\cdot \bfp_h \dd \bfx = \int_K  \bfv_h \cdot \underbrace{(c_h \bfp_h)}_{\in \bfP^{s'}_0(c^{-1}_h;K)} \dd \bfx,\\
 & \text{where} ~ 
\begin{cases}
      &  (s,s',c_h) = (e,f,\beta_h), ~ \text{if} ~ (\bfv_h,\bfp_h)\in\bfH(\curl;K)\times\bfP^e_0(c_h;K), \\
      &  (s,s',c_h) = (f,e,\alpha_h), ~ \text{if} ~ (\bfv_h,\bfp_h)\in\bfH(\div;K)\times\bfP^f_0(c_h;K).
\end{cases}
\end{split}
\end{equation}
By the language of differential forms, in order for the wedge product of a $k$-form and $l$-form to be scalar, there needs $k+l=3$ in the 3D case.
Note that $\bfv_h$ and $\bfp_h$ both belong to the $k$-form, $k=1,2$, so $c_h$ acts as a Hodge star operator \eqref{AB_HodgeStar1} mapping $\bfp_h$ to the $(3-k)$-form for the desired wedge product. 
Here, the value of $c_h=\alpha_h$ or $\beta_h$ depends on the $\bfH(\curl)$ or $\bfH(\div)$ spaces matching the underlying Maxwell's equations.

At last, we provide an alternative definition of the $\bfH(\curl)$ IVE spaces being a different subspace of $\widetilde{\bfV}^e_h(K)$, 
which has some nice mathematical properties. The key is also to impose suitable conditions to assign the interior DoFs.

For the $\bfH(\curl)$ space, we may consider
\begin{equation}
\begin{split}
\label{Hcurl_IVE_local_2}
\bfV^e_h(K) = \{ \bfv_h \in \widetilde{\bfV}^e_h(K): &\int_K  \curl \bfv_h\cdot \bfw_h \dd \bfx = \int_K  \bfPi^{f,\beta_h}_K \curl \bfv_h\cdot \bfw_h \dd \bfx, \\
& \forall \bfw_h \in \bfS^e_h(\beta^{-1}_h,0;K) \}.
\end{split}
\end{equation}
Then, the interior DoFs can be determined also through integration by parts:
\begin{equation}
\begin{split}
\label{B_proj_curl}
\int_K  \beta_h \bfv_h \cdot  \bfp_h \dd \bfx  %
& = \int_K  \curl \bfv_h\cdot \bfw_h \dd \bfx - \int_{\partial K} \bfv^{\tau}_h\cdot (\bfw_h\times\bfn) \dd s  \\
& = \int_K \bfPi^{f,\beta_h}_K  \curl \bfv_h\cdot \bfw_h \dd \bfx - \int_{\partial K} \bfv^{\tau}_h\cdot (\bfw_h\times\bfn) \dd s, ~~~ \forall \bfp_h\in \bfP^e_0(\beta_h;K),
\end{split}
\end{equation}
where $\bfw_h=(\beta_h\bfp_h/2)\times(\bfx-\bfx_K)$ from Table \ref{tab_IFE_func} 
makes the space $\bfV^e_h(K)$ only have the edge DoFs. 
The identity \eqref{B_proj_curl} also gives the formula for computing $\bfPi^{e,\beta_h}_K\bfv_h$. 
But, compared with \eqref{B_proj_curl_new}, \eqref{B_proj_curl} needs to compute the extra term $\int_K \bfPi^{f,\beta_h}_K  \curl \bfv_h\cdot \bfw_h \dd \bfx$, which is slightly more expensive.

This approach to determine the subspaces is similar to the one in~Ref. \refcite{2013AhmadAlsaediBrezziMariniRusso,2018BrennerSung} for the classical $H^1$ virtual spaces. 
Here, the benefit is to have the new spaces free of the choice of $\bfx_K$. 
Note that the spaces in \eqref{Hcurl_IVE_local_2_new} 
depends on the choice of the point $\bfx_K\in \Gamma^K_h$ which can be arbitrary on the plane $\Gamma^K_h$ with a distance $\mathcal{O}(h_K)$ to the element $K$. 
However, the new space in \eqref{Hcurl_IVE_local_2} 
is invariant with respect to the various $\bfx_{K}\in\Gamma^K_h$, even though the underlying IFE spaces $\bfS^e_h(\beta^{-1}_h,0;K)$ are not.

\subsection{A discrete de Rham Complex}\label{sec:discretedeRham}
The proposed IVE spaces inherit the de Rham complex properties of standard finite element spaces including the exact sequence and commutativity. %

Thanks to the nodal, edge and face DoFs of the proposed IVE spaces, let us first define the corresponding interpolations:
\begin{subequations}
\label{interpolation}
\begin{align}
    & I^n_h: H^2(\beta;\mathcal{T}_h) \rightarrow V^n_h ~~~~~~~~~~~~ \text{satisfying} ~~ I^n_hu(\bfx) = u(\bfx), ~~~~ \forall \bfx \in \mathcal{N}_h, \label{nodal_interp}  \\
    & I^e_h: \bfH^1(\curl,\alpha,\beta;\mathcal{T}_h) \rightarrow \bfV^e_h ~~~ \text{satisfying} ~~  \int_e I^e_h\bfu\cdot\bft \dd s = \int_e \bfu\cdot\bft \dd s, ~~~ \forall e\in \mathcal{E}_h,  \label{Hcurl_interp} \\
    & I^f_h: \bfH^1(\div,\alpha;\mathcal{T}_h) \rightarrow \bfV^f_h ~~~~~~~ \text{satisfying} ~~ \int_F I^f_h\bfu\cdot\bfn \dd s = \int_F \bfu\cdot\bfn \dd s, ~~~ \forall F\in \mathcal{F}_h.  \label{Hdiv_interp} 
\end{align}
\end{subequations}
We further need the standard $L^2$ projection denoted by $\Pi^0_K:L^2(K)\rightarrow \mathcal{P}_0(K)$, 
and define the global one as $\Pi^0_h$ such that $\Pi^0_h|_T = \Pi^0_T$, i.e., $\Pi^0_h: L^2(\Omega)\rightarrow Q_h$ where
\begin{equation}
\label{Qh}
Q_h = \{ v_h: v_h\in \mathcal{P}_0(K), ~ \forall K\in\mathcal{T}_h(K) \}.
\end{equation}
These operators together with the IVE spaces will be used to formulate the continuous and discrete de Rham complex in \eqref{IVE_deRham}. 
In fact, Lemma \ref{lem_IVE_deRham_H_complex} already shows that exactness in the continuous level. 
So our focus will be on the discrete one in the lower part of \eqref{IVE_deRham}.

\begin{lemma}
\label{lem_IVE_complex}
When $\Omega$ is topologically trivial, the following complex is exact
\begin{equation}
   \label{deRham}
\begin{tikzcd}[column sep=scriptsize]
  \mathbb{R} \arrow[r, "\hookrightarrow"] 
   & V^n_h \arrow[r, "\grad"] 
   &[1em]  \bfV^e_h \arrow[r, "\curl"] 
   &[1em]  \bfV^{f}_h \arrow[r, "\div"] 
   &[1em]  Q_h \rar 
   & 0.
\end{tikzcd}
\end{equation}
\end{lemma}
\begin{proof}
The argument for showing the sequence being a complex is basically the same as Lemma \ref{lem_IVE_deRham_H_complex}. 
To show the exactness, we can look at the DoFs which form a co-chain exact complex on the cell-complex defined by the mesh. 
For the completeness, we include a detailed proof below. 

First verify $\nabla V^n_h = \ker(\curl)\cap\bfV^e_h$. 
By the classic exact sequence, given each $\bfv_h \in \ker(\curl)\cap\bfV^e_h$, there exists $v_h\in H^1(\Omega)$ such that $\nabla v_h= \bfv_h$. 
Given an interface element $K$, the jump conditions associated with $\bfV^e_h(K)$ imply that $v_h$ also satisfies those of $V^n_h(K)$. 
In addition, let $F$ be one of its face in the boundary triangulation, since $\rot_F \bfv_h = 0$ and $\bfv_h^{\tau}|_F\in\mathcal{ND}_0(F)$, 
we have $\nabla_F v_h = \bfv_h^{\tau}|_F \in [\mathcal{P}_0(F)]^2$, which implies $v_h\in \mathcal{P}_1(F)$. Hence, $v_h\in V^n_h(K)$. 
On each non-interface element $K$, $\bfv_h$ is just a constant vector, so $v_h\in \mathcal{P}_1(K)$. Therefore, we conclude $v_h\in V^n_h$.

Second, we just, to the end, prove $\div(\bfV^f_h)= Q_h$. 
Given each $q\in Q_h$, there exists a regular potential $\bfu \in \bfH^1(\Omega)$ s.t. $\div \bfu = q$. Then, we define $\bfu_h = I^f_h\bfu$ and 
$$
|K| ~q|_K =  \int_K \div(\bfu) \dd \bfx = \int_{\partial K} \bfu \cdot\bfn \dd s  = \int_{\partial K} I^f_h\bfu \cdot\bfn \dd s =  \int_K \div(I^f_h\bfu) \dd \bfx
$$
which implies $\div(I^f_h\bfu |_K)=q|_K$ on each element $K$ finishing the proof.

To verify $\curl \bfV^e_h = \bfV^f_h\cap \ker(\div)$, we can use a dimension count. 
Denote by $\#\mathcal V_h, \#\mathcal E_h,$ $\#\mathcal F_h, \#\mathcal T_h$ the number of vertices, edges, faces, and elements, respectively. 
From the surjectivity, i.e., $\div(\bfV^f_h)= Q_h$, we know $\dim (\bfV^f_h\cap \ker(\div)) = \dim \bfV^f_h - \dim Q_h = \#\mathcal F_h - \#\mathcal T_h$. 
On the other hand, $\dim \curl \bfV^e_h = \dim \bfV^e_h - \dim( \ker(\curl)\cap  \bfV^e_h) = \dim \bfV^e_h - \dim( \nabla V^n_h) = \#\mathcal E_h - \#\mathcal V_h + 1$. 
Then by Euler's formula, we get $\dim \curl \bfV^e_h = \dim(\bfV^f_h\cap \ker(\div))$. As  $\curl \bfV^e_h \subseteq \bfV^f_h\cap \ker(\div)$, we conclude that they are equal. 
\end{proof}

\section{The Immersed Virtual Element Schemes}
\label{sec:IVEscheme}

Based on the previously established spaces and projections, in this section we are ready to present the IVE schemes. 
With the exact sequence, we also develop fast solvers for the $\bfH(\curl)$ interface problem. 
At last, we present a data structure that can facilitate an efficient and vectorized implementation of the proposed method.

We shall focus on the $H^1$ and $\bfH(\curl)$ interface problems due to their vast applications. 
For simplicity, we let $(\cdot,\cdot)_D$ be the standard $L^2$ inner product on $D$.

\subsection{The IVE Scheme for $H^1$ interface problem} For the $H^1$ case, we define a local bilinear form as
\begin{equation}
\label{IVE_H1_scheme_1}
b_K(u_h,v_h) =   ( \beta_h \bfPi^{e,\beta_h}_K \nabla u_h , \bfPi^{e,\beta_h}_K \nabla v_h )_K + S_K( (I - \bfPi^{e,\beta_h}_K ) \nabla u_h , (I - \bfPi^{e,\beta_h}_K ) \nabla v_h ) ,
\end{equation}
where the projection $\bfPi^{e,\beta_h}_K$ is given in \eqref{B_proj_1} on interface elements $K$ and simply assumed to be the identity operator on non-interface elements as the standard FE spaces are used and computable. 
\LC{The first term in \eqref{IVE_H1_scheme_1} is a reasonable and computable approximation to $(\beta\nabla u_h, \nabla v_h)$, as $\bfPi^{e,\beta_h}_K$ will preserve the piecewise constant space $\bfP^e_0(\beta_h;K)$, 
but it alone does not lead to a stable method as $\bfPi^{e,\beta_h}_K \nabla (\cdot)$ contains a non-trivial kernel. 
Namely, there exists a non-constant function $v_h\in V_h^n$ s.t. $\bfPi^{e,\beta_h}_K \nabla v_h = 0$. 
In the VEM literature, two requirements are imposed for the stabilization $S_K$. 
One is the $k$-consistency, i.e., the stabilization vanishes for polynomial spaces of degree $k$. 
As $\bfPi^{e,\beta_h}_K$ can preserve the piecewise constant space $\bfP^e_0(\beta_h;K)$ and the slice operator $I - \bfPi^{e,\beta_h}_K$ is used, $S_K$ is $0$-consistent. 
Another consideration is the norm equivalence $b_K(v_h,v_h)\eqsim \| \nabla v_h\|^2$. But we really need is the coercivity;  see Lemma \ref{lem_H1_stable} and Section \ref{sec:normequivalence} below for detailed discussion.} 

Various choices of the stabilization have been proposed in the literature~\cite{Beirao-da-Veiga;Brezzi;Cangiani;Manzini:2013principles,2018BrennerSung,Cao;Chen:2018AnisotropicNC} based on different norms on the boundary. 
In this work, we will employ the following surface $H^1$ stabilization: %
\begin{equation}
\label{H1_stab}
S_K(\bfw_h,\bfz_h): = \gamma h_K\sum_{F\in\mathcal{F}_K} ( \bfw^{\tau}_h, \bfz^{\tau}_h)_{F}, %
\end{equation}
where $\bfw_h^{\tau}|_F$ and $\bfz_h^{\tau}|_F$ are the tangential components on the face $F$. 
In particular, we note that $(\nabla v_h)^{\tau}|_F =\nabla v_h -  (\nabla v_h\cdot\bfn_F)\bfn_F = \nabla_F v_h$ is the surface gradient of $v_h$, 
and it is computable since the trace of $v_h$ on $\partial K$ belongs to the standard FE space and is known. 
As the standard FE spaces are defined on the boundary triangulation, the stabilization in \eqref{H1_stab} must be piecewisely computed. 
Then, the global bilinear form is defined as
\begin{equation}
\label{IVE_H1_scheme_2}
b_h(u_h,v_h) = \sum_{K\in \mathcal{T}_h} b_K(u_h,v_h).
\end{equation}
The proposed IVE scheme is to find $u_h\in V^n_h$ such that
\begin{equation}
\label{IVE_H1_scheme_3}
b_h(u_h, v_h) = \sum_{K\in \mathcal{T}_h} \int_K f \, \widetilde{\Pi}^{e,\beta_h}_K v_h \dd \bfx, ~~~~ \forall v_h \in V^n_h.
\end{equation}
Note that the projection of $v_h$ itself is not computable for the current space, and thus we simply employ the approximated gradient $\bfPi^{e,\beta_h}_K \nabla v_h$ and the formula in Table \ref{tab_IFE_func} to form 
\begin{equation}
\label{IVE_H1_scheme_4}
\widetilde{\Pi}^{e,\beta_h}_K v_h = (\bfPi^{e,\beta_h}_K \nabla v_h)\cdot(\bfx - \bfx_K)+ c
\end{equation}
with the constant $c$ chosen such that $\int_{\partial K} \widetilde{\Pi}^{e,\beta_h}_K v_h \dd s = \int_{\partial K} v_h\dd s $. Clearly, there holds 
\begin{equation}\label{eq:tildePi}
\nabla \widetilde{\Pi}^{e,\beta_h}_K v_h = \bfPi^{e,\beta_h}_K \nabla v_h.
\end{equation}
In fact, the $H^1$ interface problem is not our focus. 
Only the stiffness matrix for the $H^1$ interface problem is needed for the auxiliary space preconditioner in the fast solver for the $\bfH(\curl)$ interface problems.

\RG{The stabilization in \eqref{H1_stab} indeed leads to a stable method which is given by the following results.
\begin{lemma}
\label{lem_stab_norm}
For each interface element $K$, there exists a constant depending only on the shape regularity of $K$, the coefficient $\beta$ and the parameter $\gamma$ s.t.
\begin{equation}
\label{lem_stab_norm_eq0} 
\| \sqrt{\beta} \nabla v_h \|_{L^2(K)} \lesssim  |v_h|_{H^{1/2}(\partial K)} \lesssim h^{1/2}_K | v_h |_{H^{1}(\partial K)} \quad v_h\in V^n_h(K).
\end{equation}
\end{lemma}
\begin{proof}
We use the energy minimization argument. Given each $v_h\in V^n_h(K)$, consider an arbitrary function $w_h\in H^1(K)$ such that $w_h-v_h=0$ on $\partial K$. 
Then, integration by parts on $K^{\pm}$ with the flux jump condition $[\beta \nabla v_h\cdot\bfn]|_{\Gamma^K}=0$ and $\nabla\cdot (\beta \nabla v_h)=0$ yields 
$$
\int_{K} \beta \nabla v_h\cdot \nabla(v_h-w_h) \dd \bfx = \int_{\partial K} \beta \nabla v_h\cdot\bfn (v_h-w_h) \dd s =0.
$$
On one hand, with the H\"older's inequality, it implies that $v_h$ minimizes the $\|\sqrt{\beta}\nabla \cdot \|_{L^2(K)}$ energy norm, i.e.,
\begin{equation}
\label{lem_stab_norm_eq1_2}
\| \sqrt{\beta} \nabla v_h \|_{L^2(K)} \le \| \sqrt{\beta} \nabla w_h \|_{L^2(K)}. 
\end{equation}
On the other hand, by the inverse trace theorem given in~Section 27 in Ref. \refcite{2018BrennerSung}, we have a function $z_h$ such that $z_h=v_h$ on $\partial K$ and
$
 |z_h|_{H^1(K)} \lesssim |v_h|_{H^{1/2}(\partial K)}
$.
Hence, using \eqref{lem_stab_norm_eq1_2} with $w_h=z_h$ we arrive at
\begin{equation}
\label{lem_stab_norm_eq2}
\| \sqrt{\beta} \nabla v_h \|_{L^2(K)} \lesssim \|  \nabla z_h \|_{L^2(K)} \lesssim |v_h|_{H^{1/2}(\partial K)},
\end{equation}
which gives the first inequality in \eqref{lem_stab_norm_eq0}. The second inequality in \eqref{lem_stab_norm_eq0} simply follows from~(2.16) in Ref. \refcite{2018BrennerSung}.
\end{proof}
\begin{lemma}
\label{lem_H1_stable}
There holds that
\begin{equation}
\label{lem_H1_stable_eq0}
\| \sqrt{\beta}\nabla v_h \|^2_{L^2(K)} \lesssim b_K(v_h,v_h), ~~~~ \forall v_h\in V^n_h(K),%
\end{equation}
where the constant depends only on the shape regularity of $K$, $\beta$, and $\gamma$.
\end{lemma}
\begin{proof}
We will use the projection $\widetilde{\Pi}^{e,\beta_h}_K$ and the relation \eqref{eq:tildePi}. 
By Lemma \ref{lem_stab_norm}, the triangular inequality and the trace inequality by Lemma \ref{lem_trace_inequa}, we have
\begin{equation}
\begin{split}
\label{lem_H1_stable_eq1}
\| \sqrt{\beta}\nabla v_h \|_{L^2(K)} & \lesssim h^{1/2}_K | v_h |_{H^{1}(\partial K)} \lesssim h^{1/2}_K | \widetilde{\Pi}^{e,\beta_h}_K  v_h |_{H^{1}(\partial K)} + h^{1/2}_K |  v_h - \widetilde{\Pi}^{e,\beta_h}_K v_h |_{H^{1}(\partial K)} \\
& \lesssim \|  \bfPi^{e,\beta_h}_K \nabla v_h \|_{L^2(K)} + h^{1/2}_K \| \nabla_F (v_h - \widetilde{\Pi}^{e,\beta_h}_K  v_h )\|_{L^2(\partial K)} \lesssim b_K(v_h,v_h),
\end{split}
\end{equation}
which finishes the proof.
\end{proof}
}

\subsection{The IVE scheme for the $\bfH(\curl)$ interface problem}
In this case, we need to deal with the terms of $\curl \curl \bfu$ and $\bfu$ separately. 
\RG{For the same reason discussed above, we need to project both $\curl \bfu_h$ and $\bfu_h$ and then add their associated stabilization terms to enforce coercivity.} 
For the $\curl\curl$ term, we introduce
\begin{equation}
\begin{split}
\label{IVE_Hcurl_scheme_1}
a^1_{K}(\bfu_h,\bfv_h) = & ( \alpha_h \bfPi^{f,\alpha_h}_K \curl\bfu_h , \bfPi^{f,\alpha_h}_K \curl\bfv_h )_K \\
&+ S^1_K( (I- \bfPi^{f,\alpha_h}_K ) \curl\bfu_h, (I - \bfPi^{f,\alpha_h}_K ) \curl\bfv_h ),
\end{split}
\end{equation}
where, similarly, $\bfPi^{f,\alpha_h}_K$ is chosen as \eqref{A_proj_1} on interface elements but just the identity operator on non-interface elements. \LC{As $\curl \bfV^e_h\subset \bfV^f_h$}, the stabilization term is defined as
\begin{equation}
\label{Hcurl_stab_1}
S^1_K( \bfw_h,\bfz_h ) = \gamma_1 h \sum_{F\in\mathcal{F}_K} ( \bfw_h\cdot\bfn_F , \bfz_h \cdot\bfn_F)_{F},
\end{equation}
where we note that $\curl \bfu_h\cdot\bfn = \rot_F \bfu_h$ for $\bfu_h \in \bfV^e_h$ can be computed through the formula in \eqref{rotcurl} with $\rot_F \bfu_h = |F|^{-1} \int_{\partial F}\bfu_h\cdot\bft \dd s$ on each triangular face $F$.

The bilinear form for the weighted $L^2$ inner product is defined as
\begin{equation}
\begin{split}
\label{IVE_Hcurl_scheme_2}
a^0_{K}(\bfu_h,\bfv_h) = &( \beta_h \bfPi^{e,\beta_h}_K \bfu_h, \bfPi^{e,\beta_h}_K \bfv_h  )_K + S^0_K( ( I- \bfPi^{e,\beta_h}_K ) \bfu_h, ( I- \bfPi^{e,\beta_h}_K ) \bfv_h ),
\end{split}
\end{equation}
where $\bfPi^{e,\beta_h}_K$ is defined in \eqref{B_proj_curl_new} on interface elements and the identity on non-interface elements. 
The stabilization is given by
\begin{equation}
\label{Hcurl_stab_2}
S^0_K( \bfw_h,\bfz_h )  = \gamma_0 \sum_{F\in\mathcal{F}_K} (\bfw_h^{\tau}, \bfz^{\tau}_h )_{F},
\end{equation}
where $\bfw_h^{\tau}$ and $\bfz^{\tau}_h$ still denote the tangential components onto each face $F$. 
With the triangulation on faces, $\bfw_h^{\tau} \in \bfV^e_h$ is computable through the edge DoFs. 
We highlight that the scaling $h^0=1$ in the stabilization $S^0_K( \bfw_h,\bfz_h )$ is different from the usual $h$ in classical VEM in~Ref. \refcite{BEIRAODAVEIGA2021,2017VeigaBrezziDassiMarini,2016VeigaBrezziMarini,2020BeiroMascotto}, 
and this is also the key for the proposed method to produce optimal convergent solutions. 
\LC{Changing the scaling from $\mathcal O(h)$ to $\mathcal O(1)$ may increase the consistency error locally. 
More precisely $S^0_K(\bfv_h, \bfv_h) = \mathcal O(h_K^2)$ while $( \beta_h \bfPi^{e,\beta_h}_K \bfv_h, \bfPi^{e,\beta_h}_K \bfv_h  )_K = \mathcal O(h_K^3)$. 
But such a loss of order $h$ is restricted to the interface elements only whose number is $\mathcal O(h)$ fraction of the total number of elements. 
So overall the $L^2$-norm is still possible of optimal order.}
The theoretical justification has been given for the 2D case in Ref.~\refcite{2021CaoChenGuoIVEM} \LC{and will be explored in a forthcoming paper for the 3D case}.

Then, we can define the global bilinear form
\begin{equation}
\label{IVE_Hcurl_scheme_3}
a_h(\bfu_h,\bfv_h) = \sum_{K\in\mathcal{T}_h} a^1_{K}(\bfu_h,\bfv_h) + a^0_{K}(\bfu_h,\bfv_h).
\end{equation}
The proposed IVE scheme for the $\bfH(\curl)$ interface problem is to find $\bfu_h\in \bfV^e_h$ such that
\begin{equation}
\label{IVE_Hcurl_scheme_4}
a_h(\bfu_h,\bfv_h) =  \sum_{K\in\mathcal{T}_h}  \int_K \bff\cdot \bfPi^{e,\beta_h}_K \bfv_h \dd \bfx, ~~~~ \forall \bfv_h \in \bfV^e_h.
\end{equation}

Next, we show that the proposed stabilization can indeed make the bilinear form coercive. 
\begin{lemma}
\label{lem_Hdiv_stab}
There exists a constant depending only on the shape regularity of $K$ s.t.
\begin{equation}
\label{lem_Hdiv_stab_eq0}
\| \bfv_h \|_{L^2(K)} \lesssim h_K^{1/2} \| \bfv_h\cdot\bfn \|_{L^2(\partial K)} \quad \bfv_h\in \bfV^f_h(K).
\end{equation}
\end{lemma}
\begin{proof}
As $\curl(\alpha\bfv_h) = \mathbf{ 0}$, by \eqref{lem_divcurl_eq01}, we only need to estimate $\div(\bfv_h)$.
Noticing $\div(\bfv_h)$ is a constant, we can write down
\begin{equation}
\label{lem_Hdiv_stab_eq3}
\begin{aligned}
  \| \div(\bfv_h) \|_{L^2(K)} & = |K|^{-1/2} \abs{\int_K \div(\bfv_h) \dd x} 
\\ 
& = |K|^{-1/2} \abs{ \int_{\partial K} \bfv_h\cdot\bfn \dd s } \lesssim h_K^{-1/2} \| \bfv_h\cdot\bfn \|_{L^2(\partial K)},
\end{aligned}
\end{equation}
where we have used that $|K|/|\partial K|\approx h_K$. %
\end{proof}
\begin{lemma}
\label{lem_Hcurl_stab}
For every function $\bfv_h\in \bfV^e_h$, there holds
\begin{subequations}
\label{lem_Hcurl_stab_eq0}
\begin{align}
    & \| \curl \bfv_h \|^2_{L^2(K)} \lesssim a^1_K(\bfv_h,\bfv_h),  \label{lem_Hcurl_stab_eq01} \\
    & \| \bfv_h \|^2_{L^2(K)} \lesssim a^0_K(\bfv_h,\bfv_h), \label{lem_Hcurl_stab_eq02}
\end{align}
\end{subequations}
where the constants depend only on the shape regularity of $K$, $\beta$, and $\gamma_1$, $\gamma_2$.
\end{lemma}
\begin{proof}
Let us first show \eqref{lem_Hcurl_stab_eq01}. By Lemma \ref{lem_Hdiv_stab} and the de Rham complex, we have
\begin{equation}
\label{lem_Hdiv_stab_eq-1}
 \| \curl \bfv_h \|_{L^2(K)} \lesssim h^{1/2}_K \| \curl\bfv_h\cdot \bfn \|_{L^2(\partial K)}.
\end{equation}
Then, we apply the trace inequality for the $\bfH(\div)$ functions by Lemma \ref{lem_trace_inequa} to obtain
\begin{equation}
\begin{split}
\label{lem_Hcurl_stab_eq2}
\| \curl\bfv_h\cdot \bfn \|_{L^2(\partial K)} & \lesssim \| \bfPi^{f,\alpha_h}_K \curl \bfv_h\cdot\bfn \|_{L^2(\partial K)} 
\\
& \quad + \| \curl \bfv_h\cdot\bfn - \bfPi^{f,\alpha_h}_K \curl \bfv_h\cdot\bfn \|_{L^2(\partial K)} 
\\
& \lesssim h^{-1/2}_K \| \bfPi^{f,\alpha_h}_K \curl \bfv_h \|_{L^2( K)} \\
&\quad + \| \curl \bfv_h\cdot\bfn - \bfPi^{f,\alpha_h}_K \curl \bfv_h\cdot\bfn \|_{L^2(\partial K)}.
\end{split}
\end{equation}
Substituting \eqref{lem_Hcurl_stab_eq2} into \eqref{lem_Hdiv_stab_eq-1} yields \eqref{lem_Hcurl_stab_eq01}.
As for \eqref{lem_Hcurl_stab_eq02}, applying \eqref{lem_divcurl2_eq01} with the appropriate scaling, we have
\begin{equation}
\begin{split}
\label{lem_Hcurl_stab_eq1}
\| \bfv_h \|_{L^2(K)}  \lesssim h_K \| \curl \bfv_h \|_{L^2(K)} + h^{1/2}_K \| \bfv_h \times \bfn \|_{L^2(\partial K)}.
\end{split}
\end{equation}
For the second term, we apply the trace inequality for the $\bfH(\curl)$ functions by Lemma \ref{lem_trace_inequa} to obtain
\begin{equation}
\begin{split}
\label{lem_Hcurl_stab_eq3}
\| \bfv_h \times \bfn \|_{L^2(\partial K)} & \lesssim \| \bfPi^{e,\beta_h}_K \bfv_h \times \bfn \|_{L^2(\partial K)} + \| ( \bfv_h - \bfPi^{e,\beta_h}_K \bfv_h) \times \bfn \|_{L^2(\partial K)} \\
& \lesssim h^{-1/2}_K  \| \bfPi^{e,\beta_h}_K \bfv_h \|_{L^2( K)} + \| ( \bfv_h - \bfPi^{e,\beta_h}_K \bfv_h)^{\tau} \|_{L^2(\partial K)}.
\end{split}
\end{equation}
Putting \eqref{lem_Hcurl_stab_eq2} and \eqref{lem_Hcurl_stab_eq3} into \eqref{lem_Hcurl_stab_eq1}, we have the desired estimate.
\end{proof}
We emphasize that the coercivity constants depend only on the shape regularity of the underlying triangulation, 
the coefficient $\beta$, and the parameter $\gamma$, 
but most importantly, not on the location of the intersection points, i.e., robust to the cut of the interface.

\RG{
\begin{remark}
\label{rem_gamma}
Lemmas \ref{lem_H1_stable} and \ref{lem_Hcurl_stab} immediately imply that $b_h(\cdot,\cdot)$ and $a_h(\cdot,\cdot)$ are norms on $V^n_h$ and $\bfV^e_h$, respectively. 
These two lemmas hold regardless of the choice of $\gamma>0$, $\gamma_1>0$ and $\gamma_2>0$, i.e., the method does not need those parameters to be large enough required by many traditional unfitted-mesh methods~\cite{2015BurmanClaus,2015LinLinZhang}, 
and thus the resulting linear systems are always positive-definite. 
Roughly speaking, it can be understood that the proposed IVE scheme is ``more conforming" such that weaker weights are needed in the stabilization. 
This is particularly important for the $\bfH(\curl)$ problem, as we do not need to use $h^{-1}$ scaling in the stabilization, 
which can avoid the suboptimal convergence in \eqref{h-1stab}. 
Instead, $\mathcal{O}(1)$ and $\mathcal{O}(h)$ scaling are used for the stabilization associated with $\bfu_h$ and $\curl \bfu_h$ terms, 
which is key to achieve optimal convergence by our numerical experiments. 
Nevertheless, the rigorous analysis is still very involved, and in the next subsection we shall briefly describe the challenges. 
\end{remark}}

\subsection{Comments on the norm equivalence and error analysis}\label{sec:normequivalence}
\RG{In the vast VEM literature~\cite{Beirao-da-Veiga;Brezzi;Cangiani;Manzini:2013principles,2020BeiroMascotto,2018BrennerSung}, the norm equivalence results are desired for error and stability analysis:
\begin{subequations}
\label{rem_normequiv_eq1}
\begin{align}
    &  \| \nabla v_h \|^2_{L^2(K)} \lesssim b_K(v_h,v_h) \lesssim \| \nabla v_h \|^2_{L^2(K)} ~~~ \forall v_h\in V^n_h(K),    \\
    &  \| \bfv_h\|^2_{\bfH(\curl;K)} \lesssim a_{K}(\bfv_h,\bfv_h) \lesssim \| \bfv_h\|^2_{\bfH(\curl;K)} ~~~ \forall \bfv_h\in \bfV^e_h(K).
\end{align}
\end{subequations}
The left inequalities in \eqref{rem_normequiv_eq1}, i.e., the coercivity, are given by Lemmas \ref{lem_H1_stable} and \ref{lem_Hcurl_stab}, respectively, 
in which the constants are independent of interface location. 
Although the right two inequalities in \eqref{rem_normequiv_eq1} indeed hold, their constants may depend on the interface location, 
as the inverse inequalities on the boundary triangulation are needed in the analysis.} 

\RG{Let us take the $H^1$ case as an illustration example. By the boundedness property of the projection $\bfPi^{e,\beta_h}_K$, we trivially have 
$$\|\bfPi^{e,\beta_h}_K \nabla v_h \|_{L^2(K)} \le C_{Pr} \| \nabla v_h \|_{L^2(K)},$$
where the constant $C_{Pr}$ only depends on the geometry of $K$. The problem is on the stabilization term. We may prove 
\begin{equation*}
\begin{split}
 & h_K^{1/2}\|\nabla_F(I - \widetilde{\Pi}^{e,\beta_h}_K)v_h \|_{L^2(\partial K)}\\
 \le {}& C_{\rm inv} \| (I - \widetilde{\Pi}^{e,\beta_h}_K)v_h \|_{L^2(\partial K)}\\
\le {}& C_{\rm inv} C_t \left ( h^{-1/2}_K \| (I - \widetilde{\Pi}^{e,\beta_h}_K)v_h \|_{L^2(K)} + h^{1/2}_K | (I - \widetilde{\Pi}^{e,\beta_h}_K)v_h |_{H^1(K)} \right ) \\
\le {}& C_{\rm inv} C_t (C_{Pc}+1)  h^{1/2}_K | (I - \widetilde{\Pi}^{e,\beta_h}_K)v_h |_{H^1(K)} \\
\le {}& C_{\rm inv} C_t (C_{Pc}+1) (C_{Pr}+1)  h^{1/2}_K \|  \bfPi^{e,\beta_h}_K \nabla v_h \|_{L^2(K)},
\end{split}
\end{equation*}
where the first inequality is an inverse inequality on the surface triangulation with the constant $C_{\rm inv}$, 
and the second and the third ones are the trace and Poincar\'e inequalities with the constants $C_t$ and $C_{Pc}$. 
Note that $C_t$, $C_{Pr}$ and $C_{Pc}$ only depend on the geometry of $K$; but $C_{\rm inv}$ depends on the element boundary triangulation which contain anisotropic triangles, and shrinking elements may make $C_{\rm inv}$ blow up. 
Indeed, restricting to the boundary $(I - \widetilde{\Pi}^{e,\beta_h}_K)v_h|_{F}$ is linear, and its surface gradient can be computed exactly using the $\cot$ formulae. 
The existence of small angles in the boundary triangulation will make the corresponding entry large, and thus robust norm equivalence may not hold. A similar issue applies to the $\bfH(\curl)$ case. }

\RG{In fact, for 3D VEM, to our best knowledge, almost all the analysis in the literature requires shape regularity of both the elements and faces such that the norm equivalence above can hold. 
In our case, however, the boundary triangulation does not satisfy the shape regularity causing essential difficulties for analysis. 
An alternative approach is to use the ``error equation" approach~\cite{Cao;Chen:2018Anisotropic,2021CaoChenGuo,2021CaoChenGuoIVEM} that may overcome the shape regularity issue. 
A careful study of the robustness to the shape of boundary triangulations is needed.
}

\subsection{Implementation}
\label{subsec:implement}

Inherited from the classical VEM, the implementation of the proposed algorithm is highly vectorized. Computing the projections from IVE spaces and assembly of matrices significantly outperform the classical IFE methods. 
To see this, following Ref.~\refcite{2017ChenWeiWen}, we describe a \textsf{face2elem} and a \textsf{face} data structure that can greatly facilitate the implementation. 
\textsf{face2elem} is a vector mapping from each (local) face's index to its mother element's index. \textsf{face} is a matrix containing each face's DoFs (node or edge) on its rows. 
Here, we use the tetrahedral interface elements in Figure \ref{fig:tetId} to illustrate the data structures. 
In Figure \ref{fig:tetId}, the red and blue segments are, respectively, cutting edges by the interface and newly added edges for the surface triangulation. 
The indices are shown on the two plots for all the vertices and edges. Suppose the index of this element is $1$, and then the desired data structures of \textsf{face2elem} and \textsf{face} are shown in \eqref{dataStructure}.

\begin{figure}[h]
\centering
\includegraphics[width=1.45in]{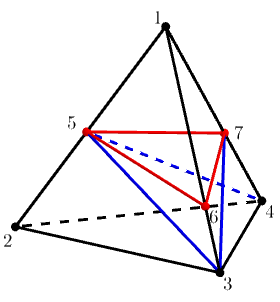}
\includegraphics[width=1.5in]{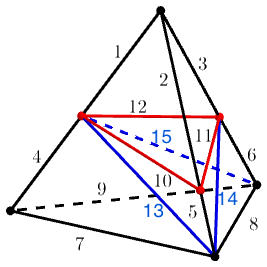}
\caption{Indices of nodes (left) and edges (right) of an interface element.}
\label{fig:tetId}
\end{figure}

\begin{equation}
\label{dataStructure}
\textsf{face2elem}:\left.\begin{array}{|c|c|}\hline 1 & 1 \\\hline 2 & 1 \\\hline 3 & 1 \\\hline 4 & 1 \\\hline 5 & 1 \\\hline 6 & 1 \\\hline 7 & 1 \\\hline 8 & 1 \\\hline 9 & 1 \\\hline 10 & 1 \\\hline \end{array}\right.
~~~~
\textsf{face}(\text{nodes}): \left.\begin{array}{c}1 \\2 \\3 \\4 \\5 \\6 \\7 \\8 \\9 \\10\end{array}\right.
\left.\begin{array}{|c|c|c|}\hline 1 & 5 & 6 \\\hline 1 & 6 & 7 \\\hline 1 & 7 & 5 \\\hline 2 & 3 & 5 \\\hline 3 & 6 & 5 \\\hline 3 & 4 & 7 \\\hline 7 & 6 & 3 \\\hline 4 & 5 & 7 \\\hline 2 & 4 & 5 \\\hline 2 & 3 & 4 \\\hline \end{array}\right.
~~~~
\textsf{face(\text{edges})}: \left.\begin{array}{c}1 \\2 \\3 \\4 \\5 \\6 \\7 \\8 \\9 \\10\end{array}\right.
\left.\begin{array}{|c|c|c|}\hline 1 & 10 & 2 \\\hline 2 & 11 & 3 \\\hline 3 & 12 & 1 \\\hline 7 & 13 & 4 \\\hline 13 & 5 & 10 \\\hline 8 & 6 & 14 \\\hline 5 & 14 & 11 \\\hline 6 & 12 & 15 \\\hline 4 & 9 & 15 \\\hline 7 & 8 & 9 \\\hline \end{array}\right.
\end{equation}

The key feature of VEM in implementation is to compute the projections through the DoFs. Let us use $\bfPi^{f,\alpha_h}_K\curl\bfv_h$ and the formula in \eqref{A_proj_2} as an example to describe the procedure. 
Given an element $K$ with the global index $i$, $i=1,2,...,|\mathcal{T}_h|$, we need to compute $\bfPi^{f,\alpha_h}_K\curl\bfvarphi^e_{h,k}$ for a local edge index $k$, $k=1,2,...,|\mathcal{E}_K|$. 
Here $\bfPi^{f,\alpha_h}_K\curl\bfvarphi^e_{h,k}$ is a constant vector denoted as $\bfc_k\in \bfP^f_0(\alpha_h;K)$ with $\bfc^{\pm}_k := ( \bfPi^{f,\alpha_h}_K \curl\bfvarphi^e_{h,k} )|_{K_{h}^{\pm}}$, where $\bfvarphi^e_{h,k}$ is the edge shape function with respect to the $k$-th edge. 
Now, we let the test function $\bfp_h$ in \eqref{A_proj_2} be the three unit vectors: $\bfp^+_{h,l}=\bfe_l$, $l=1,2,3$, with $\bfp^-_{h,l}=M^{f,\alpha_h}_K\bfp^+_{h,l}$ with $M^{f,\alpha_h}_K$ given by \eqref{lem_ABeigen_eq1}. 
Then, we can rewrite \eqref{A_proj_2} into a matrix-vector equation only about $\bfc^-_k$:
\begin{equation}
\begin{aligned}
\label{implement1}
& (\alpha^- |K_{h}^-| + \alpha^+ |K_{h}^+| ( M^{f,\alpha_h}_K )^{\top} M^{f,\alpha_h}_K ) \bfc^-_k %
\\
 =& \int_{\partial K} (\bfvarphi^e_{h,k} \times \bfn) \left[ (\alpha_h\bfp_{h,1})^{\tau}, (\alpha_h\bfp_{h,2})^{\tau}, (\alpha_h\bfp_{h,3})^{\tau} \right]^{\top}  \dd s, 
\end{aligned}
\end{equation}
where $\bfvarphi^e_{h,k} \times \bfn$ is a rotation of $(\bfvarphi^e_{h,k})^{\tau}$. $(\bfvarphi^e_{h,k})^{\tau}$ is a 2D N\'ed\'elec polynomial function tangentially defined on each face, and the integration of this function associated with the $k$-th edge can be determined by the two nodes retrieved from the \textsf{face} data structure. 
Particularly, if the $k$-th edge does not belong to the boundary of a face $F$, then there is no contribution of this face to the right-hand side of \eqref{implement1}.

With \eqref{implement1}, we highlight that the geometric information needed in the computation has been automatically encoded in the data structures \textsf{face2elem} and \textsf{face}, which is the key for the efficient vectorized code. We report the CPU time for computing the projections and generate matrices in Table \ref{table:example1} to demonstrate the efficiency. This very feature makes the proposed method distinguished from all the classical IFE methods in the literature that have to use more detailed geometric information to compute the IFE functions.

\subsection{Preconditioning}
\label{subsec:precond}

Solving the resulting linear system from Maxwell's equations is one of the central challenges in computational electromagnetism, and the interface may make it even more difficult. 
With a slight abuse of notation, we denote the linear system from the proposed IVE discretization of the $\bfH(\curl)$ interface problem as
\begin{equation}
\label{HcurlLS}
A \bfu_h =  \bff_h,
\end{equation}
where $\bfu_h\in \mathbb{R}^{\# \texttt{edge}}$ denotes the vector representation in DoFs.
This system is solved by the preconditioned conjugate gradient (PCG) method. To our best knowledge, the development of fast solvers of VEM specifically for $H(\curl)$-equations has not been discussed in any literature. Without suitable preconditioners, the PCG solver can be extremely slow, see the comparison in Table \ref{table:exampleHcurl}. 
In this work, we develop a fast solver for IVE discretization of the $\bfH(\curl)$ interface problem that involves two techniques. 
Thanks to the de Rham complex, the first one is the \emph{auxiliary space preconditioner} for the $\bfH(\curl)$ equation which is developed by Hiptmair and Xu in Ref.~\refcite{2007HiptmairXu} (HX preconditioner) based on the auxiliary space framework in Ref.~\refcite{Xu1996auxiliary}. 
The second one is a block diagonal smoother to handle the anisotropic element shape near the interface. \LC{In the experiments, both are used the implementation in $i$FEM~\cite{Chen:2008ifem}.}

The resulting HX-preconditioner for the $\bfH(\curl)$ systems is in the form
\begin{equation}\label{cnxBcurl}
B^{\curl}= R^{\curl}+\Pi {\boldsymbol B}^{\grad}
\Pi^{\top} + G\; B^{\grad} G^{\top},
\end{equation}
which consists of the following three components:
\begin{itemize}
  \item a smoother $R^{\curl}$ of the $\bfH(\curl)$ matrix $A$,
  \item an algebraic multigrid (AMG) solver $B^{\grad}$ for a scalar Laplacian matrix,
  \item an AMG solver ${\boldsymbol B}^{\grad}$ for a vector Laplacian matrix. 
\end{itemize}
We simply employ the incidence matrix associated with the operator \LC{$\nabla:\bfV^n_h\rightarrow \bfV^e_h$} as the discrete gradient matrix $G$ which resembles that from the lowest order N{\'e}d{\'e}lec element on simplicial meshes. 
$G$ maps the nodal DoFs (columns) to edge DoFs (rows). 
There are two nonzero entries, $\pm 1$, on each row. 
The columns of these entries correspond to the nodes of the edge. 
The sign is determined by the global orientation of an edge. 
The node-to-edge transfer matrix is denoted by $\Pi: \prod_{i=1}^3\mathbb{R}^{\# \texttt{node}}\to\mathbb{R}^{\# \texttt{edge}}$. 
Note that these two matrices being well-defined are based on the node and edge DoFs of the $H^1$ and $\bfH(\curl)$ IVE spaces.

For $\bfH(\curl)$ problems, it is known that a multigrid solver for Poisson-type equations is not sufficient since the discrete operator corresponding to $\curl(\alpha \curl) + \beta I$ behaves differently for a gradient field and a solenoidal field (see e.g., Ref.~\refcite{ArnoldFalkEtAl2000Multigrid}). 
When sufficient piecewise regularity is assumed, we have by Ref.~\refcite{CostabelDaugeEtAl1999Singularities}
$$
\|\alpha \grad \bfu\|^2\eqsim \|\alpha \curl \bfu\|^2+\|\alpha \div \bfu\|^2.
$$
Hence, if $\bfu =\curl \bfw \in (\ker(\curl))^\perp$, for some suitable $\bfw$, such that $\div \bfu = 0$, then 
$$
(\alpha\curl \bfu, \curl \bfu)+(\beta \bfu, \bfu)\eqsim (\alpha\grad \bfu, \grad \bfu)+(\beta \bfu, \bfu),
$$
which corresponds to the following operator:
\begin{equation}
  \label{cnx-eq:elliptic1}
{\boldsymbol B}^{\grad}\bfu:= -\div(\alpha \nabla \bfu)+\beta \bfu
\end{equation}
that will be assembled as an auxiliary matrix and can be solved by an AMG solver for the vector $H^1$-interface problem.
On the other hand,  if $\bfu, \bfv\in \ker(\curl)$, i.e., $\bfu=\nabla p$ and $\bfv=\nabla q$, for some suitable $p,q$, then
$$
(\alpha\curl \bfu, \curl \bfv)+(\beta \bfu, \bfv)= (\beta \nabla p, \nabla q),
$$
thus we can formulate the matrix problem for the gradient part of the solution by 
$B^{\grad} = G^{\top}AG$, which corresponds to the following operator:
\begin{equation}
 \label{cnx-eq:elliptic2}
B^{\grad} p :=  -\div(\beta \nabla p)
\end{equation}
that can be again solved efficiently by an AMG solver for the $H^1$-interface problem.

Next, we present a block diagonal smoother (preconditioner). A block matrix is formed by the edge DoFs in the neighborhood expanding from the interface. We begin with the collection of the DoFs that are near the interface:
\begin{equation}
\label{DoFset1}
\mathcal{D}_1 = \{ e\in\mathcal{E}_h : ~ \exists K\in \mathcal{T}^i_h ~ \text{such that} ~ e\in\mathcal{E}_K \}.
\end{equation}
Then, starting from $\mathcal{D}_1$ we iteratively define
\begin{equation}
\label{DoFset2}
\mathcal{D}_l = \{ e\in\mathcal{E}_h : e ~ \text{has at least one node belonging to the edges of} ~ \mathcal{D}_{l-1}  \}.
\end{equation}
Let $A_l$ be the matrix of the entries in $A$ associated with the DoFs in $\mathcal{D}_l $. Then, we rewrite \eqref{HcurlLS} into
\begin{equation}
\label{blockeqn}
A \bfu :=
\begin{pmatrix}
  A_N & A_{Nl} \\
A_{lN} & A_l
\end{pmatrix}
\begin{pmatrix}
    \bfu_N \\ \bfu_l
\end{pmatrix}
=\begin{pmatrix}
    \bff_N\\ \bff_l
\end{pmatrix}.
\end{equation}
Here, the key is to solve the $A_l$ block by a direct solver, which is indeed the price to be paid by the proposed method. 
However, since the size of $A_l$, i.e., $\# \mathcal{D}_l$, is in the order of $\mathcal{O}(\# \text{total DoF}^{2/3})$ for reasonably small $l$, this direct solver is generally efficient. 
The expanding width can reduce the number of iterations required for the resulting solver, see Table \ref{table:exampleHcurlFlat} for the comparison. 
Meanwhile, our numerical experience suggests that the increased cost is negligible for small $l$'s, e.g., the expanding width $l=1$ or $2$ is enough. 
It is almost equivalent to directly solving a 2D linear system which can be efficiently handled by ``backslash'' (\textsf{mldivide}) in Matlab. 
Furthermore, since the direct solver will be called multiple times, we opt to store the $LU$-factorization of $A_l$ in the inner iteration (preconditioning) to be more efficient. 
The residual equation of the $A_N$ block can be efficiently solved using a block or point-wise Gauss-Seidel smoother. 
At last, we summarize the algorithm in the following for a fixed $l$, and denote $A_I:= A_l$.

{\begin{center}
\begin{minipage}{.6\textwidth}

\begin{algorithm}[H]
\caption{An HX preconditioned CG}
\label{alg:HX}
\begin{algorithmic}[1]
\Require $\mathbf{u}^{(0)}$,  $\texttt{tol}$, $M$, $l$, block form of $A$.
\Ensure $\mathbf{u}^{\text{MG}}$.

\State{$k = 0.$}

\State{$\bfr_I \gets \bff_I - A_I \bfu^{(0)}$}

\While{True}

\State{$\bfr_I\gets \text{LUSolve}(A_I, \bfr_I)$}

\State{$\bfr_N \gets \text{Smoother}(A_N, \bfr_N)$}

\State{$\bfr \gets [\bfr_N, \bfr_I]$}

\State{$\bfr \gets \bfr + \Pi (\text{AuxSolve} (A^{\text{aux}}, \Pi^{\top} \bfr))$}

\State{$\bfr_c \gets G (\text{AuxSolve} (A,  G^{\top} \bfr))$}

\State{$\bfr\gets \bfr+\bfr_c$}

\State{$\mathbf{u}^{\text{MG}} \gets \operatorname{CG} (A, \bfr)$}

\If{$k>M$ or $\text{norm}(r) < \texttt{tol}$}
\State{Break}
\EndIf

\State{$k\gets k+1$}
\EndWhile

\end{algorithmic}
\end{algorithm}
\end{minipage}

\end{center}
}

\section{Numerical Examples}

In this section, we present a class of numerical examples to validate the aforementioned advantages of the proposed method. 
The background unfitted mesh is generated by cutting $\Omega$ into $N^3$ cubes and each cube is then cut into several tetrahedra with the mesh size be $h=1/N$.

\begin{figure}[h]
\centering
\includegraphics[width=1.6in]{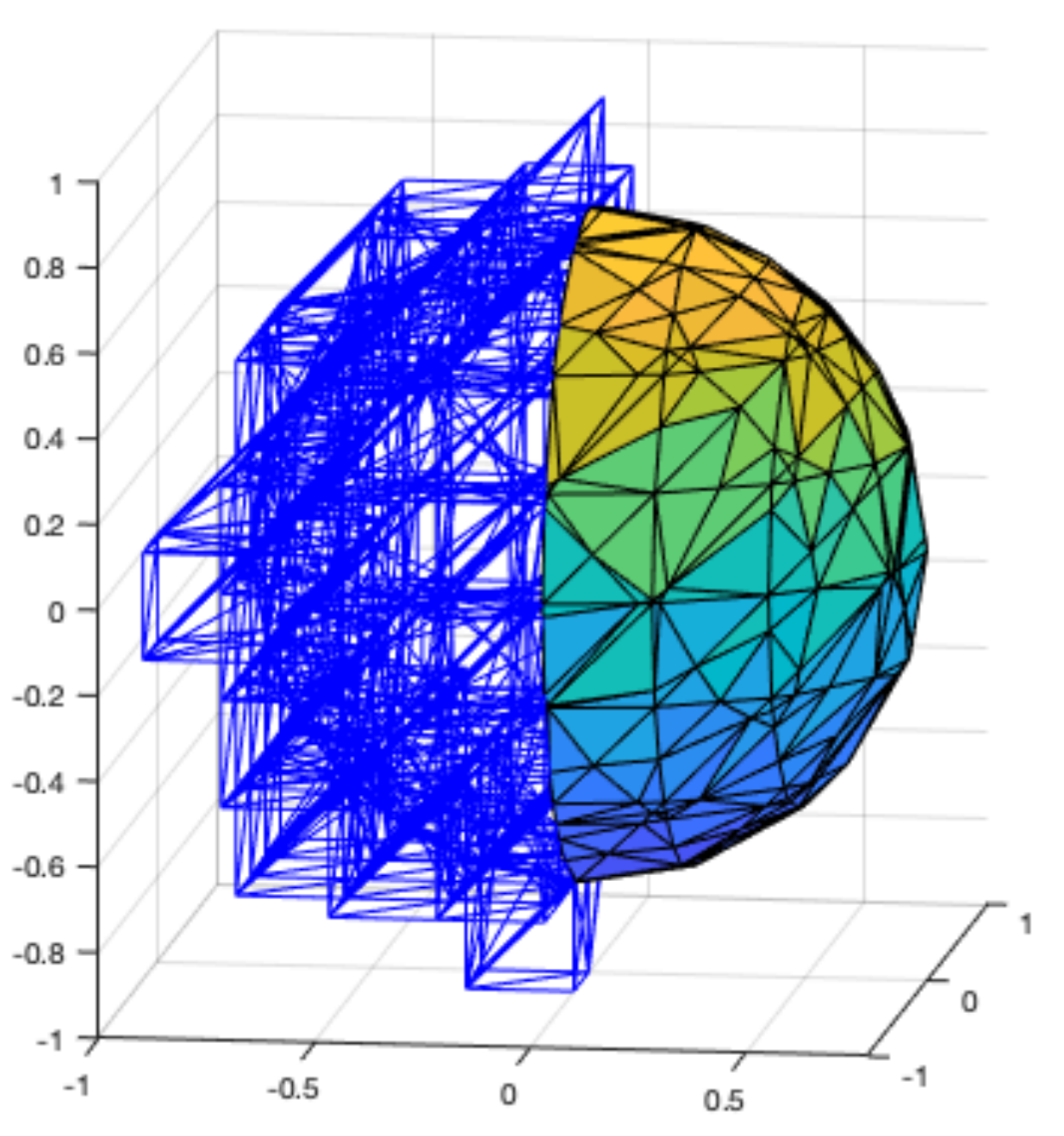}
~~~~~~~~~~
\includegraphics[width=1.6in]{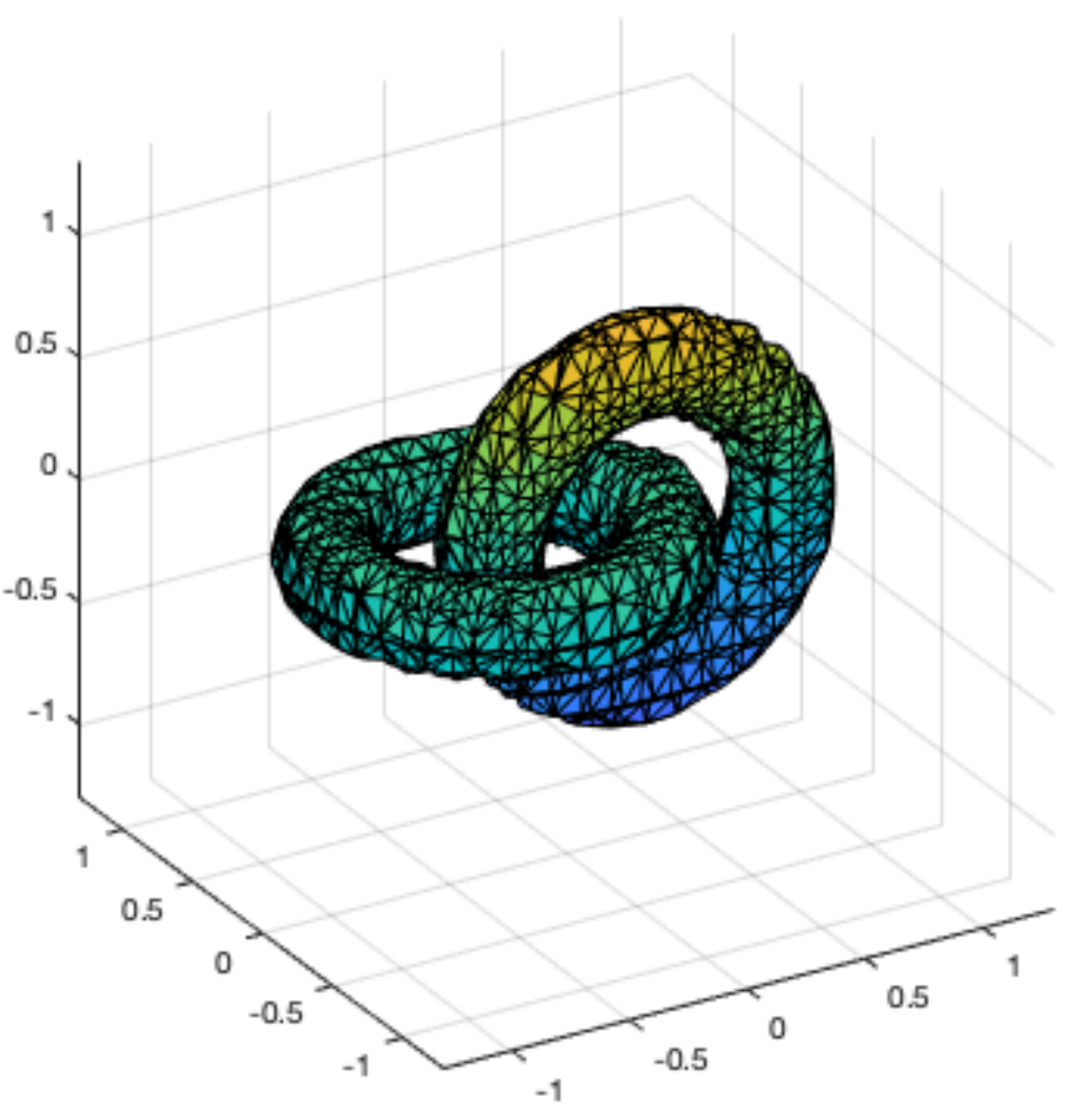}
\caption{Plots of interface and triangulation: sphere (left) and two twisted tori (right). For the spherical interface, the boundary triangulation of interface element is plotted in blue on the left half of the sphere. }
\label{fig:interfaceExp}
\end{figure}

\subsection{The $H^1$-interface problem}
We first consider the $H^1$-interface problem given by \eqref{model_H1} for a spheric interface shown in Figure \ref{fig:interfaceExp} on the domain $\Omega=(-1,1)^3$. The exact solution is constructed as
\begin{equation}
\label{H1example1}
u(\bfx) = 
\begin{cases}
      & \exp( (\| \bfx \|^2 - r^2)/\beta^- ),~~~~~~~~~~ \text{if}~\bfx\in \Omega^-, \\
      & \sin( (\| \bfx \|^2 - r^2)/\beta^+ )+1,~~~~~~ \text{if}~\bfx\in \Omega^+,
\end{cases}
\end{equation}
where the source term of \eqref{model_H1} as well as the boundary conditions are computed accordingly. The numerical experiment is carried on the meshes of $N=10,20,30,...,160$. 
We first report the CPU time to compute the projections of IVE functions and the related matrix assembling in Table \ref{table:example1} for the meshes $N=60$ to $160$. 
Note that at the finest level, there are approximately 153600 interface elements, and based on the proposed data structure, computing the IVE projections is highly efficient. 
As for the matrix assembling, we observe that the majority of time is devoted to the stabilization term, and we believe this is due to a larger number of triangular faces from the boundary triangulation. 
Certainly, these computations are highly parallelizable. In addition, we show the numerical errors for $\beta^-=1$, $\beta^+=10$ and $\beta^-=1$, $\beta^+=100$ in the left two plots of Figure \ref{fig:H1exampCir}. 
Due to the geometric errors caused by coarse meshes, the convergence orders indicated on the graph are computed by incorporating only the errors from $N=60$ to $160$, but it clearly shows the asymptotic optimal convergence. 
Remarkably, the optimal convergence is even achieved for the $L^{\infty}$ norm which is a demanding property for interface problems.
\begin{table}[H]
\begin{center}
\resizebox{\textwidth}{!}{%
\begin{tabular}{|c |c | c|c | c|c | c| }
\hline
Total \# DoFs   & 1367631       & 1771561      & 2248091      & 2803221    &  3442951    & 4173281     \\ \hline
Interface \# DoFs    & 160926       & 191322      & 224682      & 259818    &  298866    & 339774           \\ \hline
Time(s) for projection & 2.96       & 4.00      & 4.88      & 37.92    &  5.65    & 7.13        \\ \hline
Time(s) for matrix assembling   & 14.21       & 18.14     & 21.14      & 30.47    &  37.58    & 42.61        \\ \hline
\end{tabular}
}
\end{center}
\caption{CPU time for computing the projections of IFE functions and the generation of stiffness matrices including the stabilization terms.}
\label{table:example1}
\end{table}

\begin{figure}[htp]
\centering
\includegraphics[width=0.35\textwidth]{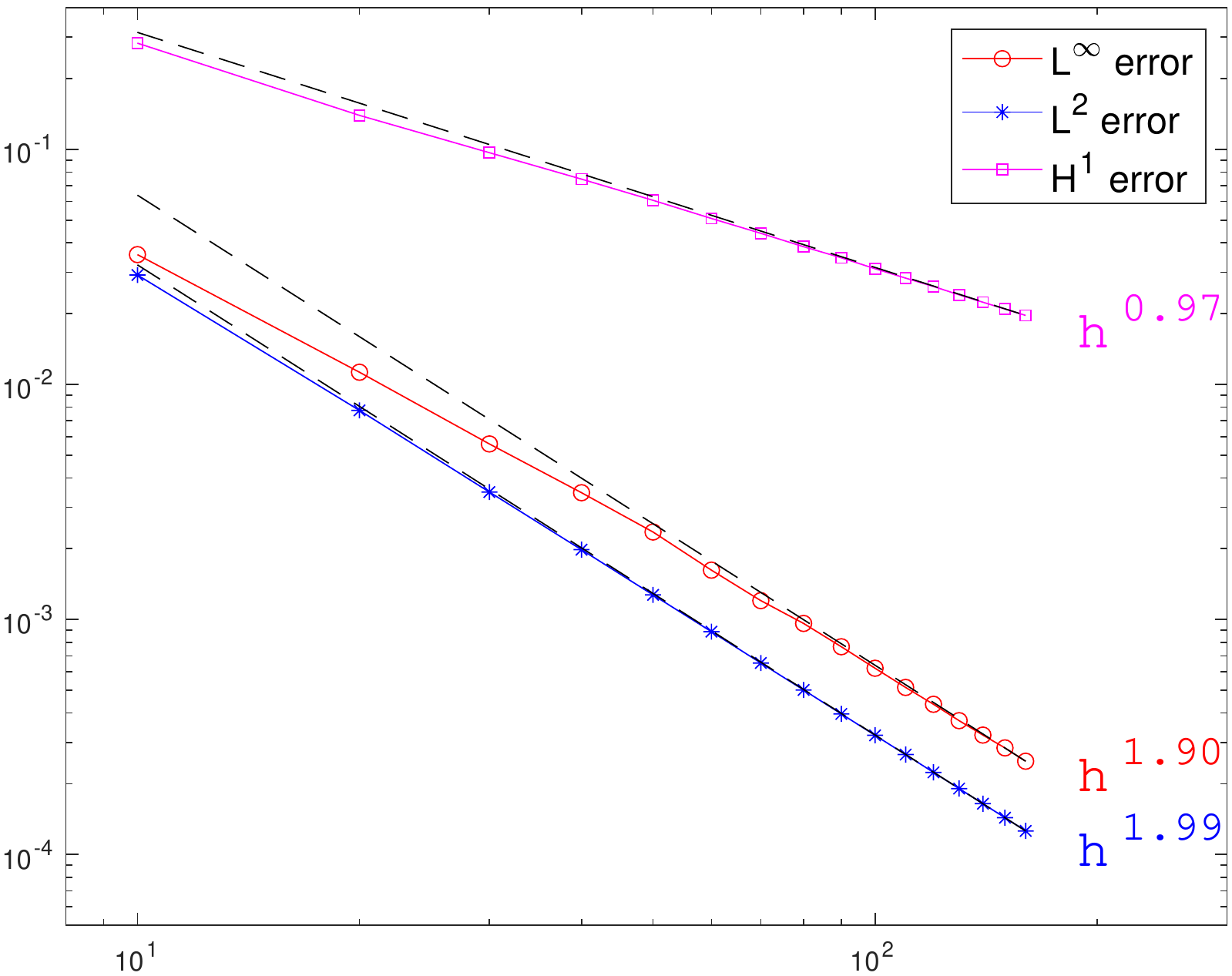}
\includegraphics[width=0.35\textwidth]{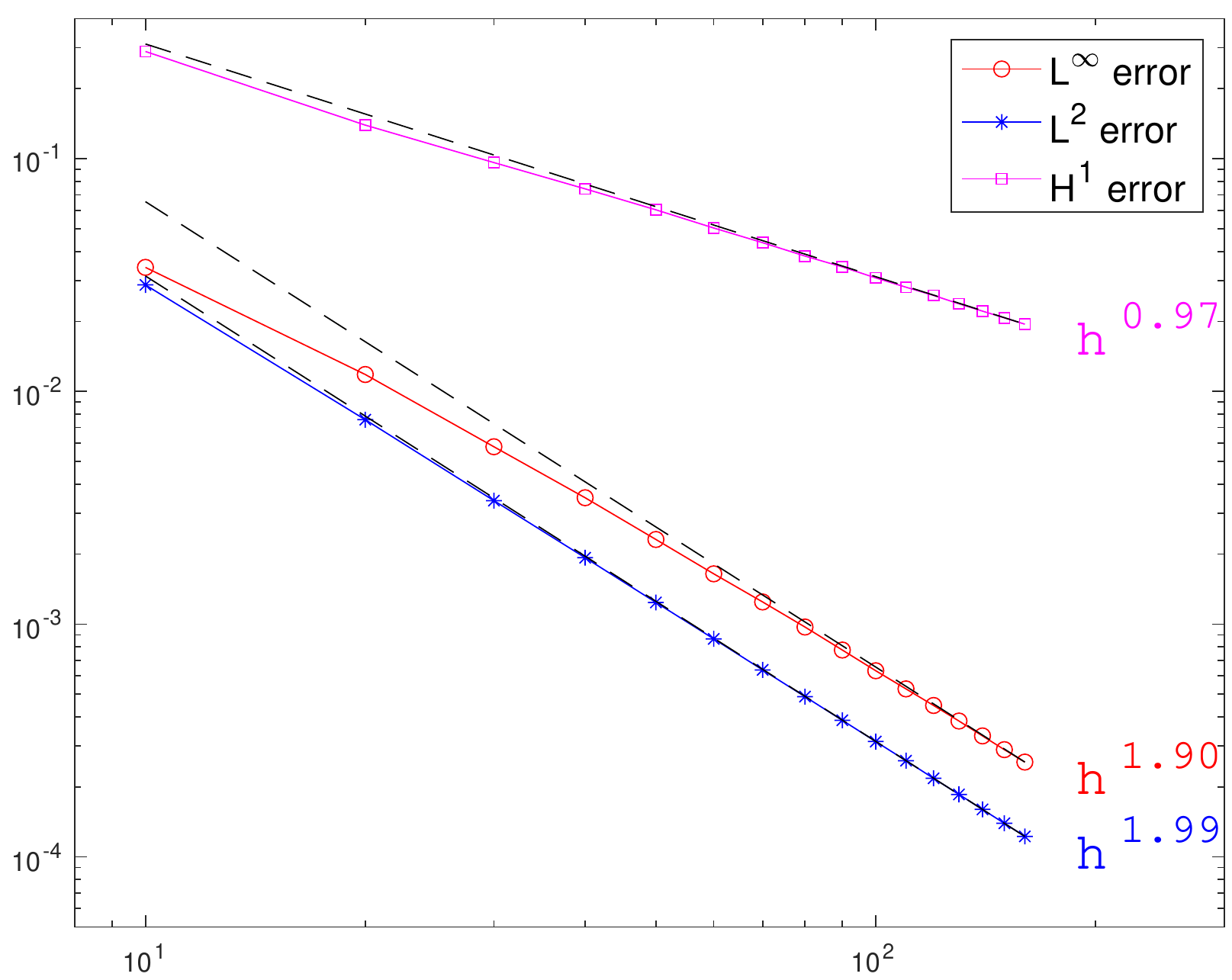}
\\
\includegraphics[width=0.35\textwidth]{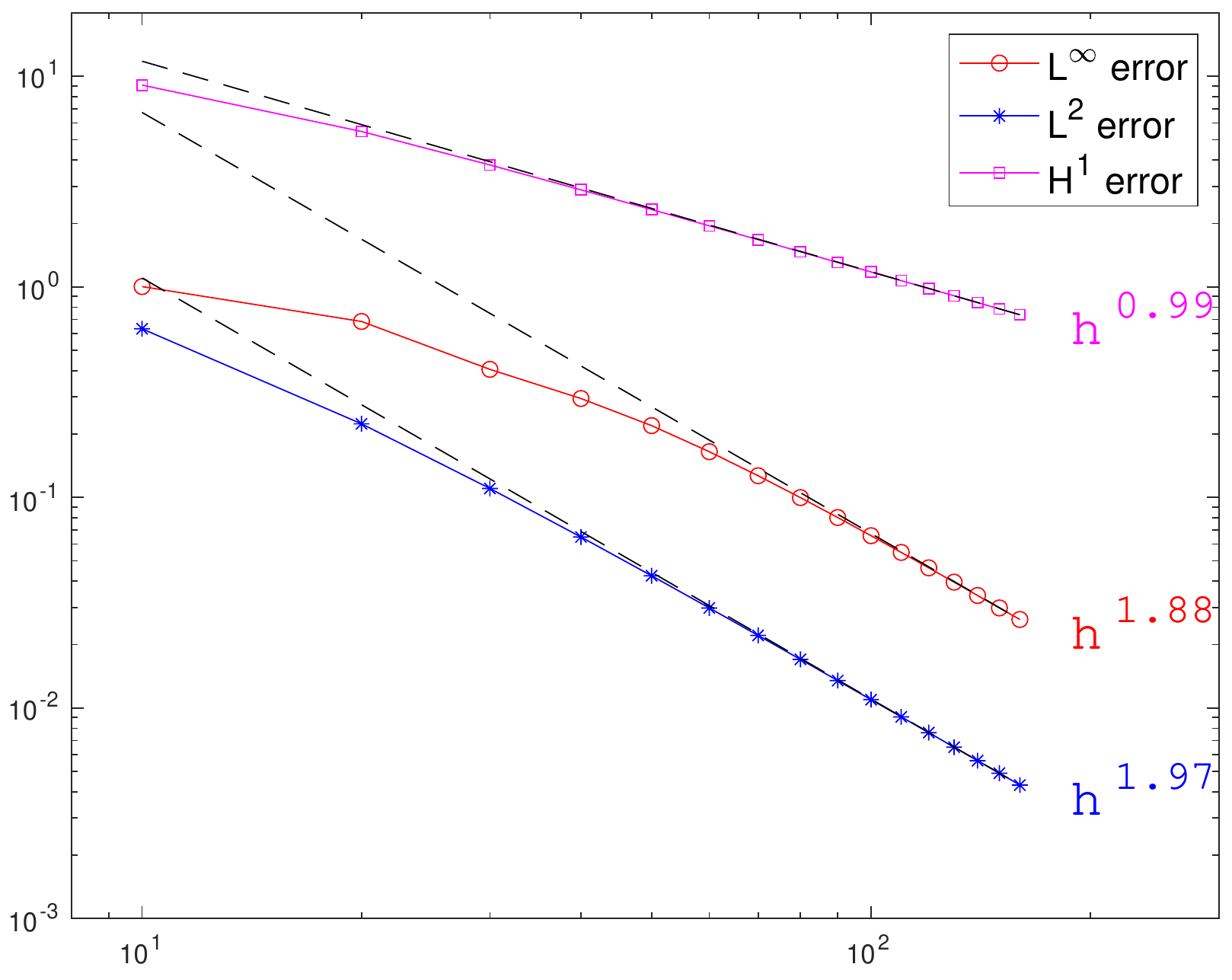}
\includegraphics[width=0.35\textwidth]{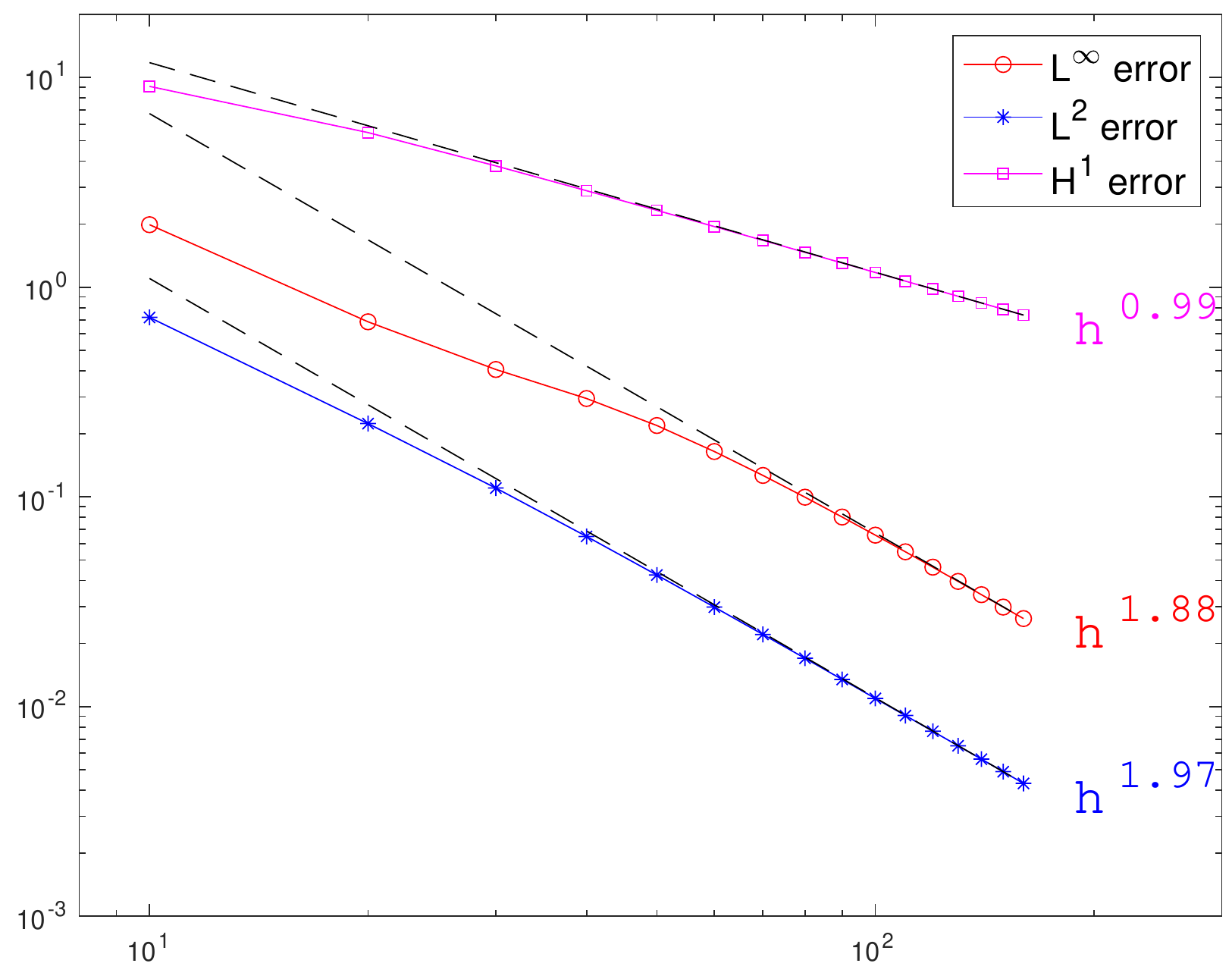}
\caption{Numerical errors and convergence order for the $H^1$ interface problem. The top two plots are for the spherical interface: $(\beta^-,\beta^+)=(1,10)$ and $(\beta^-,\beta^+)=(1,100)$, and bottom two plots are for the toroidal interface: $(\beta^-,\beta^+)=(1,10)$ and $(\beta^-,\beta^+)=(1,100)$. The black dashed lines indicate the expected $\mathcal{O}(h^2)$ convergence for the $L^{\infty}$ and $L^2$ errors and $\mathcal{O}(h)$ for the $H^1$ errors.}
\label{fig:H1exampCir}
\end{figure}

The second example concerns a more complicated interface shape that has two tori twisted with each other, see the right plot in Figure \ref{fig:interfaceExp}. 
The level-set functions of the two tori are $\phi_1 = (((x_1+0.3)^2+x_2^2)^{1/2}-0.2)^2+x_3^2-(\pi/5)^2$ and $\phi_2 = (((x_1- 0.3)^2+x_3^2)^{1/2}-0.2)^2+x_2^2-(\pi/5)^2$, 
and then the level-set function of this interface is given by $\phi(\bfx) = \min{( \phi_1(\bfx), \phi_2(\bfx) )}$.
The domain inside the two tori is $\Omega^-=\{ \bfx: \phi(\bfx) <0 \}$ and the outside one is $\Omega^+=\{ \bfx: \phi(\bfx) >0 \}$. The exact solution is given by
\begin{equation}
\label{H1example2}
u(\bfx) = 
\begin{cases}
      & 1,~~~~~~~~~~~~~~~~~~~~~~~~~    \text{if}~\bfx\in \Omega^-, \\
      & \cos( \phi_1(\bfx)\phi_2(\bfx) ),~~~~~~ \text{if}~\bfx\in \Omega^+.
\end{cases}
\end{equation}
In this case, the computational domain is $\Omega = (-1.3,1.3)^3$. 
The numerical solutions and errors are reported in the right two plots of Figure \ref{fig:H1exampCir}. As the interface has much larger curvature which requires the finer mesh to resolve. The convergence orders are estimated from the mesh size $N=60$ to $160$ which indicate the optimal convergence even for the $L^{\infty}$ errors.

\subsection{The $\bfH(\curl)$ interface problem}
\label{sec:hcurl-examples}

Now, let us consider the $\bfH(\curl)$ interface problem. 
It is known that solving the linear system from the Maxwell's equations is much more challenging. 
So here we first test the fast solver developed in Section~\ref{subsec:precond} for an extreme case that each interface element has small-cut subelements. 
For this purpose, we consider the domain $\Omega=(-1,1)^3$ with a flat interface $x_1=5\times 10^{-2-r}$. 
Fix the mesh size as $N=20$ and the parameters as $(\alpha^-,\alpha^+) = (1,10)$ and $(\beta^-,\beta^+) = (1,10)$. If $r=0$, i.e., $x_1=0.05$, the interface plane cuts all the interface elements exactly through the center and thus, each subelement has regular shape. 
In computation, we let $r=1,2,3,4$, i.e., the subelements on the left-side of the interface will become small accordingly. 
We report the condition numbers, the number of iterations and CPU time in Table \ref{table:exampleHcurlFlat}. 
We can see that small-cut interface elements can indeed make the conditioning worse, which may significantly increase the iteration numbers, see the results for $l=0$. 
However, the effect of small-cut interface elements can be successfully eliminated by the proposed block diagonal smoother $A_l$ in \eqref{blockeqn}. 
For this extreme case, $l=2$ seems sufficient to make the convergence completely independent of small subelements, but our numerical experience suggests that $l=1$ is good enough in general. 

\begin{table}[H]
\begin{center}
\resizebox{\textwidth}{!}{%
\begin{tabular}{|c|c |c | c|c | c|c | }
\hline
   & Interface location $r$  & 0       & 1      & 2      & 3    &  4     \\ \hline
    & Condition numbers  & $6.4\times10^5$       & $1.7\times10^7$      & $1.7\times10^9$      & $1.7\times10^{11}$    &  $1.7\times10^{13}$      \\ \hline
\multirow{2}{*}{$l=0$}   &\# iteration   & 44       & 53      & 107      & 327    &  842       \\\cline{2-7}
   &Time(s)   & 12    & 16       & 27      & 84      & 220          \\ \hline
   \multirow{2}{*}{$l=1$}   &\# iteration &   43   & 44       & 43      & 73      & 91          \\\cline{2-7}
   &Time(s)   & 11       & 11     & 11      & 21   &  24      \\ \hline
  \multirow{2}{*}{$l=2$}   &\# iteration   & 43       & 44      & 43      & 43    &  42         \\\cline{2-7}
   &Time(s)   & 11       & 12    & 12    &  11    & 12     \\ \hline
\end{tabular}
}
\end{center}
\caption{Condition numbers of the $\bfH(\curl)$ linear system with various interface location, and the related CPU time (in seconds) and \# iterations for the expanding width $l=0,1,2$, where $l=0$ means no block matrix used.}
\label{table:exampleHcurlFlat}
\end{table}
 
Next, we consider the spherical interface and slightly modify the benchmark example from Ref.~\refcite{2020GuoLinZou} of which the analytical solution is given by
\begin{equation}
\label{Hcurl_exactu_1}
\bfu = 
\begin{cases}
      & \frac{1}{\beta^- }\bfx +  \frac{1}{\alpha^- } n_1R_1(\bfx)[(x_2-x_3),(x_3-x_1),(x_1-x_2)]^{\top} ~~~~~~~~~~~~ \text{in}~ \Omega^-,  \\
      & \frac{1}{\beta^+ }\bfx +  \frac{1}{\alpha^+ }n_2R_1(\bfx)R_2(\bfx) [(x_2-x_3),(x_3-x_1),(x_1-x_2)]^{\top}  ~~~~ \text{in}~ \Omega^+,
\end{cases}
\end{equation}
where $\bfx = [x_1,x_2,x_3]^{\top}$ and $R_1(\bfx) = r^2_1 - \| \bfx \|^2$, $R_2(\bfx) = r^2_2 - \| \bfx \|^2$. The numerical experiment is carried on the meshes of $N=10,20,30,...,80$. In particular, the computational time and number of iterations are presented in Table \ref{table:exampleHcurl}. From the table, we can conclude that both the block-diagonal smoother and the HX preconditioner are important for reducing the iteration number for convergence. The results also show that the direct solver at each iteration does not cost significant computational time compared with the total cost of the iterative solver. Next, we report the numerical errors in both the $L^2$ and $\bfH(\curl)$ norms in the left two plots of \RG{Figure \ref{fig:HcurlexampCir}}, and the estimated convergence orders are also indicated in the plots which clearly demonstrate the optimality.

\begin{table}[H]
\begin{center}
\resizebox{\textwidth}{!}{%
\begin{tabular}{|c|c |c | c|c | c | c| c| c|}
\hline
   &Total \# DoFs          & 80554      & 244424      & 547074    &  1028452    & 1734626   &  2703384 &   3980338   \\ \hline
  \multirow{2}{*}{BD-PCG}   &\# iteration          & 962      & 1465      & 1932    &  2385    & 2828   &  3238    &      3661     \\\cline{2-9}
   &Time(s)          & 42.72    & 166.46    &  427.70    & 905.39   &  1693.95    &      2871.83   & 4542.63   \\ \hline
      \multirow{2}{*}{\makecell{BD-HX \\ $l=0$}}   &\# iteration          & 144      & 142      & 146    &  140    & 148   &  142    &      145     \\\cline{2-9}
   &Time(s)          & 38.89     & 69.09      & 117.97   &  214.89   & 401.81  &  533.28    &      808.19   \\ \hline
\multirow{2}{*}{\makecell{BD-HX \\ $l=1$}}   &\# iteration          & 75      & 76      & 81    &  77    & 80   &  83    &      90     \\\cline{2-9}
   &Time(s)         & 22.16      & 41.14      & 70.12    &  132.54    & 241.93   &  321.93    &      529.66     \\ \hline
\end{tabular}
}
\end{center}
\caption{CPU time and number of iterations for solving the $\bfH(\curl)$ linear system with the spherical interface and $(\alpha^-,\beta^-)=(1,1)$ and $(\alpha^+,\beta^+)=(100,200)$: block-diagonal HX (BD-HX) with $l=0$ and $1$ and the simple block-diagonal PCG (BD-PCG). The CPU time with respect to DoFs are approximatly $\mathcal{O}((\# \text{DoF})^{1.19})$, $\mathcal{O}((\# \text{DoF})^{0.80})$, $\mathcal{O}((\# \text{DoF})^{0.82})$ for BD-PCG, BD-HX($l=0$) and BD-HX($l=1$), respectively.}
\label{table:exampleHcurl}
\end{table}

\begin{figure}[h]
\centering
\includegraphics[width=0.35\textwidth]{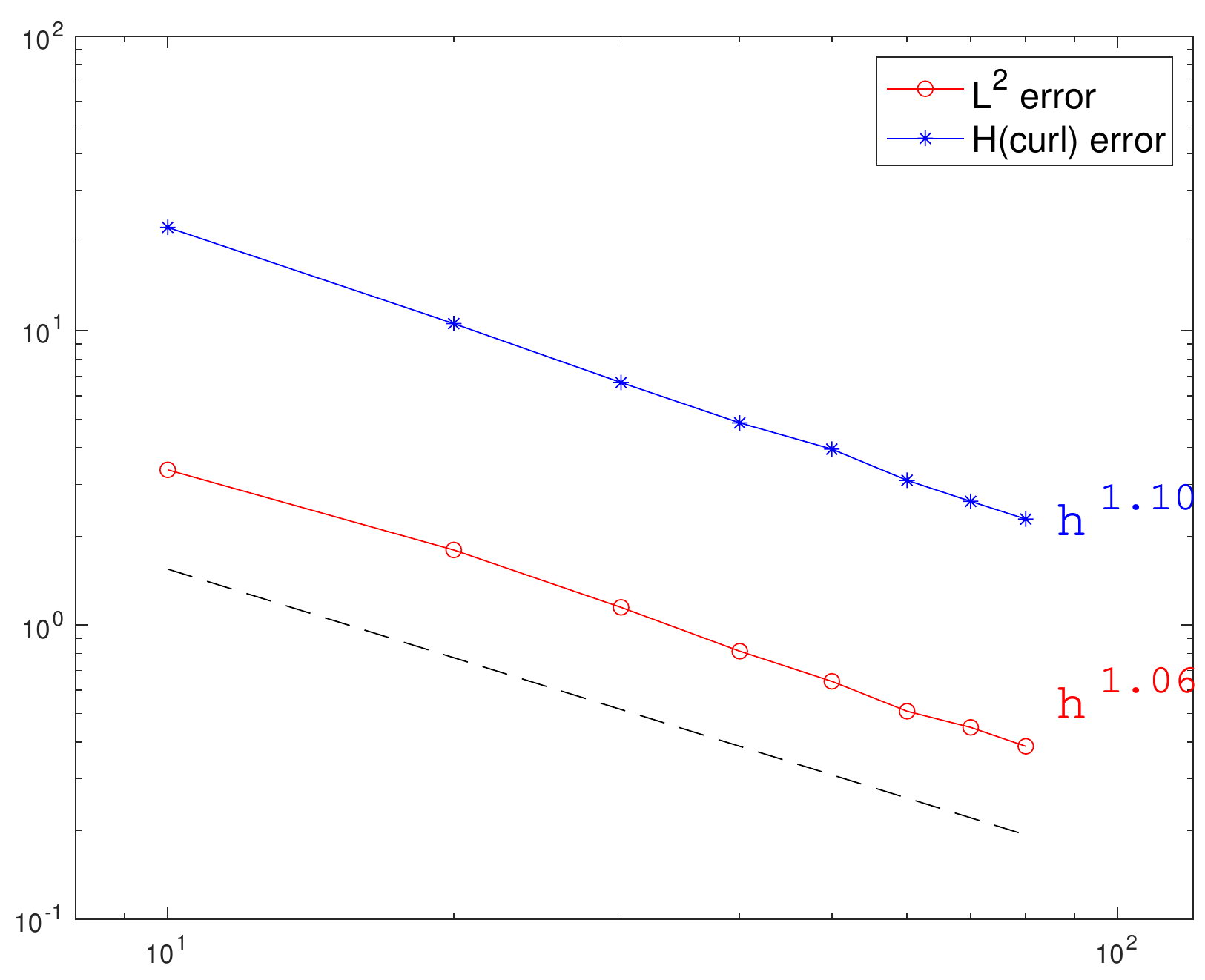}
\includegraphics[width=0.35\textwidth]{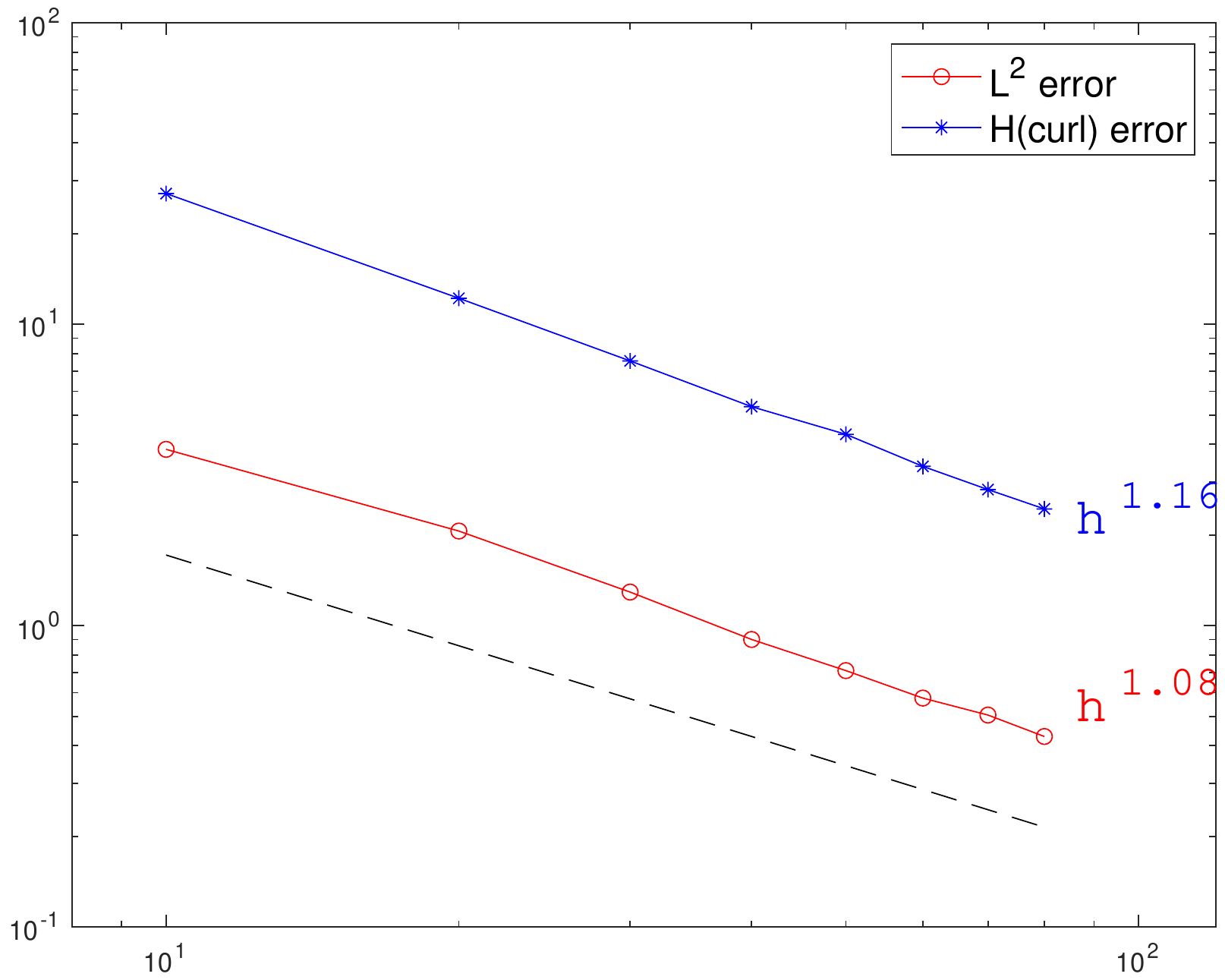}
\\
\includegraphics[width=0.35\textwidth]{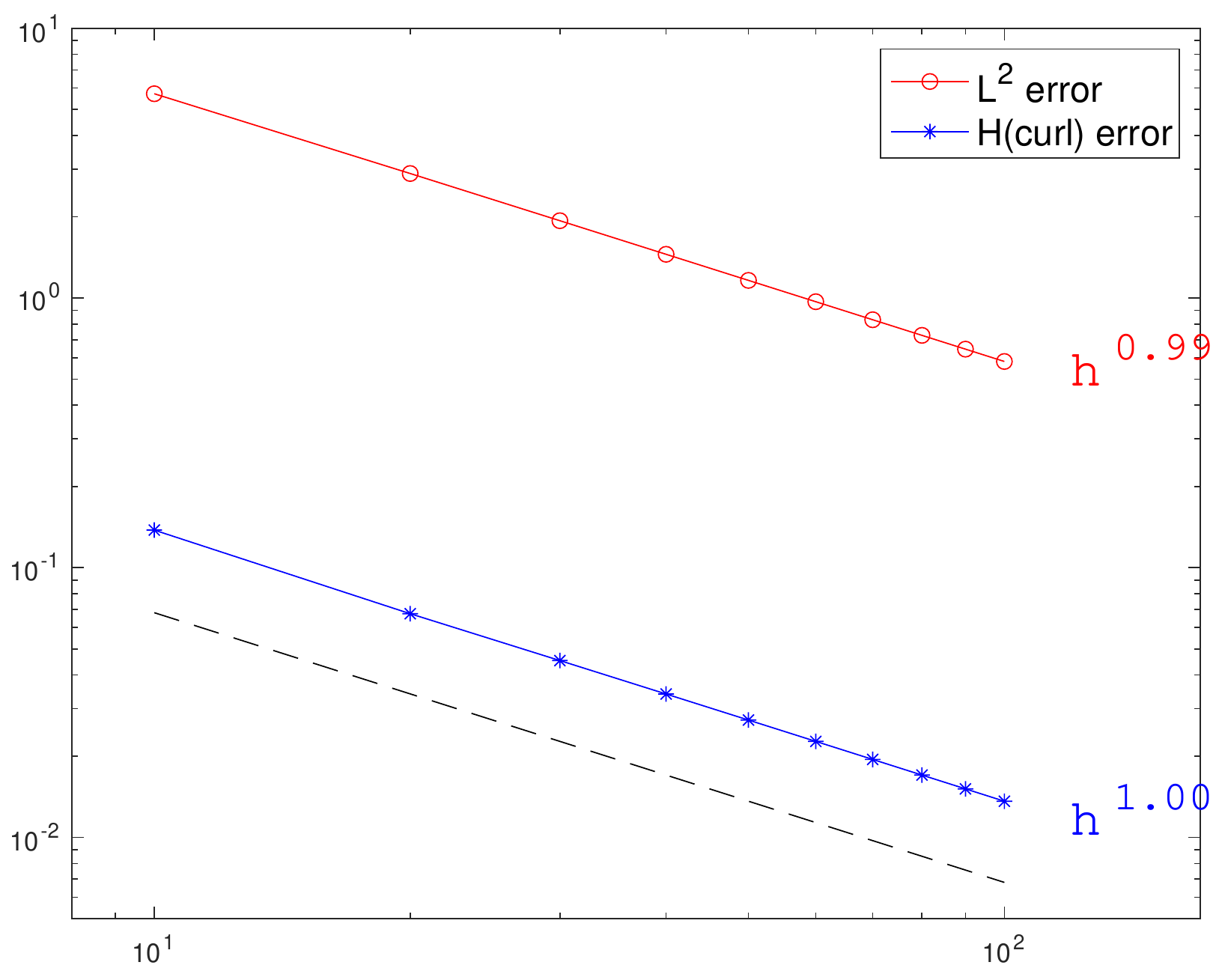}
\includegraphics[width=0.35\textwidth]{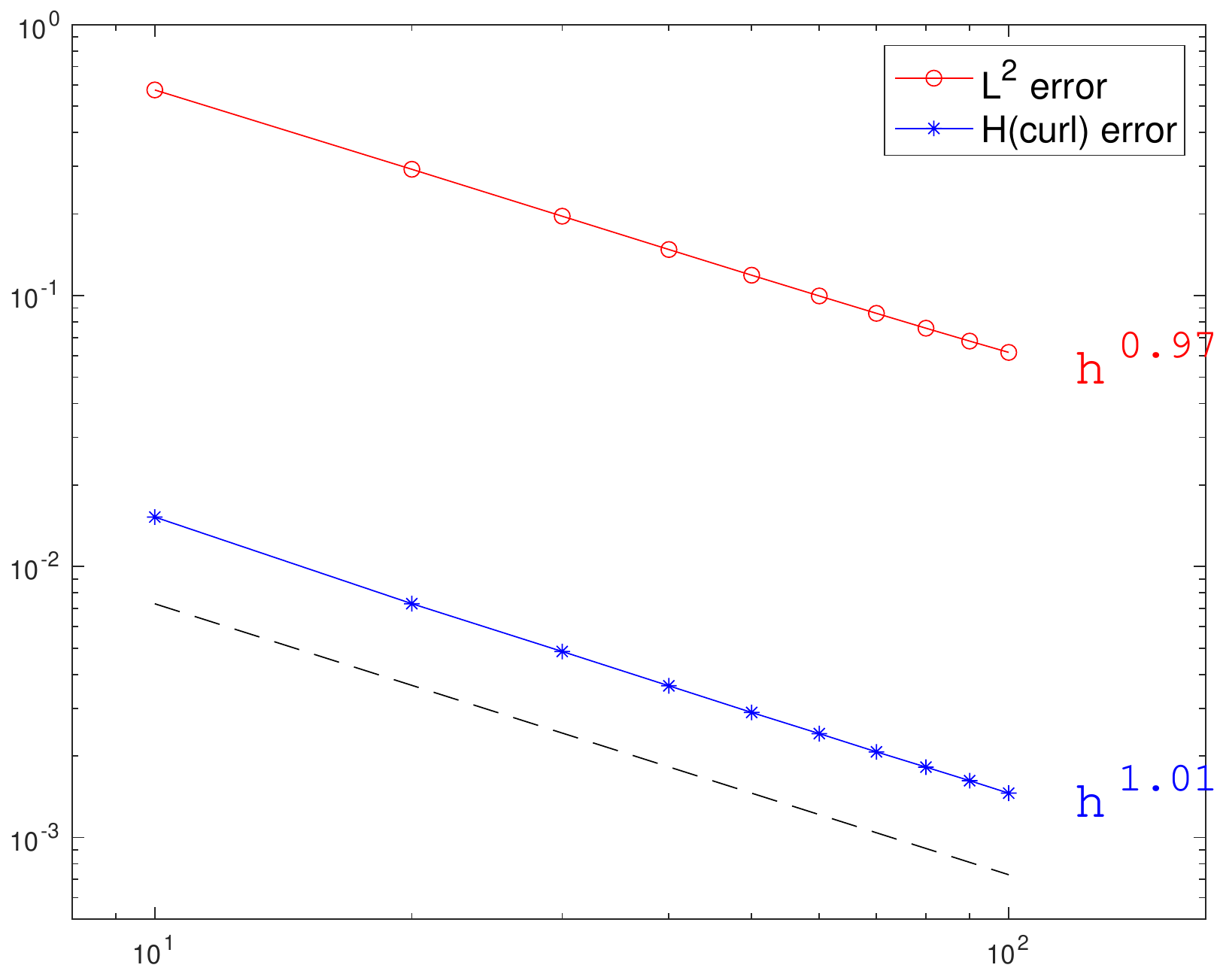}
\caption{Numerical errors and convergence order for the $\bfH(\curl)$ interface problem by the IVE method.
Top two: the first example. Bottom two: the second example.}
\label{fig:HcurlexampCir}
\end{figure}

\RG{IFE spaces can be also used in a dG-type scheme, i.e., penalties are used to handle 
the discontinuities across faces. This scheme works very well for $H^1$ interface problems
\cite{2020GuoLin,2020GuoZhang}, but results in only suboptimal convergence for $\bfH(\curl)$
problems. For the $\bfH(\curl)$ case, let us recall the scheme below. Let $\mathcal{F}^i_h$ 
be the collection of the interface faces. Let $\mathcal{S}_h$ be
a space containing IFE functions which may not be continuous on faces in $\mathcal{F}^i_h$.
Note that IFE functions can main tangential continuity on non-interface faces, and thus penalties
are only needed on interface faces. Then, the scheme is to find $\bfu_h\in\mathcal{S}_h$ such that
\begin{equation}
   \label{ppife_1}
   \tilde{a}_h(\bfu_h,\bfv_h) = \int_{\Omega} \bff\cdot\bfv_h \dd s
\end{equation}
where
\begin{equation}
   \label{ppife_2}
   \begin{aligned}
   \tilde{a}_h(\bfu_h,\bfv_h) := & \; (\alpha_h\curl\bfu_h, \curl \bfv_h)_{\Omega} + (\beta_h \bfu_h, \bfv_h)_{\Omega} 
   \\
   & + \sum_{F\in \mathcal{F}^i_h}\int_{F}\{\alpha_h \curl \bfu_h\}\cdot[\bfv_h\times\bfn] \dd s
   \\
   & + \sum_{F\in \mathcal{F}^i_h} \int_{F}\{\alpha_h \curl \bfv_h\}\cdot [\bfu_h\times\bfn] \dd s 
   \\
   & + \gamma  h^{-1} \sum_{F\in \mathcal{F}^i_h} \int_{F} [\bfu_h\times\bfn] [\bfv_h\times\bfn] \dd s, ~~~ \forall \bfv_h \in \mathcal{S}_h,
   \end{aligned}
\end{equation}
where $\gamma$ is a stabilization parameter which should be large enough and generally depends on $\alpha$.
We present the numerical results in Figure \ref{fig:Hcurlppife}. For the semi-$\bfH(\curl)$ norm,
we can clearly observe the sub-optimal convergence. The convergence under the $L^2$ norm deteriorates a little
as the mesh becomes finer. In some other setting, we can also observe much worse behavior for the $L^2$ norm.
\begin{figure}[h]
   \centering
   \includegraphics[width=1.6in]{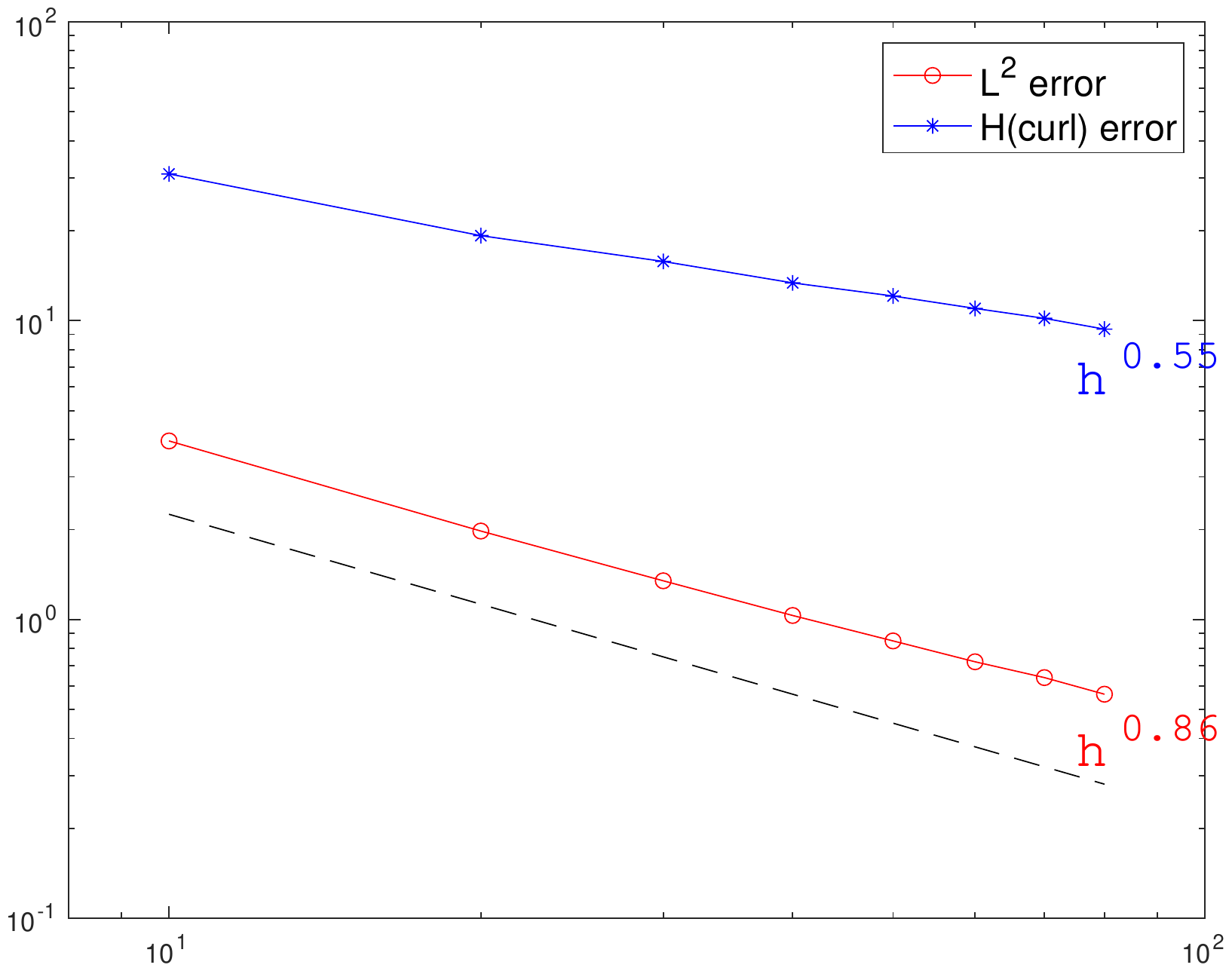}
   \includegraphics[width=1.6in]{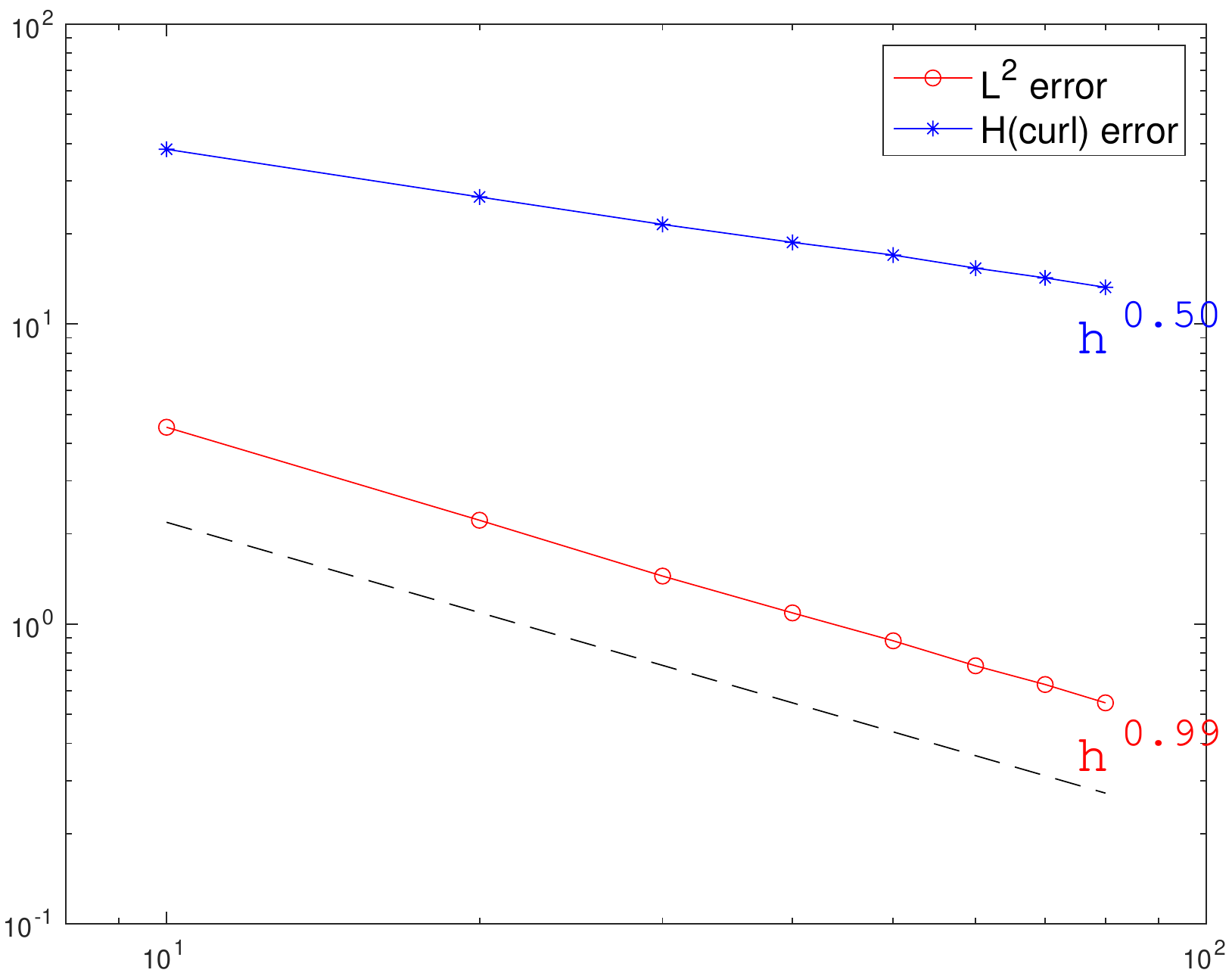}
   \caption{Numerical errors and convergence order for the $\bfH(\curl)$ interface problem by the penalty-type IFE method.
   Left two: the first example. Right two: the second example.}
   \label{fig:Hcurlppife}
   \end{figure}
}

In the second example, we also consider the twisted tori in the right plot of Figure \ref{fig:interfaceExp} on the domain $\Omega=(-1.3,1.3)^3$. To construct a function that satisfies the corresponding jump condition on the torus surface, we let $f(\bfx)=\phi_1(\bfx)\phi_2(\bfx) ( (x_1+0.3)^2+x_2^2 )( (x_1-0.3)^2+x_3^2 )$. Then, the exact solution is then defined as 
\begin{equation}
\label{Hcurl_exactu_2}
\bfu = \frac{1}{\beta} \nabla f(\bfx) + \frac{1}{\alpha} \cos( f(\bfx) )\bfv_0, ~~~~ \text{with} ~ \bfv_0 = [0,0,1]^{\top},
\end{equation}
where the boundary conditions and the source term are computed accordingly. The numerical results are reported in the right two plots of Figure \ref{fig:HcurlexampCir} which also shows the clear optimal convergence rate. These results demonstrate that the IVE method works well for complex surfaces.  

At last, we note that the penalty-type method cannot achieve optimal convergence numerically including both the IFE method~\cite{2020GuoLinZou} and the interface-penalty method~\cite{2016CasagrandeWinkelmannHiptmairOstrowski,2016CasagrandeHiptmairOstrowski}. 
Therefore, we believe the present method has distinguishing advantages in computational electromagnetism.

\begin{appendices}

\section{IFE Spaces on Complicated Geometry}
\label{sec:appen}

Here, we describe the IFE spaces for complicated interface element geometry, i.e., $\Gamma^K$ may have multiple components.
Let us assume that $\Gamma^K$ consists of the multiple components $\Gamma^{K,m}$, $m=1,2,...,M-1$, each of which is a simply-connected smooth surface. As $\Gamma$ is supposed not to intersect itself, $\Gamma^{K,m}$'s then do not intersect with one another. Without loss of generality, we assume the subelement containing $A_1$ is $K_1$ and $\Gamma^{K,1} = \partial K_1 \backslash \partial K$. Then, $K_m$ is the subelement bounded by $\partial K$, $\Gamma^{K,m}$ and $\Gamma^{K,m+1}$, $m=2,...,M-1$, and the remaining one is denoted by $K_M$. We show a 2D illustration of the geometry by the right plot in Figure \ref{fig:face_mesh}. The parameters associated with the subelement $K_m$ are denoted as $\alpha^m$ and $\beta^m$, which should take the values of $\alpha^-$ and $\alpha^+$ alternatively.

For each $\Gamma^{K,m}$, we let $\Gamma^{K,m}_h$ be its planar approximation. 
Similarly, define the subelement containing $A_1$ as $K_{h,1}$, and the others, i.e., $K_{h,2}$,...,$K_{h,M}$, are defined in a similar manner as their counterparts $K_1$,...,$K_M$. Note that each of $K_{h,1}$,...,$K_{h,M}$ is a polyhedron. Let $\alpha_h$ and $\beta_h$ be the piecewise constant functions defined on these polyhedral subelements, and denote $\alpha_{h,m}=\alpha_h|_{K_{h,m}}$ and $\beta_{h,m}=\beta_h|_{K_{h,m}}$, $m=1,...,M$.

Similar to \eqref{AB_space2}, we are able to derive explicit formulas for the functions in the spaces \eqref{AB_space1}. For each linear interface component $\Gamma^{K}_h$, we further let $\bar{\bft}^1_{m}$ and $\bar{\bft}^2_{m}$ be the two orthogonal tangential unit vectors to $\Gamma^{K,m}_h$, and denote the matrix $T_m=[\bar{\bfn}_{m}, \bar{\bft}^1_{m},\bar{\bft}^2_{m}]$. Then define the transformation matrices: 
\begin{equation}
\label{lem_ABeigen_eq1_new}
M^{f,c_h}_{K,m} = T \left[\begin{array}{ccc} 1 & 0 & 0 \\ 0 & \tilde{c}_m & 0 \\ 0 & 0 & \tilde{c}_m \end{array}\right] (T_m)^{\top} ~~~\text{and} ~~~ 
M^{e,c_h}_{K,m} = T \left[\begin{array}{ccc} \tilde{c}_m & 0 & 0 \\ 0 & 1 & 0 \\ 0 & 0 & 1 \end{array}\right] (T_m)^{\top},
\end{equation}
where $\tilde{c}_m=c_{h,m}/c_{h,m+1}$, with $m=1,2,...,M-1$, and define the spaces $\bfP^e_h(c_h;K)$ and $\bfP^f_h(c_h;K)$ as
\begin{subequations}
\label{AB_space2_new}
\begin{align}
      \bfP^e_h(c_h;K) = \{ & \bfc: \bfc_m= \bfc|_{K_{h,m}} \in \mathcal{P}_0(K_{h,m}), ~ m=1,...,M, \nonumber \\ 
      &  \bfc_{m+1} = M^{e,c_h}_{K,m} \bfc_m,  ~ m=1,...,M-1 \}, 
      \\
      \bfP^f_h(c_h;K) = \{  & \bfc: \bfc_m= \bfc|_{K_{h,m}} \in \mathcal{P}_0(K_{h,m}), ~ m=1,...,M, \nonumber \\  
      & \bfc_{m+1} = M^{f,c_h}_{K,m} \bfc_m,  ~ m=1,...,M-1 \}. 
\end{align}
\end{subequations}
Again, the constant vectors at different cut regions are related by the jump conditions and thus the dimension of both $\bfP^e_h(c_h;K)$ and $\bfP^f_h(c_h;K)$ is also $3$. In this case, the formulas of IFE functions are slightly more complicated which are presented in the following lemma.

\begin{lemma}
\label{lem_IFEfun}
Let $\bfx_{K,m}$ be an arbitrary point at $\Gamma^{K,m}_h$. Let $\bfa=\bfa_m$ and $\bfb=\bfb_m$ be two arbitrary vectors in $\bfP^f_h(a_h;K)$ and $\bfP^e_h(b_h;K)$, and let $c$ be an arbitrary constant. %
Then, the formulas for the functions in $S^n_h(b_h;K)$, $\bfS^e_h(a_h,b_h;K)$ and $\bfS^f_h(a_h;K)$, respectively, are
\begin{equation}
\label{lem_IFEfun_H1}
v^n_h = 
\begin{cases}
      & \bfb_m\cdot(\bfx-\bfx_{K,1}) + c  ~~~~~~~~~~~~~~~~~~~~~~~~~~~~~~~~~~~~~~~~~~~ \text{in} ~ K_{h,m}, ~ m =1,2, \\
      & \bfb_m\cdot(\bfx-\bfx_{K,m-1}) + c + \sum_{l=2}^{m-1} \bfb_{l}\cdot(\bfx_{K,l} - \bfx_{K,l-1} ), ~~~ \text{in} ~ K_{h,m}, ~ m \geqslant 3,
\end{cases}
\end{equation}
\begin{equation}
\label{lem_IFEfun_Hcurl}
\bfv^e_h = 
\begin{cases}
      & \bfa_m\times(\bfx-\bfx_{K,1}) + \bfb_m ~~~~~~~~~~~~~~~~~~ \text{in} ~ K_{h,m}, ~ m =1,2, \\
      & \bfa_m\times(\bfx-\bfx_{K,m-1}) + \bfb_m + \bfxi_m, ~~~~~~ \text{in} ~ K_{h,m}, ~ m \geqslant 3,
\end{cases}
\end{equation}
with 
$$ \bfxi_m :=  \sum_{l=2}^{m-1} ( M^{e,b_h}_{K,m-1} \cdots  M^{e,b_h}_{K,l} ) \left[ \bfa_{l}\times(\bfx_{K,l} - \bfx_{K,l-1} ) \right],$$ 
and 
\begin{equation}
\label{lem_IFEfun_Hdiv}
\bfv^f_h = 
\begin{cases}
      & c(\bfx-\bfx_{K,1}) + \bfa_m ~~~~~~~~~~~~~~~\text{in} ~ K_{h,m}, ~ m =1,2, \\
      & c(\bfx-\bfx_{K,m-1}) + \bfa_m + \bfeta_m, ~~~ \text{in} ~ K_{h,m}, ~ m \geqslant 3,
\end{cases}
\end{equation}
with 
$$
\begin{aligned}
\bfeta_m & :=  c \sum_{l=2}^{m-1} ( M^{f,a_h}_{K,m-1} \cdots  M^{f,a_h}_{K,l} ) (\bfx_{K,l} - \bfx_{K,l-1} ) 
\\ 
&=  c  \sum_{l=2}^{m-1} (\prod^{l}_{n=m-1} M^{f,a_h}_{K,n} )(\bfx_{K,l} - \bfx_{K,l-1} ). 
\end{aligned}
$$
The formed IFE spaces also have the dimension $4, 6,$ and $4$ for the $H^1$, $\bfH(\curl)$ and $\bfH(\div)$ cases, respectively.
\end{lemma}
\begin{proof}
One can directly verify that these piecewisely-defined functions satisfy the corresponding jump conditions shown in Table \ref{tab_IFE_func} but on each $\Gamma^{K,m}_h$. The dimension can be simply counted by the number of free variables of $\bfa$, $\bfb$ and $c$ in the formulas above.
\end{proof}

Note that in Lemma \ref{lem_IFEfun}, the points $\{\bfx_{K,m}\}_{m=1}^{M-1}$ should be chosen and fixed.

\end{appendices}

\section*{Acknowledgment}
This work was supported in part by the National Science Foundation under grants DMS-1913080, DMS-2012465, and DMS-2136075. We would like to thank the anonymous reviewers for the suggestions on improving this article. We would like to thank Dr. Xuehai Huang (Shanghai University of Finance and Economics) for his numerous advices on revising the article.


\begin{thebibliography}{10}
\newcommand{\enquote}[1]{#1}

\bibitem{1999AcostaRicardo}
G.~Acosta and R.~G. Dur{\'a}n, \enquote{The maximum angle condition for mixed
  and nonconforming elements: Application to the stokes equations}, {\it SIAM
  J. Numer. Anal.} \textbf{37} (1999) 18--36.

\bibitem{2020AdjeridBabukaGuoLin}
S.~Adjerid, I.~Babu{\v s}ka, R.~Guo and T.~Lin, \enquote{An enriched immersed
  finite element method for interface problems with nonhomogeneous jump
  conditions}, {\it Comput. Methods Appl. Mech. Engrg.} \textbf{404} (2023)
  115770.

\bibitem{2015AdjeridChaabaneLin}
S.~Adjerid, N.~Chaabane and T.~Lin, \enquote{{An immersed discontinuous finite
  element method for Stokes interface problems}}, {\it Comput. Methods Appl.
  Mech. Engrg.} \textbf{293} (2015) 170--190, in press.

\bibitem{2013AhmadAlsaediBrezziMariniRusso}
B.~Ahmad, A.~Alsaedi, F.~Brezzi, L.~Marini and A.~Russo, \enquote{Equivalent
  projectors for virtual element methods}, {\it Comput. Math. Appl.}
  \textbf{66} (2013) 376--391.

\bibitem{1998AmroucheBernardiDaugeGirault}
C.~Amrouche, C.~Bernardi, M.~Dauge and V.~Girault, \enquote{Vector potentials
  in three-dimensional non-smooth domains}, {\it Math. Meth. Appl. Sci.}
  \textbf{21} (1998) 823--864.

\bibitem{1997ARNOLDRICHARDWINTHER}
D.~N. Arnold, R.~S. Falk and R.~Winther, \enquote{Preconditioning in
  ${H}(\rm{div})$ and applications}, {\it Math. Comp.} \textbf{66} (1997)
  957--984.

\bibitem{ArnoldFalkEtAl2000Multigrid}
D.~N. Arnold, R.~S. Falk and R.~Winther, \enquote{Multigrid in ${H}(\rm{div})$
  and ${H}(\rm{curl})$}, {\it Numer. Math.} \textbf{85} (2000) 197--217.

\bibitem{1970Babuska}
I.~Babu{\v{s}}ka, \enquote{{The finite element method for elliptic equations
  with discontinuous coefficients}}, {\it Computing (Arch. Elektron. Rechnen)}
  \textbf{5} (1970) 207--213.

\bibitem{1976BabuskaAziz}
I.~Babu{\v s}ka and A.~K. Aziz, \enquote{On the angle condition in the finite
  element method}, {\it SIAM J. Numer. Anal.} \textbf{13} (1976) 214--226.

\bibitem{1994BabuskaCalozOsborn}
I.~Babu{\v{s}}ka, G.~Caloz and J.~E. Osborn, \enquote{{Special finite element
  methods for a class of second order elliptic problems with rough
  coefficients}}, {\it SIAM J. Numer. Anal.} \textbf{31} (1994) 945--981.

\bibitem{1983BabuskaOsborn}
I.~Babu{\v{s}}ka and J.~E. Osborn, \enquote{{Generalized finite element
  methods: their performance and their relation to mixed methods}}, {\it SIAM
  J. Numer. Anal.} \textbf{20} (1983) 510--536.

\bibitem{2022BarkerCaoSternnonconforming}
M.~Barker, S.~Cao and A.~Stern, \enquote{A nonconforming primal hybrid finite
  element method for the two-dimensional vector {L}aplacian}, {\it arXiv
  preprint arXiv:2206.10567} .

\bibitem{BeiraodaVeigaBrezziDassiEtAl2018Lowest}
L.~{Beir{\~a}o da Veiga}, F.~Brezzi, F.~Dassi, L.~Marini and A.~Russo,
  \enquote{Lowest order virtual element approximation of magnetostatic
  problems}, {\it Comput. Methods Appl. Mech. Engrg.} \textbf{332} (2018)
  343--362.

\bibitem{Beirao-da-Veiga;Brezzi;Cangiani;Manzini:2013principles}
L.~Beir{\~a}o~da Veiga, F.~Brezzi, A.~Cangiani, G.~Manzini, L.~Marini and
  A.~Russo, \enquote{{Basic principles of virtual element methods}}, {\it Math.
  Models Methods Appl. Sci.} \textbf{23} (2013) 199--214.

\bibitem{2017VeigaBrezziDassiMarini}
L.~Beir{\~a}o~da Veiga, F.~Brezzi, F.~Dassi, L.~Marini and A.~Russo,
  \enquote{{Virtual Element approximation of 2D magnetostatic problems}}, {\it
  Comput. Methods Appl. Mech. Engrg.} \textbf{327} (2017) 173--195.

\bibitem{2016VeigaBrezziMarini}
L.~Beir{\~a}o~da Veiga, F.~Brezzi, L.~D. Marini and A.~Russo,
  \enquote{{\(H({\text{div}})\) and \(H(\mathbf{curl})\)-conforming virtual
  element methods}}, {\it Numer. Math.} \textbf{133} (2016) 303--332.

\bibitem{BEIRAODAVEIGA2021}
L.~Beir{\~a}o~da Veiga, F.~Dassi, G.~Manzini and L.~Mascotto, \enquote{Virtual
  elements for {M}axwell's equations}, {\it Comput. Math. Appl.}  (2021)
  82--99.

\bibitem{beirao2017stability}
L.~Beir{\~a}o~da Veiga, C.~Lovadina and A.~Russo, \enquote{{Stability analysis
  for the virtual element method}}, {\it Math. Models Methods Appl. Sci.}
  \textbf{27} (2017) 2557--2594.

\bibitem{2023BeiraoLiuNonconforming}
L.~Beir{\~a}o~da Veiga, Y.~Liu, L.~Mascotto and A.~Russo, \enquote{The nonconforming virtual element method with curved edges}, {\it arXiv preprint arXiv:2303.15204
}.

\bibitem{2020BeiroMascotto}
L.~Beir{\~a}o~da Veiga and L.~Mascotto, \enquote{Interpolation and stability
  properties of low order face and edge virtual element spaces}, {\it IMA J.
  Numer. Anal.} \textbf{drac008}.

\bibitem{2023BeiroMascotto}
L.~Beir{\~a}o~da Veiga and L.~Mascotto, \enquote{Stability and interpolation properties of serendipity nodal virtual elements}, {\it Appl. Math. Lett.} \textbf{42} (2023) 108639.

\bibitem{2001BenBuffaMaday}
F.~Ben~Belgacem, A.~Buffa and Y.~Maday, \enquote{The mortar finite element
  method for 3{D} {M}axwell equations: First results}, {\it SIAM J. Numer.
  Anal.} \textbf{39} (2001) 880--901.

\bibitem{Bordas_Burman_Larson_Olshanskii2017}
S.~Borda, E.~Burman, M.~Larson and M.~O. (Editors), {\it Geometrically Unfitted
  Finite Element Methods and Applications, Proceedings of the UCL Workshop
  2016}, volume 121 of {\it Lecture Notes in Computational Science and
  Engineering} (Springer, 2017).

\bibitem{1996BrauerRuehl}
J.~R. Brauer, J.~J. Ruehl, M.~A. Juds, M.~J.~V. Heiden and A.~A. Arkadan,
  \enquote{{Dynamic stress in magnetic actuator computed by coupled structural
  and electromagnetic finite elements}}, {\it IEEE Trans. Magn.} \textbf{32}
  (1996) 1046 -- 1049.

\bibitem{2018BrennerSung}
S.~Brenner and L.-Y. Sung, \enquote{Virtual element methods on meshes with
  small edges or faces}, {\it Math. Models Methods Appl. Sci.} \textbf{28}
  (2018) 1291--1336.

\bibitem{2008BrennerCuiLiSung}
S.~C. Brenner, J.~Cui, F.~Li and L.~Y. Sung, \enquote{A nonconforming finite
  element method for a two-dimensional curl--curl and grad-div problem}, {\it
  Numer. Math.} \textbf{109} (2008) 509--533.

\bibitem{2005BuffaCostabelDauge}
A.~Buffa, M.~Costabel and M.~Dauge, \enquote{{Algebraic convergence for
  anisotropic edge elements in polyhedral domains}}, {\it Numer. Math.}
  \textbf{101} (2005) 29--65.

\bibitem{2015BurmanClaus}
E.~Burman, S.~Claus, P.~Hansbo, M.~G. Larson and A.~Massing, \enquote{{CutFEM:
  Discretizing geometry and partial differential equations}}, {\it Internat. J.
  Numer. Methods Engrg.} \textbf{104} (2015) 472--501.

\bibitem{Cao;Chen:2018Anisotropic}
S.~Cao and L.~Chen, \enquote{{Anisotropic error estimates of the linear virtual
  element method on polygonal meshes}}, {\it SIAM J. Math. Anal.} \textbf{56}
  (2018) 2913--2939.

\bibitem{Cao;Chen:2018AnisotropicNC}
S.~Cao and L.~Chen, \enquote{Anisotropic error estimates of the linear
  nonconforming virtual element methods}, {\it SIAM J. Numer. Anal.}
  \textbf{57} (2019) 1058--1081.

\bibitem{2021CaoChenGuo}
S.~Cao, L.~Chen and R.~Guo, \enquote{A virtual finite element method for two
  dimensional {M}axwell interface problems with a background unfitted mesh},
  {\it Math. Models Methods Appl. Sci.}  (2021) 2907--2936.

\bibitem{2021CaoChenGuoIVEM}
S.~Cao, L.~Chen, R.~Guo and F.~Lin, \enquote{Immersed virtual element methods
  for elliptic interface problems}, {\it J. Sci. Comput.} \textbf{93} (2022)
  1--41.

\bibitem{2016CasagrandeHiptmairOstrowski}
R.~Casagrande, R.~Hiptmair and J.~Ostrowski, \enquote{{An a priori error
  estimate for interior penalty discretizations of the Curl-Curl operator on
  non-conforming meshes}}, {\it J. Math. Ind.} \textbf{6} (2016) 4.

\bibitem{2016CasagrandeWinkelmannHiptmairOstrowski}
R.~Casagrande, C.~Winkelmann, R.~Hiptmair and J.~Ostrowski, \enquote{{DG
  Treatment of Non-conforming Interfaces in 3D Curl-Curl Problems}}, in {\it
  Scientific Computing in Electrical Engineering} (Springer International
  Publishing, Cham, 2016), pp. 53--61.

\bibitem{Chen:2008ifem}
L.~Chen, \enquote{$i${FEM}: an integrated finite element methods package in
  {MATLAB}}, Technical report, University of California at Irvine, 2009.

\bibitem{2022ChenGuoZoufamily}
L.~Chen, R.~Guo and J.~Zou, \enquote{{A family of immersed finite element
  spaces and applications to three dimensional H(curl) interface problems}},
  {\it arXiv preprint arXiv:2205.14127} .

\bibitem{2007ChenHolstXu}
L.~Chen, M.~Holst and J.~Xu, \enquote{The finite element approximation of the
  nonlinear {P}oisson--{B}oltzmann equation}, {\it SIAM J. Numer. Anal.}
  \textbf{45} (2007) 2298--2320.

\bibitem{ChenHuang2022Finite}
L.~Chen and X.~Huang, \enquote{Finite element {de Rham} and {Stokes} complexes
  in three dimensions}, {\it arXiv preprint arXiv:2206.09525} .

\bibitem{2017ChenWeiWen}
L.~Chen, H.~Wei and M.~Wen, \enquote{{An interface-fitted mesh generator and
  virtual element methods for elliptic interface problems}}, {\it J. Comput.
  Phys.} \textbf{334} (2017) 327--348.

\bibitem{ChenWuEtAl2018Multigrid}
L.~Chen, Y.~Wu, L.~Zhong and J.~Zhou, \enquote{Multigrid preconditioners for
  mixed finite element methods of the vector {L}aplacian}, {\it J. Sci.
  Comput.} \textbf{77} (2018) 101--128.

\bibitem{2000ChenDuZou}
Z.~Chen, Q.~Du and J.~Zou, \enquote{Finite element methods with matching and
  nonmatching meshes for {M}axwell equations with discontinuous coefficients},
  {\it SIAM J. Numer. Anal.} \textbf{37} (2000) 1542--1570.

\bibitem{2015ChenwuXiao}
Z.~Chen, Z.~Wu and Y.~Xiao, \enquote{An adaptive immersed finite element method
  with arbitrary {L}agrangian-{E}ulerian scheme for parabolic equations in time
  variable domains}, {\it Int. J. Numer. Anal. Mod.}  (2015) 567--591.

\bibitem{2009ChenXiaoZhang}
Z.~Chen, Y.~Xiao and L.~Zhang, \enquote{{The adaptive immersed interface finite
  element method for elliptic and Maxwell interface problems}}, {\it J. Comput.
  Phys.} \textbf{228} (2009) 5000--5019.

\bibitem{1998ChenZou}
Z.~Chen and J.~Zou, \enquote{{Finite element methods and their convergence for
  elliptic and parabolic interface problems}}, {\it Numer. Math.} \textbf{79}
  (1998) 175--202.

\bibitem{2010ChuGrahamHou}
C.-C. Chu, I.~G. Graham and T.-Y. Hou, \enquote{{A new multiscale finite
  element method for high-contrast elliptic interface problems}}, {\it Math.
  Comp.} \textbf{79} (2010) 1915--1955.

\bibitem{1996ColtonKress}
D.~Colton and R.~Kress, {\it Inverse Acoustic and Electromagnetic Scattering
  Theory} (Springer, NY, 1996).

\bibitem{1990Costabel}
M.~Costabel, \enquote{A remark on the regularity of solutions of maxwell's
  equations on lipschitz domains}, {\it Mathematical Methods in the Applied
  Sciences} \textbf{12} (1990) 365--368.

\bibitem{CostabelDaugeEtAl1999Singularities}
M.~Costabel, M.~Dauge and S.~Nicaise, \enquote{Singularities of {M}axwell
  interface problems}, {\it ESAIM: Mathematical Modelling and Numerical
  Analysis} \textbf{33} (1999) 627--649.

\bibitem{2021DassiFumagalliLosapio}
F.~Dassi, A.~Fumagalli, D.~Losapio, S.~Scial{\`o}, A.~Scotti and G.~Vacca,
  \enquote{The mixed virtual element method on curved edges in two dimensions},
  {\it Comput. Methods Appl. Mech. Engrg.} \textbf{386} (2021) 114098.

\bibitem{2014ErcanJaewook}
E.~M. Dede, J.~Lee and T.~Nomura, {\it {Multiphysics Simulation:
  Electromechanical System Applications and Optimization}} (Springer, 2014).

\bibitem{FernandesGilardi1997Magnetostatic}
P.~Fernandes and G.~Gilardi, \enquote{Magnetostatic and electrostatic problems
  in inhomogeneous anisotropic media with irregular boundary and mixed boundary
  conditions}, {\it Math. Models Methods Appl. Sci.} \textbf{7} (1997)
  957--991.

\bibitem{2011GiraultRaviart}
V.~Girault and P.-A. Raviart, {\it {Finite Element Methods for Navier-Stokes
  Equations: Theory and Algorithms}} (Springer Publishing Company, 2011), 1st
  edition.

\bibitem{2007GongLiLi}
Y.~Gong, B.~Li and Z.~Li, \enquote{{Immersed-interface finite-element methods
  for elliptic interface problems with nonhomogeneous jump conditions}}, {\it
  SIAM J. Numer. Anal.} \textbf{46} (2008) 472--495.

\bibitem{2010GongLi}
Y.~Gong and Z.~Li, \enquote{{Immersed interface finite element methods for
  elasticity interface problems with non-homogeneous jump conditions}}, {\it
  Numer. Math. Theory Methods Appl.} \textbf{3} (2010) 23--39.

\bibitem{1971Gordon}
W.~J. Gordon, \enquote{Blending-function methods of bivariate and multivariate
  interpolation and approximation}, {\it SIAM J. Numer. Anal.} \textbf{8}
  (1971) 158--177.

\bibitem{2022Guomaximum}
R.~Guo, \enquote{On the maximum angle conditions for polyhedra with virtual
  element methods}, {\it arXiv preprint arXiv:2212.07241} .

\bibitem{2016GuoLin}
R.~Guo and T.~Lin, \enquote{A group of immersed finite element spaces for
  elliptic interface problems}, {\it IMA J.Numer. Anal.} \textbf{39} (2017)
  482--511.

\bibitem{2020GuoLin}
R.~Guo and T.~Lin, \enquote{{An immersed finite element method for elliptic
  interface problems in three dimensions}}, {\it J. Comput. Phys.} \textbf{414}
  (2020) 109478.

\bibitem{2020GuoLinZou}
R.~Guo, Y.~Lin and J.~Zou, \enquote{{Solving two dimensional
  $H(\mathbf{curl})$-elliptic interface systems with optimal convergence on
  unfitted meshes}}, {\it European J. Appl. Math. (in press)} .

\bibitem{2020GuoZhang}
R.~Guo and X.~Zhang, \enquote{Solving three-dimensional interface problems with
  immersed finite elements: A-priori error analysis}, {\it J. Comput. Phys.}
  \textbf{441} (2020) 110445.

\bibitem{2010HiptmairLiZou}
R.~Hiptmair, J.~LI and J.~Zou, \enquote{Convergence analysis of finite element
  methods for ${H}(\text{div};{\Omega})$-elliptic interface problems}, {\it J.
  Numer. Math.} \textbf{18} (2010) 187--218.

\bibitem{2007HiptmairXu}
R.~Hiptmair and J.~Xu, \enquote{Nodal auxiliary space preconditioning in
  {H}(curl) and {H}(div) spaces}, {\it SIAM J. Numer. Anal.} \textbf{45} (2007)
  2483--2509.

\bibitem{2005HolderDavid}
D.~Holder, {\it {Electrical impedance tomography: methods, history, and
  applications}} (Institute of Physics Pub, 2005).

\bibitem{2005HoustonPerugiaSchneebeli}
P.~Houston, I.~Perugia, A.~Schneebeli and D.~Sch{\"o}tzau, \enquote{{Interior
  penalty method for the indefinite time-harmonic Maxwell equations}}, {\it
  Numer. Math.} \textbf{100} (2005) 485--518.

\bibitem{2004HoustonPerugiaSchotzau}
P.~Houston, I.~Perugia and D.~Schotzau, \enquote{Mixed discontinuous galerkin
  approximation of the maxwell operator}, {\it SIAM J. Numer. Anal.}
  \textbf{42} (2004) 434--459.

\bibitem{2008HuShuZou}
Q.~Hu, S.~Shu and J.~Zou, \enquote{A mortar edge element method with nearly
  optimal convergence for three-dimensional {M}axwell's equations}, {\it Math.
  Comp.} \textbf{77} (2008) 1333--1353.

\bibitem{2002HuangZou}
J.~Huang and J.~Zou, \enquote{Some new a priori estimates for second-order
  elliptic and parabolic interface problems}, {\it J. Differential Equations}
  \textbf{184} (2002) 570--586.

\bibitem{2007HuangZou}
J.~Huang and J.~Zou, \enquote{{Uniform a priori estimates for elliptic and
  static Maxwell interface problems.}}, {\it Disc. Cont. Dynam. Sys., Series B}
  \textbf{7} (2007) 145--170.

\bibitem{2010WuXiao}
P.~Huang, H.~Wu and Y.~Xiao, \enquote{{An unfitted interface penalty finite
  element method for elliptic interface problems}}, {\it Comput. Methods Appl.
  Mech. Engrg.} \textbf{323} (2017) 439--460.

\bibitem{2022JiImmersed}
H.~Ji, \enquote{An immersed {R}aviart--{T}homas mixed finite element method for
  elliptic interface problems on unfitted meshes}, {\it J. Sci. Comput.}
  \textbf{91} (2022) 1--33.

\bibitem{2005KafafyLinLinWang}
R.~Kafafy, T.~Lin, Y.~Lin and J.~Wang, \enquote{{Three-dimensional immersed
  finite element methods for electric field simulation in composite
  materials}}, {\it Internat. J. Numer. Methods Engrg.} \textbf{64} (2005)
  940--972.

\bibitem{2010LiMelenkWohlmuthZou}
J.~Li, J.~M. Melenk, B.~Wohlmuth and J.~Zou, \enquote{{Optimal a priori
  estimates for higher order finite elements for elliptic interface problems}},
  {\it Appl. Numer. Math.} \textbf{60} (2010) 19--37.

\bibitem{2015LinLinZhang}
T.~Lin, Y.~Lin and X.~Zhang, \enquote{{Partially penalized immersed finite
  element methods for elliptic interface problems}}, {\it SIAM J. Numer. Anal.}
  \textbf{53} (2015) 1121--1144.

\bibitem{2020LiuZhangZhangZheng}
H.~Liu, L.~Zhang, X.~Zhang and W.~Zheng, \enquote{{Interface-penalty finite
  element methods for interface problems in $H^1$, H(curl), and H(div)}}, {\it
  Comput. Methods Appl. Mech. Engrg.} \textbf{367} (2020) 113137.

\bibitem{2014LoFinite}
D.~Lo, {\it Finite Element Mesh Generation} (CRC Press, 2014).

\bibitem{1985Donatella}
L.~D. Marini, \enquote{An inexpensive method for the evaluation of the solution
  of the lowest order {R}aviart--{T}homas mixed method}, {\it SIAM J. Numer.
  Anal.} \textbf{22} (1985) 493--496.

\bibitem{2003Monk}
P.~Monk, {\it {Finite Element Methods for Maxwell's Equations}} (Oxford
  University Press, 2003).

\bibitem{Nedelec1980}
J.-C. N{\'e}d{\'e}lec, \enquote{{Mixed finite elements in $\bf R^3$}}, {\it
  Numer. Math.} \textbf{35} (1980) 315--341.

\bibitem{1971Nitsche}
J.~Nitsche, \enquote{{\"Uber ein Variationsprinzip zur L\"osung von
  Dirichlet-Problemen bei Verwendung von Teilr\"aumen, die keinen
  Randbedingungen unterworfen sind}}, {\it Abhandlungen aus dem Mathematischen
  Seminar der Universit{\"a}t Hamburg} \textbf{36} (1971) 9--15.

\bibitem{OlshanskiiReuskenEtAl2009finite}
M.~A. Olshanskii, A.~Reusken and J.~Grande, \enquote{A finite element method
  for elliptic equations on surfaces}, {\it SIAM J. Numer. Anal.} \textbf{47}
  (2009) 3339--3358.

\bibitem{2000Pflaum}
C.~Pflaum, \enquote{{Subdivision of boundary cells in 3d}}, .

\bibitem{Raviart.P;Thomas.J1977}
P.-A. Raviart and J.~M. Thomas, \enquote{{A mixed finite element method for 2nd
  order elliptic problems}}, in {\it Mathematical aspects of finite element
  methods ({P}roc. {C}onf., {C}onsiglio {N}az. delle {R}icerche ({C}.{N}.{R}.),
  {R}ome, 1975)} (Springer, Berlin, 1977), pp. 292--315. Lecture Notes in
  Math., Vol. 606.

\bibitem{1982Saranen}
J.~Saranen, \enquote{On generalized harmonic fields in domains with anisotropic
  nonhomogeneous media}, {\it J. Math. Anal. Appl} \textbf{88} (1982) 104--115.

\bibitem{1983Saranen}
J.~Saranen, \enquote{On electric and magnetic problems for vector fields in
  anisotropic nonhomogeneous media}, {\it J. Math. Anal. Appl.} \textbf{91}
  (1983) 254--275.

\bibitem{2005TaiWinther}
X.-C. Tai and R.~Winther, \enquote{A discrete de rham complex with enhanced
  smoothness}, {\it Calcolo} \textbf{43} (2006) 287 -- 306.

\bibitem{2010VallaghePapadopoulo}
S.~Vallagh{\'e} and T.~Papadopoulo, \enquote{{A trilinear immersed finite
  element method for solving the electroencephalography forward problem}}, {\it
  SIAM J. Sci. Comput.} \textbf{32} (2010) 2379--2394.

\bibitem{Xu1996auxiliary}
J.~Xu, \enquote{The auxiliary space method and optimal multigrid
  preconditioning techniques for unstructured grids}, {\it Computing}
  \textbf{56} (1996) 215--235.

\bibitem{2016XuZhang}
J.~Xu and S.~Zhang, \enquote{Optimal finite element methods for interface
  problems}, in {\it Domain Decomposition Methods in Science and Engineering
  XXII}, eds. T.~Dickopf, M.~J. Gander, L.~Halpern, R.~Krause and L.~F.
  Pavarino (Springer International Publishing, Cham, 2016), pp. 77--91.

\bibitem{2011XuZhu}
J.~Xu and Y.~Zhu, \enquote{{Robust Preconditioner for H(curl) Interface
  Problems}}, in {\it Domain Decomposition Methods in Science and Engineering
  XIX} (Springer, Berlin, Heidelberg, 2011), pp. 173--180.

\bibitem{2015YingXie}
J.~Ying and D.~Xie, \enquote{{A new finite element and finite difference hybrid
  method for computing electrostatics of ionic solvated biomolecule}}, {\it J.
  Comput. Physics} \textbf{298} (2015) 636--651.

\bibitem{2004ZhaoWei}
S.~Zhao and G.~W. Wei, \enquote{{High-order FDTD methods via derivative
  matching for Maxwell's equations with material interfaces}}, {\it J. Comput.
  Phys.} \textbf{200} (2004) 60--103.

\end{thebibliography}
\end{document}